\documentclass[a4paper]{article}

\usepackage{amsfonts, latexsym, amssymb, amsthm, amsmath, mathrsfs, esint, mathtools, mathscinet}
\usepackage{xcolor}

\newcommand{\R}{\mathbb{R}}
\newcommand{\N}{\mathbb{N}}
\newcommand{\Z}{\mathbb{Z}}

\newcommand{\Ha}{\mathcal{H}}
\newcommand{\Le}{\mathcal{L}}
\newcommand{\M}{\mathbf{M}}

\newcommand{\sym}{\mathrm{sym}}
\newcommand{\loc}{\mathrm{loc}}
\newcommand{\BV}{\mathrm{BV}}

\newcommand{\grad}{\nabla}
\newcommand{\ncrt}{\mathcal{C}}
\newcommand{\transpose}{\mathsf{T}}

\let\div\relax\DeclareMathOperator{\div}{div}

\DeclareMathOperator{\dist}{dist}

\DeclareMathOperator{\supp}{supp}
\DeclareMathOperator{\curl}{curl}
\DeclareMathOperator{\tr}{tr}
\DeclareMathOperator*{\esssup}{ess \, sup}
\newcommand{\blank}{{\mkern 2mu\cdot\mkern 2mu}}

\newcommand{\scp}[2]{\left\langle #1, #2 \right\rangle}
\newcommand{\dd}[2]{\frac{\partial #1}{\partial #2}}
\newcommand{\set}[2]{\left\{ #1 \colon #2 \right\}}

\newcommand{\restr}{\mathchoice
{\kern2pt\mbox{\vrule width 0.08ex height1.5ex depth0ex\kern-0.08ex\vrule width 1.5ex height.08ex depth0ex}\kern2pt}
{\kern2pt\mbox{\vrule width 0.08ex height1.5ex depth0ex\kern-0.08ex\vrule width 1.5ex height.08ex depth0ex}\kern2pt}
{\kern1.5pt\mbox{\vrule width 0.06ex height1.1ex depth0ex\kern-0.06ex\vrule width 1.1ex height.06ex depth0ex}\kern1.5pt}
{\kern1pt\mbox{\vrule width 0.04ex height0.75ex depth0ex\kern-0.04ex\vrule width 0.75ex height.04ex depth0ex}\kern1pt}
}

\newtheorem{theorem}{Theorem}
\newtheorem{lemma}[theorem]{Lemma}
\newtheorem{proposition}[theorem]{Proposition}
\newtheorem{corollary}[theorem]{Corollary}
\newtheorem{question}[theorem]{Question}
\theoremstyle{definition}
\newtheorem{definition}[theorem]{Definition}
\newtheorem{example}[theorem]{Example}
\newtheorem*{acknowledgement}{Acknowledgement}
\theoremstyle{remark}
\newtheorem*{notation}{Notation}{}

\begin{document}

\title{Asymptotic minimality of one-dimensional transition profiles in Aviles-Giga type models: an approach via $1$-currents}

\author{Radu Ignat\footnote{Institut de Math\'ematiques de Toulouse, UMR 5219, Universit\'e de Toulouse,
CNRS, UPS IMT, F-31062 Toulouse Cedex 9, France. E-mail: Radu.Ignat@math.univ-toulouse.fr} \ 
and Roger Moser\footnote{University of Bath, Bath, BA2 7AY, UK. E-mail: r.moser@bath.ac.uk}}

\maketitle

\begin{abstract}
We study the asymptotic behaviour
of Modica-Mortola (or Allen-Cahn) type functionals for two-dimensional
vector fields under the assumption that their divergences converge to $0$ at a certain rate,
which effectively produces a model of Aviles-Giga type. This problem will typically give
rise to transition layers, which degenerate into discontinuities in the limit.
We analyse the energy concentration at these discontinuities and the structure
of the asymptotically minimal transition profiles.

We derive an estimate for the energy concentration in terms of a novel
geometric variational problem involving the notion of $\R^2$-valued $1$-currents
from geometric measure theory. This in turn leads to criteria, under which
the energetically favourable transition profiles are essentially one-dimensional.
\end{abstract}

\section{Introduction}

\subsection{The problem}

Let $\Omega \subseteq \R^2$ be an open domain. Suppose that
$W \colon \R^2 \to [0, \infty)$ is a locally H\"older continuous function. For
two-dimensional vector fields $u \colon \Omega \to \R^2$ and for $\epsilon > 0$,
consider a Modica-Mortola (or Allen-Cahn) type functional of the form
\[
E_\epsilon(u; \Omega) = \frac{1}{2} \int_\Omega \left(\epsilon |D u|^2 + \frac{1}{\epsilon} W(u)\right) \, dx.
\]
We are interested in the asymptotic behaviour of a family of vector fields
$u_\epsilon$ such that
\begin{equation} \label{eq:energy-and-divergence}
\limsup_{\epsilon \searrow 0} \left(E_\epsilon(u_\epsilon; \Omega) + \epsilon^{-2\tau} \|\div u_\epsilon\|_{L^s(\Omega)}^2\right) < \infty
\end{equation}
for some $\tau > 0$ and some $s > 2$. This is relevant in the context of
models that combine a Modica-Mortola type energy functional with a divergence
penalisation of the above form, or even with the constraint $\div u = 0$.

As in the classical theory of Modica and Mortola \cite{Modica-Mortola:77.1, Modica-Mortola:77.2} (which has been extended by Modica \cite{Modica:87.1} and
Sternberg \cite{Sternberg:88}),
the above conditions typically imply
convergence of a subsequence to a limit $u_0 \colon \Omega \to \R^2$ that
takes values in $W^{-1}(\{0\})$ almost everywhere. (Also see
the related work on Aviles-Giga type functionals
\cite{Ambrosio-DeLellis-Mantegazza:99, DeSimone-Kohn-Mueller-Otto:01, Jabin-Perthame:01, Jabin-Otto-Perthame:02}.)
In addition, the limit will satisfy
$\div u_0 = 0$. Transitions between different zeroes of $W$ are possible, but will require
a certain amount of energy.

Depending on the exact structure of the potential function $W$, the
vector field $u_0$ may belong to
$\BV(\Omega; \R^2)$, or may belong to a larger space, but even then we typically have a countably $1$-rectifiable jump
set $J \subseteq \Omega$, where the values of $u_0$ jump from one value to another \cite{DeLellis-Otto:03}.
More precisely, this set is characterised by the behaviour of a blow-up
around a point $x_0 \in J$: choose a sequence $r_k \searrow 0$
such that the functions $x \mapsto u_0(x_0 + r_k x)$ converge, say in $L_\loc^1(\R^2; \R^2)$,
to the limit $v \colon \R^2 \to \R^2$. Let $\nu \in S^1$ be one of the approximate normal
vectors to $J$ at $x_0$ (which exist almost everywhere with respect to the
$1$-dimensional Hausdorff measure by the countable rectifiability).
Then there exist $a^-, a^+ \in W^{-1}(\{0\})$ such that $v(x) = a^+$ when
$\nu \cdot x > 0$ and $v(x) = a^-$ when $\nu \cdot x < 0$.

A blow-up is useful, too, when we want to understand how much energy
will be concentrated at a point $x_0 \in J$ in the limit as
$\epsilon \searrow 0$; or in other words, how much energy is required to
generate a transition between $a^-$ and $a^+$. Suppose that we rescale
the vector fields $u_\epsilon$ similarly. Then we can expect that the limit
$u_0 \colon \R^2 \to \R^2$ is already of the form
\begin{equation} \label{eq:limiting-vector-field}
u_0(x) = \begin{cases}
a^+ & \text{if $\nu \cdot x > 0$}, \\
a^- & \text{if $\nu \cdot x < 0$}.
\end{cases}
\end{equation}
Since $u_0$ must be divergence free under the conditions we are interested in,
we expect that $\nu \perp (a^+ - a^-)$.

The density of the energy concentrated at a corresponding jump point is measured by the
quantity
\[
\frac{1}{2} \liminf_{\epsilon \searrow 0} E_\epsilon(u_\epsilon; B_1(0)),
\]
where $B_r(x)$ denotes the open disc of radius $r > 0$ centred at $x \in \R^2$,
and where we multiply with the factor $\frac{1}{2}$ to compensate for the fact that $B_1(0)$ contains a jump set of length $2$.
For $a^- \neq a^+$, we therefore consider the set
$\mathcal{U}(a^-, a^+)$, comprising all families $(u_\epsilon)_{\epsilon > 0}$ of vector fields $u_\epsilon \in W^{1, 2}(B_1(0); \R^2)$ such that
$u_\epsilon \to u_0$ in $L^1(B_1(0); R^2)$ and such that there exist
$\tau > 0$ and $s > 2$ with the property that
\begin{equation} \label{eq:small-divergence}
\lim_{\epsilon \searrow 0} \epsilon^{-\tau} \|\div u_\epsilon\|_{L^s(B_1(0))} = 0,
\end{equation}
where $u_0$ is defined as in \eqref{eq:limiting-vector-field} with $\nu = (a^- - a^+)^\perp/|a^- - a^+|$.
(We disregard the possibility that $\nu$ may point in the opposite direction, because that situation may
be reduced to this one by applying reflections in the domain and codomain and adjusting $W$ accordingly.)
Then we define
\[
\mathcal{E}(a^-, a^+) = \frac{1}{2} \inf\set{\liminf_{\epsilon \searrow 0} E_\epsilon(u_\epsilon; B_1(0))}{(u_\epsilon)_{\epsilon > 0} \in  \mathcal{U}(a^-, a^+)}.
\]

One important question is whether the same infimum is obtained when we
consider only one-dimensional, divergence-free transition profiles, i.e.,
vector fields of the form $u_\epsilon(x) = a^- + w_\epsilon(x \cdot \nu) (a^+ - a^-)$
for some functions $w_\epsilon \colon \R \to \R$. Under the typical
assumptions on $W$, the resulting number is easy to compute with the methods
from the Modica-Mortola theory and is
\[
\int_{[a^-, a^+]} \sqrt{W} \, d\Ha^1,
\]
where $[a^-, a^+]$ denotes the line segment connecting $a^-$ with $a^+$ and
$\Ha^1$ stands for the $1$-dimensional Hausdorff measure.

This question is the main focus of this paper, and it can be formulated as follows.

\begin{question} \label{qst:1D}
Under what conditions is
\[
\mathcal{E}(a^-, a^+) = \int_{[a^-, a^+]} \sqrt{W} \, d\Ha^1?
\]
\end{question}

\subsection{Main results}

We now fix the points $a^-$ and $a^+$. It is convenient to assume that
$\nu = (\begin{smallmatrix} 1 \\ 0 \end{smallmatrix})$, and thus $a_1^- = a_1^+$,
as this will simplify the presentation of our results. The general situation
can always be reduced to this case by a change of coordinates, so there is no
loss of generality.

We also assume that $W$ has a specific polynomial rate of growth as $|y| \to \infty$.
More precisely, we assume that there exist certain constants $c_1, c_2 > 0$
and $\bar{p} > 0$ such that
\begin{equation} \label{eq:growth-W}
c_1|y|^{2\bar{p}} - 1 \le W(y) \le c_2(|y|^{2\bar{p}} + 1)
\end{equation}
for all $y \in \R^2$.

To formulate our first result, we need to introduce some tools, including the notion
of $\R^2$-valued $1$-currents. This is a variant of a standard concept from geometric
measure theory. Its definition is normally given in terms of differential
forms in $\R^2$, but for our purpose, the following, equivalent definition is just as convenient.
Here we use the notation $|M|$ for the Frobenius norm of a matrix $M \in \R^{2 \times 2}$. We will later also use the notation $M :N$ for the Frobenius
inner product.

\begin{definition} \label{def:current-intro}
An \emph{$\R^2$-valued $1$-current on $\R^2$} is an element of the dual space of
$C_0^\infty(\R^2; \R^{2 \times 2})$. If $T$ is an $\R^2$-valued $1$-current in $\R^2$,
then its \emph{boundary} $\partial T$ is the
$\R^2$-valued distribution such that $\partial T(\xi) = T(D\xi)$
for every $\xi \in C_0^\infty(\R^2; \R^2)$. We say that $T$ is \emph{normal} if there
exists $C \ge 0$ such that
\[
T(\zeta) + \partial T(\xi) \le C\sup_{x \in \R^2} \bigl(|\zeta(x)| + |\xi(x)|\bigr)
\]
for all $\zeta \in C_0^\infty(\R^2; \R^{2 \times 2})$ and all $\xi \in C_0^\infty(\R^2; \R^2)$.
\end{definition}

We are particularly interested in normal $\R^2$-valued $1$-currents $T$ with a specific boundary,
given by the condition that
\[
\partial T(\xi) = \xi_1(a^+) - \xi_1(a^-)
\]
for all $\xi = (\begin{smallmatrix} \xi_1 \\ \xi_2 \end{smallmatrix})\in C_0^\infty(\R^2; \R^2)$. We write $\ncrt_{2 \times 2}^0$ for the set of
all normal currents with this boundary.

Given any normal $\R^2$-valued $1$-current $T$, there always exist a Radon measure $\|T\|$ on $\R^2$ and a
$\|T\|$-measurable, matrix-valued function $\vec{T} \colon \R^2 \to \R^{2 \times 2}$ with
$|\vec{T}| = 1$ almost everywhere, such that
\[
T(\zeta) = \int_{\R^2} \zeta : \vec{T} \, d\|T\|
\]
for any $\zeta \in C_0^\infty(\R^2; \R^{2 \times 2})$.

The following is an example of a current with some relevance for our results. Define
$T^0 \in \ncrt_{2 \times 2}^0$ by the condition that
\begin{equation} \label{eq:line-segment}
T^0(\zeta) = \int_{[a^-, a^+]} \zeta : \begin{pmatrix} 0 & 1 \\ 0 & 0 \end{pmatrix} \, d\Ha^1
\end{equation}
for $\zeta \in C_0^\infty(\R^2; \R^{2 \times 2})$. Then $\|T^0\| = \Ha^1 \restr [a^-, a^+]$
and $\vec{T}^0 = (\begin{smallmatrix} 0 & 1 \\ 0 & 0 \end{smallmatrix})$ almost everywhere.
(If we consider the rows of $\vec{T}$ separately, we obtain two conventional,
$\R$-valued $1$-currents, which we can interpret as the components of $T$.
The first component of $T^0$ is a representation of the oriented line segment
between $a^-$ and $a^+$, whereas the second component vanishes.)

We now consider the function
$F^* \colon \R^2 \times \R^{2 \times 2} \to \R \cup \{\infty\}$ such that
\[
F^*(y, N) = \begin{cases}
\frac{1}{4} W(y) \max\{|N|^2 - 2\det N, (n_{12} - n_{21})^2\} & \text{if $\tr N = 0$}, \\
\infty & \text{else},
\end{cases}
\]
where we write $N = (\begin{smallmatrix} n_{11} & n_{12} \\ n_{21} & n_{22} \end{smallmatrix})$.
For any $T \in \ncrt_{2 \times 2}^0$, we define
\[
\M_F(T) = \int_{\R^2} \sqrt{F^*(y, \vec{T}(y))} \, d\|T\|(y).
\]
(This is a variant of the mass that is normally associated to a current. The connections
will become more apparent in Section \ref{sct:L-infinity} below.)

We have the following results.

\begin{theorem} \label{thm:main}
The inequality
\[
\mathcal{E}(a^-, a^+) \ge 2\inf_{T \in \ncrt_{2 \times 2}^0} \M_F(T)
\]
holds true.
\end{theorem}

\begin{corollary} \label{cor:1D}
Let $T^0 \in \ncrt_{2 \times 2}^0$ be given by \eqref{eq:line-segment}.
If $T^0$ is a minimiser of $\M_F$ in $\ncrt_{2 \times 2}^0$
(i.e., if $\M_F(T^0) \le \M_F(T)$ for every $T \in \ncrt_{2 \times 2}^0$), then
\[
\mathcal{E}(a^-, a^+) = \int_{[a^-, a^+]} \sqrt{W} \, d\Ha^1.
\]
\end{corollary}

Thus we may be able to give an answer to Question \ref{qst:1D} by
solving a different variational problem involving currents. Since currents
can be interpreted geometrically, that variational problem is geometric in nature.
It is also rather unusual because of the structure of the above function $F^*$.
It may be difficult to solve in general, but we can give some estimates that
allow further conclusions.

For $j = 0, 1, 2, \dotsc$, let $C_{\bar{p}}^j(\R^2)$ denote the space of all $\phi \in C^j(\R^2)$ such
that there exists a constant $C \ge 0$ satisfying $|D^k\phi(y)| \le C(|y|^{\bar{p} - k} + 1)$ for
all $y \in \R^2$ and $k = 0, \dotsc, j$. We now have the following result.

\begin{corollary} \label{cor:PDE}
Suppose that $W = w^2$, and suppose that there exist Borel functions
$\iota, \kappa, \lambda \colon \R^2 \to [-1, 1]$ with
\[
\iota^2 \le \min\{1 - \lambda^2, (1 + \kappa)(1 - \lambda), (1 - \kappa)(1 + \lambda)\},
\]
such that $\iota w, \kappa w, \lambda w \in C^2(\R^2) \cap C_{\bar{p}}^1(\R^2)$ and
\begin{equation} \label{eq:PDE-criterion}
\frac{\partial^2}{\partial y_1^2}(\kappa w) = \frac{\partial^2}{\partial y_2^2}(\lambda w) + 2\frac{\partial^2}{\partial y_1 \partial y_2}(\iota w).
\end{equation}
Then
\[
\mathcal{E}(a^-, a^+) \ge \int_{a_2^-}^{a_2^+} (\kappa w)(a_1^-, t) \, dt
\]
for any $T \in \ncrt_{2 \times 2}^0(\R^2)$.
\end{corollary}

If in addition, we know that $\kappa(a_1^-, y_2) = 1$ for all $y \in [a_2^-, a_2^+]$, then
it follows, of course, that
\[
\mathcal{E}(a^-, a^+) = \int_{[a^-, a^+]} \sqrt{W} \, d\Ha^1.
\]
Thus in this case, we again obtain an answer to Question \ref{qst:1D}.

It may not be obvious in general whether a given potential function
permits functions $\iota, \kappa, \lambda$ as in Corollary \ref{cor:PDE},
but we give some examples in Section \ref{sct:examples} below, where the
result applies. This includes the potential $W(y) = (1 - |y|^2)^2$
from the classical Aviles-Giga functional \cite{Aviles-Giga:87}, but also
the generalisation $W(y) = |y|^{4n} (1 - |y|^{2m})^2$ for arbitrary $n, m \in \N$,
at least in the case of transitions between the points $a^- = (0, -1)$
and $a^+ = (0, 1)$. We have a similar result for $W(y) = (1 - |y|^2)^{2\beta}$ for $\beta \in (0, 1)$, but only when we restrict the class of admissible
vector fields by the condition
$|u| \le 1$ (see Corollary \ref{cor:examples} for both examples).

Corollary \ref{cor:PDE} is a consequence of another, more general estimate, which may be more
useful in certain situations. Since the statement is also more technical, however, we postpone
the formulation to Section \ref{sct:geometric} (see Theorem~\ref{thm:decomposition}).

\subsection{Background}

Problems like the above are relevant for a number of physical systems, including micromagnetics
\cite{Hubert-Schaefer:98}, smectic-A liquid crystals \cite{Kleman-Parodi:75, Aviles-Giga:87}, thin film blisters
\cite{Ortiz-Gioia:94}, or crystal surfaces \cite{Stewart-Goldenfeld:92}. Such models typically
arise when a Ginzburg-Landau type energy functional is combined with a divergence penalisation,
or is applied to a gradient vector field. Indeed, if we consider a quantity such as
\begin{equation} \label{eq:Aviles-Giga}
\frac{1}{2} \int_\Omega \left(\epsilon |D^2 \phi|^2 + \frac{1}{\epsilon} W(D\phi)\right) \, dx,
\end{equation}
then the identification $u = \nabla^\perp \phi$ will give rise to $E_\epsilon(u; \Omega)$, and
in this case, we even have the condition $\div u = 0$. The integral in \eqref{eq:Aviles-Giga} gives a variant
of the Aviles-Giga functional \cite{Aviles-Giga:87}.

Despite its importance, remarkably little is known about Question \ref{qst:1D}, let alone about
how to determine $\mathcal{E}(a^-, a^+)$ in general, with the exception of some special cases.
It can happen, of course, that the constructions from the vector-valued Modica-Mortola problem
\cite{Sternberg:88, Baldo:90} happen to be divergence free, in which case they also provide a solution to
the above problem. Otherwise, only the case of the classical Aviles-Giga functional, which corresponds to
$W(y) = (1 - |y|^2)^2$, has a reasonably comprehensive theory. One of the key contributions is
of Jin and Kohn \cite{Jin-Kohn:00}, who (among other things) determined the value of $\mathcal{E}(a^-, a^+)$
in this situation. Without attempting to give a complete list, we mention some
other noteworthy contributions to this theory
\cite{Ambrosio-DeLellis-Mantegazza:99, DeSimone-Kohn-Mueller-Otto:01, Jabin-Perthame:01, Jabin-Otto-Perthame:02}.

More general potential functions have been studied by Ignat and Monteil \cite{Ignat-Monteil:20}.
In particular, they give some results similar to Corollary \ref{cor:PDE} (although weaker), which they
prove with methods different from what we use here.

Theorem \ref{thm:main} and its corollaries (including Theorem \ref{thm:decomposition} in Section \ref{sct:geometric}
below) add a completely new tool to the study of these problems. The theorem provides an estimate for $\mathcal{E}(a^-, a^+)$
in terms of another variational problem, which is geometric in nature, and whose connection to
the functionals $E_\epsilon$ is far from obvious. That variational problem is difficult to solve
in general, but this novel connection is clearly of theoretical value, and we show in
Section \ref{sct:examples} that it can be used to answer Question \ref{qst:1D} for some
examples where the problem was previously open.

There is then the obvious question of how to determine $\mathcal{E}(a^-, a^+)$ when the equality
from Question \ref{qst:1D} is \emph{not} satisfied. Almost nothing is known for this question in general,
although for some specific problems of a similar nature, it can be answered
\cite{Riviere-Serfaty:01, Riviere-Serfaty:03, Alouges-Riviere-Serfaty:02, Ignat-Moser:12}.
We provide no general results about this question here, but we give some examples in Section \ref{sct:examples}
which suggest that Theorem~\ref{thm:main} may be useful in this context, too.
We consider two specific constructions of transition profiles, one proposed by Jin and Kohn \cite{Jin-Kohn:00} and one arising in micromagnetics and known as a cross-tie wall \cite{DeSimone-Kohn-Mueller-Otto:03}. We compare their asymptotic energies
with the $F$-masses of specific $\R^2$-valued $1$-currents, the structures of
which mirror some features of the constructions, and find that they satisfy the
relationship suggested by Theorem \ref{thm:main}. It is not clear whether
these $1$-currents are minimisers of $\M_F$ for any interesting choice of $W$
(apart from $W(y) = (1 - |y|^2)^2$, which is already well understood),
but if they are, then it follows that the transition profiles are asymptotically
energy minimising.

There are some aspects of the theory that we implicitly take for granted in the formulation
of Question~\ref{qst:1D}. If we were to fully analyse the problem with respect to $\Gamma$-convergence, we would have to prove that
\begin{itemize}
\item the limiting energy is really concentrated on a countably $1$-rectifiable jump set,
where we can perform an appropriate blow-up, and
\item after the blow-up, we have convergence of a subsequence in $L^1(B_1(0); \R^n)$ to a limit $u_0$ as above.
\end{itemize}
That is, we would need some information about the structure of limit points and compactness of families
$(u_\epsilon)_{\epsilon > 0}$ satisfying \eqref{eq:energy-and-divergence}. Such information is relatively
easy to obtain when $W$ has only isolated zeroes, and results of this type are available for the potential function
$W(y) = (1 - |y|^2)^2$ (the Aviles-Giga functional) \cite{DeLellis-Otto:03, DeSimone-Kohn-Mueller-Otto:01}
and some generalisations thereof \cite{Bochard-Pegon:17, Lamy-Lorent-Peng:22}. Obviously, if we have such
results for a potential function $\tilde{W}$ such that $\tilde{W} \le CW$ for some constant $C > 0$, then the same follows for
$W$. Nevertheless, these questions are open in general and are not studied here.

\subsection{Strategy for the proofs and organisation of the paper}

Theorem \ref{thm:main} may appear mysterious at first, as the connection between the
energy $E_\epsilon$ and the $F$-mass $\M_F$ becomes apparent only when the
ingredients for the proof are known. For this reason, we give an informal overview
of the arguments here. At the same time, we explain how the paper is organised.

The first key idea in the proof is that of a `calibration' (also called `entropy' by some
authors, because of some analogy with entropies for conservation laws). This idea
goes back to the paper of Jin and Kohn \cite{Jin-Kohn:00}, but has been refined by
DeSimone, Kohn, M\"uller, and Otto \cite{DeSimone-Kohn-Mueller-Otto:01} and subsequently
studied by a number of authors
\cite{DeLellis-Otto:03, Ignat-Merlet:11, Ignat-Merlet:12, Ignat-Moser:12, Ignat-Monteil:20}.
The formulation that we use here is as follows. Let $\mathcal{L}(\R^2; \R^{2 \times 2})$
denote the space of linear maps $\R^2 \to \R^{2 \times 2}$. Suppose that there exist
$\Phi \in C^1(\R^2; \R^2)$, $\alpha \in C^0(\R^2)$, and $a \in C^1(\R^2; \mathcal{L}(\R^2; \R^{2 \times 2}))$
such that
\begin{equation} \label{eq:calibration0}
\div \Phi(u) + \alpha(u) \div u \le \frac{\epsilon}{2} |Du|^2 + \frac{1}{2\epsilon} W(u) + \epsilon \div(a(u) Du)
\end{equation}
for $\epsilon > 0$ and for all sufficiently regular vector fields $u \colon B_1(0) \to \R^2$.
Then it is not difficult to see, when we integrate over $B_1(0)$ and integrate by parts,
that we obtain an estimate of the form
\[
\mathcal{E}(a^-, a^+) \ge \Phi_1(a^+) - \Phi_1(a^-)
\]
under reasonable assumptions. Clearly, such an inequality is potentially useful for
answering Question \ref{qst:1D}.

But it is not clear at all how to find $\Phi$, $\alpha$, and $a$ in general, at least not
such that they give rise to a \emph{useful} estimate. (The choice $\Phi = 0$, $\alpha = 0$, and
$a = 0$ will always work, but the resulting estimate is trivial.) Good calibrations have
been constructed in special cases, most notably for the Aviles-Giga functional \cite{Jin-Kohn:00},
but no general construction is known.

In Section \ref{sct:differential-inequalities} we derive a condition that is equivalent, for a given
$\Phi$, to the existence of $\alpha$ and $a$ such that \eqref{eq:calibration0} holds true. The equivalence requires more regularity than we can typically
expect, but the condition is still useful more generally.
If we define the function
\[
f(M) = \frac{1}{2} \left(|M|^2 - \frac{1}{2} (\tr M)^2 + |m_{12} - m_{21}| \sqrt{|M|^2 - 2\det M}\right)
\]
for $M = (\begin{smallmatrix} m_{11} & m_{12} \\ m_{21} & m_{22} \end{smallmatrix}) \in \R^{2 \times 2}$,
then it takes the form of the inequality
\[
f(D\Phi) \le W
\]
(see Proposition \ref{prp:calibration}).
We use arguments inspired by the work of Ignat and Merlet \cite{Ignat-Merlet:11} in this step, but
we extend these ideas considerably.

This gives a convenient way to check whether a given function $\Phi$ gives rise to a calibration,
but still does not tell us how to construct one. But suppose that we want to find the best possible
calibration, which for our purposes means that $\Phi_1(a^+) - \Phi_1(a^-)$ should be as large as
possible. Then the above inequality suggests that we determine
\[
\eta_0 = \sup\set{\Phi_1(a^+) - \Phi_1(a^-)}{f(D\Phi) \le W}.
\]
If we can solve this variational problem, then we have the best estimate that can be
achieved with this approach.

It is convenient here to recast the problem in a different form. Define the function
\[
F(y, M) = \frac{f(M)}{W(y)}
\]
(assuming for the moment that $W(y) > 0$ for all $y \in \R^2$ and ignoring the fact that
Question \ref{qst:1D} is more interesting for a potential function with zeroes). Then
we may instead try to determine
\[
e_\infty = \inf\set{\|F(\blank, D\Phi)\|_{L^\infty(\R^2)}^{1/2}}{\Phi_1(a^+) - \Phi_1(a^-) = 1}.
\]
It is easy to see that $\eta_0 = 1/e_\infty$. We thus obtain a variational problem involving the
$L^\infty$-norm.

Very little is known about problems of this sort. For similar problems involving a \emph{scalar} function
(in place of the vector-valued $\Phi$), there is a body of literature going back to the work of Aronsson
\cite{Aronsson:65, Aronsson:66, Aronsson:67, Aronsson:68} and including papers by many other authors.
Once more we give an incomplete list \cite{Bhattacharya-DiBenedetto-Manfredi:89, Jensen:93, Savin:05, Evans-Savin:08}.
For vector-valued functions, this theory does not apply. There is some work by
Katzourakis \cite{Katzourakis:12, Katzourakis:13, Katzourakis:17.3}, but these results
do not tell us much about the solutions to the above problem. Fortunately, we do not need to know
anything about the structure of the solutions, we merely need to determine the number $e_\infty$.
For this purpose, the ideas of a recent paper by Katzourakis and Moser \cite{Katzourakis-Moser:25}
are useful. This paper treats only the case of the function $F(y, M) = \frac{1}{2} |M|^2$, but
the methods can be generalised, and this is what we do in Section \ref{sct:L-infinity}.
We can think of the results as a characterisation of the essential behaviour of the minimisers
through a dual problem, in our case that of minimising $\M_F$ for $\R^2$-valued $1$-currents.
The analysis has to be carried out for a regularised version of $F$, but then we can prove that there
exists a minimiser $T$ of $\M_F$ in $\ncrt_{2 \times 2}^0$ such that $\Phi_1(a^+) - \Phi_1(a^-) = 2\M_F(T)$ (see Theorem \ref{thm:current}).
This is where the inequality from Theorem \ref{thm:main} ultimately comes from.

Remarkably, even though calibrations are central to our approach, this result means that we do not
need to construct any calibrations in the end. We only need to know $\Phi_1(a^+) - \Phi_1(a^-)$, and
this information is encoded in the $\R^2$-valued $1$-current $T$.

As already mentioned, these arguments require a regularisation of $F$, and we need to make sure
that we can recover the relevant information when we relax the conditions on $F$ again.
This is the purpose of Section \ref{sct:regularisation}. At this point, the proof of Theorem \ref{thm:main}
is complete. But to make use of it, we have to study the problem of minimising $\M_F$ in $\ncrt_{2 \times 2}^0$.

This is the problem that we study in Section \ref{sct:geometric}. Superficially, it may look
deceptively simple. After all, we may think of $1$-currents as generalised curves in $\R^2$, and
$\M_F$ resembles an anisotropic version of the length functional. That is, we have a variant of
the problem of finding geodesics. (Incidentally, geodesics for a degenerate Riemannian metric appear
in the solutions of the vector-valued Modica-Mortola problem as well \cite{Baldo:90}.)
There are, however, several complications. First, we have \emph{$\R^2$-valued} $1$-currents, so
we should really think of a \emph{pair} of curves linked through $\M_F$. Second, the function
$F$ is degenerate in some sense in both variables. Third, although $T$ should be thought of as a
one-dimensional object, it does not follow that it is supported on a one-dimensional set (and in
general it is not; see Example \ref{exm:cross-tie} below). Because of all of this, the standard
methods from geometric analysis do not apply here.

We do not have any general methods to solve the problem, but we can nevertheless give some
estimates, which show that $T^0$, as defined in \eqref{eq:line-segment}, is a minimiser under certain conditions. One of the key tools
we use for this purpose, is a result and Bonicatto and Gusev \cite{Bonicatto-Gusev:22}
(see also the work of Smirnov \cite{Smirnov:93} and of Baratchart, Hardin, and Villalobos-Guill\'en
\cite{Baratchart-Hardin-Villalobos-Guillen:21}), which gives a decomposition of a normal $1$-current
into actual curves. This result applies to conventional $1$-currents, not $\R^2$-valued ones,
but at least we can apply it to the first component of $T \in \ncrt_{2 \times 2}^0$.
We can then give some estimates relying on convexity and the structure of $F$ to also take the
second component into account. This gives rise to a functional for Lipschitz curves,
which is now really similar to an anisotropic version of the length functional and, in
principle, can be analysed with standard methods involving ordinary differential equations.
Unfortunately, it also involves some unknown functions, and therefore, the task is not so
simple after all. Notwithstanding, with some further estimates, we finally prove Corollary \ref{cor:PDE}
as a result.

We conclude the paper with some examples in Section \ref{sct:examples}. First, we discuss some potential
functions $W$ such that Corollary \ref{cor:PDE} applies, and the optimal transition layers therefore
have one-dimensional profiles. This includes the well-known Aviles-Giga functionals, but also includes some new examples. In particular, we have
partial results for
the potential $W(y) = (1 - |y|^2)^{2\beta}$ for $\beta \in (0, 1)$ (see Corollary \ref{cor:examples}), which
is still badly understood despite the attention it has received in the literature
\cite{Aviles-Giga:99, Ambrosio-DeLellis-Mantegazza:99}.

Finally, we consider the question what Theorem \ref{thm:main} can tell us
in situations where the equality from Question \ref{qst:1D} does \emph{not} hold true (i.e., when the optimal transition profiles are not one-dimensional
and form microstructures instead).
We have no general results here, but we can compare the number $\M_F(T)$ for some specific
currents with the energy density for certain known microstructures.
If $T$ minimises $\M_F$, then the former gives a bound for $\mathcal{E}(a^+, a^-)$ from below by Theorem \ref{thm:main},
while the latter gives a
bound from above by definition. If the two bounds match, then
this proves that the microstructures are asymptotically minimising.
We can achieve this for two different examples (assuming that the potential
function $W$ is such that the corresponding $T$ is indeed a minimiser of $\M_F(T)$): a diamond pattern introduced by Jin and Kohn \cite{Jin-Kohn:00}
and the cross-tie wall from micromagnetics \cite{DeSimone-Kohn-Mueller-Otto:03}.
This raises the question whether the estimate from Theorem \ref{thm:main} might be sharp
in general. We have no evidence for this, however, beyond these two examples.

\subsection{Notation}

The following notation is used throughout the paper, with the exception of Section \ref{sct:L-infinity},
where some adjustments are required due to a more general setting.

We write $M^\transpose$ for the transpose of a matrix $M$ and
$I$ for the identity $(2 \times 2)$-matrix.
As mentioned previously, for $M, N \in \R^{2 \times 2}$, we use the notation $M : N = \tr(M^\transpose N)$ and $|M| = \sqrt{M : M}$
for the Frobenius inner product and norm, respectively. (In Section \ref{sct:L-infinity}, we will
also use the corresponding notation for $(m \times n)$-matrices.)

Given two vector spaces $X$ and $Y$, the space of linear maps $X \to Y$ is denoted by $\mathcal{L}(X, Y)$.

Although our problem is concerned with vector fields $u \colon \Omega \to \R^2$, much of our
analysis will take place entirely in the codomain $\R^2$. We generally use the notation $x$ for
a generic point in the domain $\Omega$, and $y$ for a generic point in the codomain $\R^2$.
(Section \ref{sct:L-infinity} is an exception here, too, as it is about an auxiliary problem independent
of $u$.)

We will frequently work with convolutions with a standard mollifier. Therefore,
we fix $\rho \in C_0^\infty(B_1(0))$ with $\rho \ge 0$ and $\int_{B_1(0)} \rho(y) \, dy = 1$.
For $\delta > 0$, we set $\rho_\delta(y) = \delta^{-2} \rho(y/\delta)$.

We further use the notation $\N_0 = \{0, 1, 2, \dotsc\}$.

\section{Characterising calibrations through differential inequalities}
\label{sct:differential-inequalities}

In this section, we derive some conditions in the form of certain
differential inequalities related to the inequality
\begin{equation} \label{eq:calibration}
\div \Phi(u) + \alpha(u) \div u \le \frac{\epsilon}{2} |D u|^2 + \frac{1}{2\epsilon} W(u) + \epsilon \div(a(u) Du)
\end{equation}
that characterises calibrations.
These conditions will make it easier to study suitable
calibrations later on. Some of the following arguments go back to
the work of Ignat and Merlet \cite{Ignat-Merlet:12}, but we extend the theory significantly.

\subsection{Pointwise conditions}

For a given map $\Phi \colon \R^2 \to \R^2$ and a function $\alpha \colon \R^2 \to \R$,
we want to understand the above inequality \eqref{eq:calibration}.
First we show that it suffices to consider tensor fields $a$
of a specific form.

\begin{proposition}
Suppose that $\Phi \in C^1(\R^2; \R^2)$ and $\alpha \in C^0(\R^2)$.
Let $\epsilon > 0$. If $a \colon \R^2 \to \mathcal{L}(\R^{2 \times 2}; \R^2)$ is continuously differentiable and
satisfies \eqref{eq:calibration} for all $u \in C^2(B_1(0); \R^2)$,
then there exists a vector field $\omega \in C^1(\R^2; \R^2)$ such that
\[
a(y)M = -(M^\transpose \omega(y))^\perp
\]
for any $y \in \R^2$ and
\[
\div(a(u) Du) = (\curl \omega)(u) \det Du
\]
for any $u \in C^2(B_1(0); \R^2)$.
\end{proposition}

\begin{proof}
Let $a_{jk}^i \in C^1(\R^2)$, for $i, j, k = 1, 2$, denote the coefficients of $a$, so that
\[
a(y)M = \sum_{j, k = 1}^2 m_{jk} \begin{pmatrix} a_{jk}^1(y) \\ a_{jk}^2(y) \end{pmatrix}
\]
for all $M = (\begin{smallmatrix} m_{11} & m_{12} \\ m_{21} & m_{22} \end{smallmatrix}) \in \R^{2 \times 2}$
and all $y \in \R^2$.

Given an arbitrary point $y \in \R^2$ and two symmetric matrices
\[
\Lambda^1 = \begin{pmatrix} \lambda_{11}^1 & \lambda_{12}^1 \\ \lambda_{21}^1 & \lambda_{22}^1 \end{pmatrix} \quad \text{and} \quad \Lambda^2 = \begin{pmatrix} \lambda_{11}^2 & \lambda_{12}^2 \\ \lambda_{21}^2 & \lambda_{22}^2 \end{pmatrix},
\]
we can find $u = (\begin{smallmatrix} u_1 \\ u_2 \end{smallmatrix}) \in C^2(B_1(0); \R^2)$ such that $u(0) = y$ and
$Du(0) = 0$, while at the same time, $D^2u_k(0) = \Lambda^k$
for $k = 1, 2$. Then
\[
\div(a(u)Du)(0) = \sum_{i, j, k = 1}^2 a_{ij}^k(y) \lambda_{jk}^i.
\]
Inequality \eqref{eq:calibration}, evaluated at $0$, thus gives
\[
0 \le \frac{1}{2\epsilon} W(y) + \epsilon \sum_{i, j, k = 1}^2 a_{ij}^k(y) \lambda_{jk}^i.
\]
Since this also holds true for all real multiples of $\Lambda^1$ and
$\Lambda^2$, it follows in fact that
\[
\sum_{i, j, k = 1}^2 a_{ij}^k(y) \lambda_{jk}^i = 0
\]
for any pair of symmetric matrices. Therefore, the coefficients
$a_{11}^1$, $a_{12}^2$, $a_{21}^1$, and $a_{22}^2$ must vanish, and
\[
a_{12}^1 + a_{11}^2 = 0 \quad \text{and} \quad a_{22}^1 + a_{21}^2 = 0.
\]
Set
\[
\omega = \begin{pmatrix} a_{12}^1 \\ a_{22}^1 \end{pmatrix},
\]
then the desired formulas follow by a direct calculation.
\end{proof}

We have the following characterisation of \eqref{eq:calibration}.

\begin{proposition} \label{prp:pointwise-characterisation}
Suppose that $\Phi \in C^1(\R^2; \R^2)$ and $\alpha \in C^0(\R^2)$.
Set $\Xi = D\Phi + \alpha I$.
Let $\omega \in C^1(\R^2; \R^2)$ and $\sigma = \curl \omega$. Suppose that
$\epsilon > 0$. Then the following statements are equivalent.
\begin{itemize}
\item[(A)]
The inequality
\begin{equation} \label{eq:calibration-omega}
\div \Phi(u) + \alpha(u) \div u  \le \frac{\epsilon}{2} |Du|^2 + \frac{1}{2\epsilon} W(u) - \epsilon \div\bigl((Du)^\transpose \omega(u)\bigr)^\perp
\end{equation}
is satisfied for all $u \in C^2(B_1(0); \R^2)$.

\item[(B)]
The inequalities $|Du|^2 + 2 \sigma(u) \det Du \ge 0$ and
\[
\div \Phi(u) + \alpha(u) \div u \le \Bigl(W(u) \bigl(|Du|^2 + 2 \sigma(u) \det Du\bigr)\Bigr)^{1/2}
\]
are satisfied for all $u \in C^2(B_1(0); \R^2)$.

\item[(C)]
For all $y \in \R^2$ and all $M \in \R^{2 \times 2}$,
\[
(M : \Xi(y))^2 \le W(y) \bigl(|M|^2 + 2 \sigma(y) \det M\bigr)
\]
and $|\sigma(y)| \le 1$.
\end{itemize}
\end{proposition}

\begin{proof}
If we assume that (B) holds true, then (A) follows from the observation that
\[
-\div\bigl((Du)^\transpose \omega(u)\bigr)^\perp = \sigma(u) \det Du
\]
and Young's inequality.

Now suppose that (A) holds true. We want to show that (C) follows.
We note that
\[
\Xi(u) : (Du)^\transpose = \div \Phi(u) + \alpha(u) \div u \le  \frac{\epsilon}{2} |Du|^2 + \frac{1}{2\epsilon} W(u) + \epsilon \sigma(u) \det Du
\]
for any $u \in C^2(B_1(0); \R^2)$. Consider an arbitrary point $y \in \R^2$.
If $W(y) = 0$, then we choose an arbitrary matrix $M \in \R^{2 \times 2}$ and
consider $u \in C^2(B_1(0); \R^2)$ such that $u(0) = y$ and $Du(0) = M^\transpose$. Then
we conclude that
\begin{equation} \label{eq:quadratic-form0}
\Xi(y) : M \le \frac{\epsilon}{2} |M|^2 + \epsilon \sigma(y) \det M
\end{equation}
for any $M \in \R^{2 \times 2}$. Since the left-hand side is linear in $M$ and the
right-hand side is quadratic, this can only hold true when $\Xi(y) = 0$. In this case,
the first inequality in (C) is clear, and the second one follows from the fact that
the right-hand side of \eqref{eq:quadratic-form0} must be positive semi-definite in $M$.

Now suppose that $W(y) \neq 0$. Choose a matrix $M \in \R^{2 \times 2}$ such that
\[
|M|^2 = \frac{W(y)}{\epsilon^2}.
\]
We can again choose $u$ such that $u(0) = y$ and $Du(0) = M^\transpose$. Thus
\[
\Xi(y) : M \le  \frac{\epsilon}{2} |M|^2 + \frac{1}{2\epsilon} W(y) + \epsilon \sigma(y) \det M = \sqrt{W(y)} \left(|M| + \frac{\sigma(y) \det M}{|M|}\right).
\]
Since the left-hand side and the right-hand side are both positive homogeneous
in $M$ of degree $1$, it follows that in fact,
\[
\Xi(y) : M \le \sqrt{W(y)} \left(|M| + \frac{\sigma(y) \det M}{|M|}\right)
\]
for \emph{all} $M \in \R^{2 \times 2} \setminus \{0\}$. If $\sigma(y) = 0$, then the desired inequalities hold at $y$.

If $\sigma(y) \neq 0$, then we fix a number $c \in [0, 1)$ such that
$c |\sigma(y)| < 1$ and consider $M \in \R^{2 \times 2}$ such that
\[
|M|^2 + 2c \sigma(y) \det M = \frac{W(y)}{\epsilon^2}.
\]
In this case, we obtain the inequality
\[
\begin{split}
\Xi(y) : M & \le \frac{\epsilon}{2} |M|^2 + \frac{1}{2\epsilon} W(y) + \epsilon \sigma(y) \det M \\
& = \frac{\epsilon}{2} \bigl(|M|^2 + 2c \sigma(y) \det M\bigr) + \frac{1}{2\epsilon} W(y) + \epsilon (1 - c)\sigma(y) \det M \\
& = \sqrt{W(y)} \left(\sqrt{|M|^2 + 2c \sigma(y) \det M} + \frac{(1 - c) \sigma(y) \det M}{\sqrt{|M|^2 + 2c \sigma(y) \det M}}\right).
\end{split}
\]
Again we conclude that
\begin{equation} \label{eq:interpolation}
\Xi(y) : M \le \sqrt{W(y)} \left(\sqrt{|M|^2 + 2c \sigma(y) \det M} + \frac{(1 - c) \sigma(y) \det M}{\sqrt{|M|^2 + 2c \sigma(y) \det M}}\right)
\end{equation}
for all $M \in \R^{2 \times 2} \setminus \{0\}$. If we replace $M$ by $-M$, then the left-hand
side changes its sign while the right-hand side stays the same. Therefore,
the inequality
\[
\sqrt{|M|^2 + 2c \sigma(y) \det M} + \frac{(1 - c) \sigma(y) \det M}{\sqrt{|M|^2 + 2c \sigma(y) \det M}} \ge 0
\]
must be satisfied for all $M \in \R^{2 \times 2}$, which implies that
\[
|M|^2 + (1 + c) \sigma(y) \det M \ge 0.
\]
We conclude that
\[
(1 + c) |\sigma(y)| \le 2.
\]
Since we have proved this inequality for any $c$ such that
\[
0 \le c < \min\left\{1, \frac{1}{|\sigma(y)|}\right\},
\]
it follows that $|\sigma(y)| \le 1$.

It now follows that \eqref{eq:interpolation} is satisfied for any
$c \in [0, 1)$. Letting $c \nearrow 1$, we derive the other inequality in
(C) as well.

Now suppose that (C) is satisfied. Since
\[
\div \Phi(u) + \alpha(u) = (Du)^\transpose : \Xi(u),
\]
statement (B) follows immediately.
\end{proof}

Next we examine inequalities as in statement (C) above. For this purpose, we
require the function $g \colon \R^{2 \times 2} \to \R$ defined by
\[
g(M) = \frac{1}{2} \left(|M|^2 + \sqrt{|M|^4 - 4(\det M)^2}\right).
\]
We note that $g$ is convex, which is most easily seen in different coordinates:
let
\begin{align*}
q_1 & = \frac{1}{\sqrt{2}} (m_{11} + m_{22}), & q_2 & = \frac{1}{\sqrt{2}} (m_{11} - m_{22}), \\
q_3 & = \frac{1}{\sqrt{2}} (m_{12} + m_{21}), & q_4 & = \frac{1}{\sqrt{2}} (m_{12} - m_{21}).
\end{align*}
Then
\[
g(M) = \frac{1}{2} \left(\sqrt{q_1^2 + q_4^2} + \sqrt{q_2^2 + q_3^2}\right)^2,
\]
which is clearly convex. Since $g$ is also homogeneous of degree $2$,
it follows that for all $M, N \in \R^{2 \times 2}$ and for
$s, t \in (1, \infty)$ with $\frac{1}{s} + \frac{1}{t} = 1$,
\begin{equation} \label{eq:interpolation-g}
g(M + N) = g\left(\frac{sM}{s} +  \frac{tN}{t}\right) \le \frac{g(sM)}{s} + \frac{g(tN)}{t} = sg(M) + tg(N).
\end{equation}

\begin{lemma} \label{lem:minimising-over-sigma}
Suppose that $\Lambda \in \R^{2 \times 2} \setminus \{0\}$. Let
\[
s_0 = \frac{\det \Lambda}{g(\Lambda)}.
\]
\begin{enumerate}
\item
The inequality $g(\Lambda) \le 1$ holds true if, and only if,
there exists $s \in [-1, 1]$ such that
\begin{equation} \label{eq:quadratic-form}
(\Lambda : M)^2 \le |M|^2 + 2s \det M
\end{equation}
for all $M \in \R^{2 \times 2}$.
\item If there is any $s \in [-1, 1]$ such that \eqref{eq:quadratic-form}
is satisfied for all $M \in \R^{2 \times 2}$, then the same holds true for
$s = s_0$.
\end{enumerate}
\end{lemma}

\begin{proof}
We first consider a matrix $\Lambda$ such that $2\det \Lambda = |\Lambda|^2$.
Then $\Lambda = (\begin{smallmatrix} a & b \\ -b & a \end{smallmatrix})$
for some $a, b \in \R$. We then calculate $g(\Lambda) = \frac{1}{2} |\Lambda|^2$
and $s_0 = 1$.
If \eqref{eq:quadratic-form} is satisfied for some $s \in [-1, 1]$,
then inserting $M = \Lambda$ yields
$|\Lambda|^2 \le 2$, i.e., $g(\Lambda) \le 1$. Conversely, if $g(\Lambda) \le 1$,
then $a^2 + b^2 \le 1$. Hence
\[
\begin{split}
(\Lambda : M)^2 & = \bigl(a(m_{11} + m_{22}) + b(m_{12} - m_{21})\bigr)^2 \\
& \le (m_{11} + m_{22})^2 + (m_{12} - m_{21})^2 \\
& = |M|^2 + 2\det M
\end{split}
\]
by the Cauchy-Schwarz inequality. Both statements of the lemma follow immediately.

If $2\det \Lambda = -|\Lambda|^2$, then we can use practically the same arguments, except that a few signs will change in the above calculations.

We now assume that $2|\det \Lambda| < |\Lambda|^2$.
In this case, we first note that \eqref{eq:quadratic-form} cannot be
satisfied for $s = \pm 1$. Indeed, if it did hold true for $s = 1$, then we
could test it with the matrices $(\begin{smallmatrix} 1 & 0 \\ 0 & -1 \end{smallmatrix})$ and $(\begin{smallmatrix} 0 & 1 \\ 1 & 0 \end{smallmatrix})$
to find that $\Lambda = (\begin{smallmatrix} a & b \\ -b & a \end{smallmatrix})$
for some $a, b \in \R$. That is, we would find that we are in fact in the first
case. For $s = -1$, the arguments are similar.

We therefore consider $s \in (-1, 1)$ now. Define the bilinear form
\[
\scp{M}{N}_s = M : N + s(m_{11} n_{22} - m_{12} n_{21} - m_{21} n_{12} + m_{22} n_{11}).
\]
It is easy to see that this constitutes an inner product on $\R^{2 \times 2}$.
We also note that
\[
|M|_s^2 \coloneqq \scp{M}{M}_s = |M|^2 + 2s \det M.
\]
By the Riesz representation theorem, there exists
$\Theta_s \in \R^{2 \times 2}$ such that
\[
\Lambda : M = \scp{\Theta_s}{M}_s
\]
for all $M \in \R^{2 \times 2}$. By the Cauchy-Schwarz inequality,
inequality \eqref{eq:quadratic-form} is satisfied if, and only if,
\[
|\Theta_s|_s^2 \le 1.
\]

We can easily determine $\Theta_s = (\begin{smallmatrix} \theta_{11} & \theta_{12} \\ \theta_{21} & \theta_{22} \end{smallmatrix})$ by solving a linear system of equations.
We obtain
\begin{align*}
\theta_{11} & = \frac{\lambda_{11} - s\lambda_{22}}{1 - s^2}, &
\theta_{12} & = \frac{\lambda_{12} + s\lambda_{21}}{1 - s^2}, \\
\theta_{21} & = \frac{\lambda_{21} + s\lambda_{12}}{1 - s^2}, &
\theta_{22} & = \frac{\lambda_{22} - s\lambda_{11}}{1 - s^2}. 
\end{align*}
Therefore,
\[
|\Theta_s|_s^2 = \Lambda : \Theta_s = \frac{|\Lambda|^2 - 2 s \det \Lambda}{1 - s^2}.
\]
Thus we see that for $-1 < s < 1$, inequality
\eqref{eq:quadratic-form} is satisfied if, and only if,
\[
\frac{|\Lambda|^2 - 2 s \det \Lambda}{1 - s^2} \le 1.
\]

We now define
\[
\phi(s) = \frac{|\Lambda|^2 - 2 s \det \Lambda}{1 - s^2}, \quad -1 < s < 1,
\]
and minimise this function over $(-1, 1)$. Differentiating, we compute
\[
\phi'(s) = - \frac{2s^2 \det \Lambda - 2s|\Lambda|^2 + 2\det \Lambda}{(1 - s^2)^2}.
\]
The derivative has a unique zero in $(-1, 1)$, which is at
\begin{equation} \label{eq:s0}
\frac{|\Lambda|^2 - \sqrt{|\Lambda|^4 - 4(\det \Lambda)^2}}{2 \det \Lambda} = s_0
\end{equation}
(unless $\det \Lambda = 0$, in which case the left-hand side is meaningless
but the unique zero is still at $s_0$).
Moreover, we know that $\phi(s) \to \infty$ as $s \nearrow 1$ or $s \searrow -1$.
It follows that $\phi$ has a unique minimum, which is attained at $s_0$. We further compute
\[
\phi(s_0) = g(\Lambda).
\]
(This is easier to calculate with the expression for $s_0$ on the left-hand
side of \eqref{eq:s0} rather than in the definition of $s_0$.)

Hence if there is any $s \in [-1, 1]$
such that \eqref{eq:quadratic-form} holds true, then $\phi(s) \le 1$,
and it follows that $g(\Lambda) = \phi(s_0) \le 1$. The number $s_0$ then
also satisfies \eqref{eq:quadratic-form}. Conversely, if $g(\Lambda) \le 1$,
then we still conclude that \eqref{eq:quadratic-form} holds true for
$s = s_0$.
\end{proof}

\subsection{A regularity gap}

The combination of Proposition \ref{prp:pointwise-characterisation} and
Lemma \ref{lem:minimising-over-sigma} suggests that the functions $\Phi$
and $\alpha$ give rise to an inequality of the form \eqref{eq:calibration}
if $g(\Xi) \le W$. Assuming that $\Phi$ is the quantity we are interested
in primarily, we may also wish to minimise $g(D\Phi + \alpha I)$ over
$\alpha$, which will give $\alpha = -\frac{1}{2} \div \Phi$.
(Then $\Xi$ is the trace free part of $D\Phi$.)
If we define the function $f \colon \R^{2 \times 2} \to \R$ by
\[
\begin{split}
f(M) & = g\left(M - \frac{\tr M}{2} I\right) \\
& = \frac{1}{2} \left(|M|^2 - \frac{1}{2} (\tr M)^2 + |m_{12} - m_{21}| \sqrt{|M|^2 - 2\det M}\right),
\end{split}
\]
then we obtain the inequality $f(D\Phi) \le W$.
Lemma \ref{lem:minimising-over-sigma}
then also gives a good idea of how to choose $\sigma = \curl \omega$.

In the following sections, we will indeed construct $\Phi$ such that $f(D\Phi) \le W$.
Unfortunately, this function will not satisfy the
regularity requirements of the preceding subsection. For this reason, we have to use
a regularisation scheme, which eventually necessitates the construction of a family of
vector fields $\omega_\delta$ rather than a single $\omega$ as in Proposition \ref{prp:pointwise-characterisation}.

Another technical difficulty arises from the fact that once we have $\sigma$, we need to
invert the $\curl$ operator. It does not quite suffice to use standard results here,
because we do not necessarily have decay at infinity for $\sigma$, but we still want to
control the growth of $\omega$. We use the following result here.

\begin{lemma} \label{lem:inverting-curl}
There exists a constant $C$ such that the following holds true.
Suppose that $\sigma \in C^{0, 1/2}(\R^2)$ is bounded. Then there exists
$\omega \in C^1(\R^2; \R^2)$ such that $\curl \omega = \sigma$ and
\[
\sup_{y \in \R^2} \frac{|\omega(y)|}{1 + |y| \log |y|} \le C\sup_{y \in \R^2} |\sigma(y)|.
\]
\end{lemma}

\begin{proof}
We use a (one-sided) dyadic decomposition of $\sigma$ in terms of a partition of unity
\[
1 = \sum_{k = 0}^\infty \eta_k,
\]
where $\eta_k \in C_0^\infty(\R^2)$ are functions such that $0 \le \eta_k \le 1$
for all $k \in \N_0$ and $\supp \eta_0 \subseteq B_2(0)$, while
$\supp \eta_k \subseteq B_{2^{k + 1}}(0) \setminus B_{2^{k - 1}}(0)$ for
$k \ge 1$.

Let $G$ denote the fundamental solution of the Laplace equation in $\R^2$.
For $k \in \N_0$, define $\tilde{\phi}_k = G * (\eta_k \sigma)$. Then
$\Delta \tilde{\phi}_k = \eta_k \sigma$. Standard Schauder estimates imply
that $\tilde{\phi}_k \in C^{2, 1/2}(\R^2)$. Furthermore, we know that $\tilde{\phi}_k$ is
smooth in $B_{2^{k - 1}}(0)$ for $k \ge 1$.

Set $S = \|\sigma\|_{L^\infty(\R^2)}$.
For any $y \in \R^2$, we have the estimate
\[
\begin{split}
|D\tilde{\phi}_k(y)| & = \frac{1}{2\pi} \left|\int_{\R^2} \frac{y - z}{|y - z|^2} \eta_k(z) \sigma_k(z) \, dz\right| \\
& \le \frac{S}{2\pi} \int_{B_{2^{k + 1}}(0)} \frac{dz}{|y - z|} \\
& \le \frac{S}{2\pi} \int_{B_{2^{k + 1}}(0)} \frac{dz}{|z|} = 2^{k + 1} S
\end{split}
\]
for all $k \in \N_0$. If $k \ge 1$, then we also estimate
\[
|D^2 \tilde{\phi}_k(0)| \le \frac{S}{2\pi} \int_{B_{2^{k + 1}}(0) \setminus B_{2^{k - 1}}(0)} \frac{dz}{|z|^2} = S \log 4.
\]
Let $R \ge 1$. If $k \in \N$ is such that $2^{k - 2} \ge R$, then with the
same arguments, we find a universal constant $C_1 \ge 0$ such that
\[
|D^3 \tilde{\phi}_k(y)| \le 2^{-k} C_1S
\]
for all $y \in B_R(0)$.

Now define $\phi_0 = \tilde{\phi}_0$ and
\[
\phi_k(y) = \tilde{\phi}_k(y) - \tilde{\phi}_k(0) - D\tilde{\phi}_k(0)y - \frac{1}{2} D^2\tilde{\phi}_k(0)(y, y), \quad k \ge 1.
\]
Then $\phi_k(0) = 0$, $D\phi_k(0) = 0$, and $D^2\phi_k(0) = 0$ for $k \ge 1$.
We still compute $\Delta \phi_k = \eta_k \sigma$ in $\R^2$. (For $k \ge 1$,
this is because $\Delta \tilde{\phi}_k(0) = 0$.) Moreover, if
$2^{k - 2} \ge R$, then
\[
|D^3 \phi_k(y)| \le 2^{-k} C_1S,
\]
which implies that there exists a universal constant $C_2$ such that
\[
R^{-3} |\phi_k| + R^{-2} |D\phi_k| + R^{-1} |D^2\phi_k| \le 2^{-k} C_2 S
\]
uniformly in $B_R(0)$ when $2^{k - 2} \ge R$. Therefore, the series
\[
\phi = \sum_{k = 0}^\infty \phi_k
\]
converges in $C^2(\overline{B_R(0)})$ for any $R \ge 1$. In particular,
the function $\phi$ is twice continuously differentiable.

Furthermore, for all $k \in \N_0$, we find that
\[
|D\phi_k| \le C_3(2^k + R) S
\]
in $B_R(0)$ for another universal constant $C_3$. If we choose $k_0$ such that
$2^{k_0 - 3} \le R < 2^{k_0 - 2}$, then
\[
\begin{split}
|D\phi(y)| & \le C_3 S \sum_{k = 0}^{k_0 - 1} (2^k + R) + C_2 SR^2 \sum_{k = k_0}^\infty 2^{-k} \\
& \le C_3 S (2^{k_0} + k_0 R) + 2^{1 - k_0} C_2 SR^2
\end{split}
\]
for all $y \in B_R(0)$. Thus we find a universal constant $C_4$ such that
\[
|D\phi (y)| \le C_4 S(1 + |y|\log|y|).
\]

Now we set
\[
\omega = \begin{pmatrix} -\dd{\phi}{y_2} \\ \dd{\phi}{y_1} \end{pmatrix},
\]
and then we have all the desired properties.
\end{proof}

We can now prove the following.

\begin{proposition} \label{prp:calibration}
There exists a constant $C \ge 0$ with the following property.
Let $\Phi \in \bigcap_{p < \infty} W_\loc^{1, p}(\R^2; \R^2)$ and set
$\alpha = - \frac{1}{2} \div \Phi$. Define
$\Phi_\delta = \rho_\delta * \Phi$, $\alpha_\delta = \rho_\delta * \alpha + \delta$, and
\[
W_\delta = \frac{\rho_\delta * W}{1 - \delta} + \delta.
\]
If $f(D\Phi) \le W$, then for every $\delta > 0$ there exists
$\omega_\delta \in C^1(\R^2; \R^2)$ such that
\begin{equation} \label{eq:calibration-approximation}
\div \Phi_\delta(u)  + \alpha_\delta(u) \div u \le \frac{\epsilon}{2} |Du|^2 + \frac{1}{2\epsilon} W_\delta(u) - \epsilon \div \bigl((Du)^\transpose \omega_\delta(u)\bigr)^\perp
\end{equation}
for every $u \in C^2(B_1(0); \R^2)$ and every $\epsilon > 0$. Furthermore,
\begin{equation} \label{eq:growth-omega}
|\omega_\delta(y)| \le C(1 + |y| \log |y|)
\end{equation}
for every $y \in \R^2$.
\end{proposition}

\begin{proof}
In addition to the above quantities, define
$\tilde{\alpha}_\delta = \rho_\delta * \alpha$ and
$\tilde{W}_\delta = \rho_\delta * W$. Furthermore, set $\Xi = D\Phi + \alpha I$ and
$\tilde{\Xi}_\delta = D\Phi_\delta + \tilde{\alpha}_\delta I$.
The inequality $f(D\Phi) \le W$ is equivalent to $g(\Xi) \le W$.
Then by Jensen's inequality, the convexity of $g$ implies that
\[
\begin{split}
g(\tilde{\Xi}_\delta(y)) & = g\left(\int_{B_\delta(y)} \rho_\delta(y - z) \Xi(z) \, dz\right) \\
& \le \int_{B_\delta(y)} \rho_\delta(y - z) g(\Xi(z)) \, dz \\
& \le \int_{B_\delta(y)} \rho_\delta(y - z) W(z) \, dz  = \tilde{W}_\delta(y).
\end{split}
\]
Now recall that $\alpha_\delta = \tilde{\alpha}_\delta + \delta$. Set
$\Xi_\delta = D\Phi_\delta + \alpha_\delta I = \tilde{\Xi}_\delta + \delta I$.
Then \eqref{eq:interpolation-g} gives
\[
g(\Xi_\delta) \le \frac{g(\tilde{\Xi}_\delta)}{1 - \delta} + \frac{g(\delta I)}{\delta} \le W_\delta.
\]

Lemma \ref{lem:minimising-over-sigma}
implies that
\[
(\Xi_\delta(y) : M)^2 \le W_\delta(y) \bigl(|M|^2 + 2 \sigma_\delta \det M\bigr)
\]
for all $M \in \R^{2 \times 2}$, where
\[
\sigma_\delta = \frac{\det \Xi_\delta}{g(\Xi_\delta)}.
\]
Note that $g(\Xi_\delta) \ge \frac{1}{2}|\Xi_\delta|^2 \ge \frac{1}{4} (\tr \Xi_\delta)^2 = \delta^2$. As $\Xi_\delta$ is smooth, it follows that
$\sigma_\delta \in C^{0, 1/2}(\R^2)$. It is clear that $|\sigma_\delta| \le 1$.

Lemma \ref{lem:inverting-curl} provides vector fields
$\omega_\delta \in C^1(\R^2; \R^2)$ such that
$\curl \omega_\delta = \sigma_\delta$ and such that \eqref{eq:growth-omega}
is satisfied for a universal constant $C$.
Inequality \eqref{eq:calibration-approximation} then follows from
Proposition \ref{prp:pointwise-characterisation}.
\end{proof}

We conclude this section with a brief discussion of how we proceed in the proofs
of our main results.
Owing to Proposition \ref{prp:calibration}, one of the central questions of this
paper is now whether we can satisfy the inequality
\[
f(D\Phi) \le W
\]
while simultaneously keeping $\Phi_1(a^+) - \Phi_1(a^-)$ large enough to
obtain a useful estimate. (The best possible value here is of course
\[
\Phi_1(a^+) - \Phi_1(a^-) = \int_{[a^-, a^+]} \sqrt{W} d\Ha^1.)
\]
Thus we want to find the number
\[
\eta_0 = \sup\set{\Phi_1(a^+) - \Phi_1(a^-)}{\Phi \in C^{0, 1}(\R^2; \R^2) \text{ with } f(D\Phi) \le W}.
\]
Alternatively, assuming that $W > 0$, we can define
\[
F(y, M) = \frac{f(M)}{W(y)}
\]
and try to determine
\[
e_\infty = \inf\set{\|F(y, D\Phi)\|_{L^\infty(\R^2)}^{1/2}}{\Phi_1(a^+) - \Phi_1(a^-) = 1}.
\]
(In general, we are interested in potential functions
$W$ that do have zeroes, but they can be approximated by
positive functions.)

For $\epsilon > 0$, if $\Phi \in C^{0, 1}(\R^2; \R^2)$ satisfies $f(D\Phi) \le W$ and $\Phi_1(a^+) - \Phi_1(a^-) \ge \eta_0 - \epsilon$, then the function
$\Psi = \Phi/(\Phi_1(a^+) - \Phi_1(a^-))$ will satisfy $\Psi_1(a^+) - \Psi_1(a^-) = 1$ and $F(y, D\Psi) \le (\eta_0 - \epsilon)^{-2}$.
It then follows that $e_\infty \le 1/\eta_0$. A similar scaling argument shows that $\eta_0 \ge 1/e_\infty$.
Thus $\eta_0 = 1/e_\infty$.

We therefore study problems of this nature in the next section.
Since these results are potentially of independent interest, we formulate them
more generally here.

\section{An $L^\infty$-minimisation problem} \label{sct:L-infinity}

In this section, we assume that $n, m \in \N$ and study functions
$\phi \colon \R^n \to \R^m$. Let $A \subseteq \R^n$ be a finite set
and $\Omega = \R^n \setminus A$. Given $m_0 \in \{1, \dotsc, m\}$, we
fix a non-constant function $\phi^0 \colon A \to \R^{m_0}$.
We further assume that $F \colon \R^n \times \R^{m \times n} \to [0, \infty)$
is a continuously differentiable function such that for every $x \in \R^n$, the function
$F(x, \blank)$ is homogeneous of degree $2$, $C^2$-regular away from $0$,
and uniformly strictly convex in the sense that there exists a constant
$c > 0$ such that
\begin{equation} \label{eq:strict-convexity}
D_M^2 F(x, M)(N, N) \ge 2c|N|^2
\end{equation}
for all $x \in \R^n$ and all $M, N \in \R^{m \times n}$, where $D_M^2F$
denotes the second derivative with respect to the second argument.
(The function $F(x, \blank)$ is not twice Fr\'echet differentiable at $0$
in general, but we can always interpret the left-hand side of
\eqref{eq:strict-convexity} in the G\^ateaux sense even at $0$.)
We further write $\grad F$ for the gradient of
$F$ with respect to the second argument only. We assume that there exists
another constant $C > 0$ such that
\begin{equation} \label{eq:uniformly-bounded}
|\grad F(x, M)| \le 2C|M|
\end{equation}
for all $x \in \R^n$ and $M \in \R^{m \times n}$.

We also consider the Legendre transform
\[
F^*(x, N) = \sup_{M \in \R^{m \times n}} (M : N - F(x, M))
\]
of $F$ with respect to the second argument.

Motivated by the previous section, we study functions that minimise the
functional
\[
E_\infty(\phi) = \esssup_{x \in \R^n} \sqrt{F(x, D\phi(x))}
\]
subject to the condition $\phi_i = \phi_i^0$ on $A$ for $i = 1, \dotsc, m_0$.

\begin{notation}
In this section, we write $B_r(x)$ for an open ball in $\R^n$ with radius $r > 0$ and centre $x \in \R^n$.
The symbol $\rho$ now denotes a function $\rho \in C_0^\infty(B_1(0))$ with $\rho \ge 0$ and
$\int_{B_1(0)} \rho \, dx = 1$. For $\delta > 0$, we then set $\rho_\delta(x) = \delta^{-n} \rho(x/\delta)$.
\end{notation}

\subsection{Summary of the results}

Following the ideas from a paper of Katzourakis and Moser \cite{Katzourakis-Moser:25}, we
derive some properties of the minimisers of $E_\infty$. They will be described in terms
of an $\R^m$-valued $1$-current on $\R^n$ and a mass functional depending
on $F$. Although currents will normally be defined in terms of differential
forms and the exterior derivative, if we only consider $1$-currents, then we
can work with matrix-valued functions and the Fr\'echet derivative instead.

The following is  a more general version of Definition \ref{def:current-intro}.

\begin{definition} \label{def:current}
An \emph{$\R^m$-valued $1$-current on $\R^n$} is an element of the dual space
of $C_0^\infty(\R^n; \R^{m \times n})$. If $T$ is an $\R^m$-valued
$1$-current, then its \emph{boundary} $\partial T$ is the $\R^m$-valued distribution such that $\partial T(\xi) = T(D\xi)$
for every $\xi \in C_0^\infty(\R^n; \R^m)$. The \emph{$F$-mass} of $T$
is
\[
\M_F(T) = \frac{1}{2} \sup\set{T(\zeta)}{\zeta \in C_0^\infty(\R^n; \R^{m \times n}) \text{ with } \|F(x, \zeta)\|_{C^0(\R^n)} \le 1}.
\]
We say that $T$ is \emph{normal} if there exists
$C \ge 0$ such that
\[
|T(\zeta)| + |\partial T(\xi)| \le C \left(\sup_{x \in \R^n} |\zeta(x)| + \sup_{n \in \R^n} |\xi(x)|\right)
\]
for all $\zeta \in C_0^\infty(\R^n; \R^{m \times n})$ and all $\xi \in C_0^\infty(\R^n; \R^m)$.
We write $\ncrt_{m \times n}(\R^n)$ for the space of all normal $\R^m$-valued $1$-currents on $\R^n$.
\end{definition}

Given $i \in \{1, \dotsc, m\}$, we write $T_i$ for the $\R$-valued
$1$-current such that $T_i(\xi) = T(e_i \otimes \xi)$ for
$\xi \in C_0^\infty(\R^n; \R^n)$, where $e_i$ denotes the $i$-th standard basis
vector of $\R^m$. We can think of these as the components of $T$. 

Many of the standard properties of currents, as described, e.g., in a book
by Simon \cite{Simon:83}, also apply to this variant. In particular, if $T$ has
finite $F$-mass, then, by the properties of $F$, it automatically has finite
mass in the standard sense. (In the above terminology, that means that $\M_{\tilde{F}}(T) < \infty$
for the function $\tilde{F}(x, M) = \frac{1}{4} |M|^2$.) It then follows that
there exist a Radon measure $\|T\|$ and a $\|T\|$-measurable, matrix-valued function
$\vec{T} \colon \R^n \to \R^{m \times n}$ with $|\vec{T}| = 1$ almost
everywhere, such that
\[
T(\zeta) = \int_{\R^n} \vec{T} : \zeta \, d\|T\|
\]
for any $\zeta \in C_0^\infty(\R^n; \R^{m \times n})$. In this situation, we can
also make sense of the expression $T(\zeta)$ for $\zeta \in C_0^0(\R^n; \R^{m \times n})$.
If $T$ is normal, then we can make sense of $\partial T(\xi)$ for any $\xi \in C_0^0(\R^n; \R^m)$.

We write $W_*^{1, \infty}(\R^n; \R^m)$ for the space of all
$\phi \in W_\loc^{1, \infty}(\R^n; \R^m)$ such that $\phi_i = \phi_i^0$ on
$A$ for $i = 1, \dotsc, m_0$. We set
\[
e_\infty = \inf\set{E_\infty(\phi)}{\phi \in W_*^{1, \infty}(\R^n; \R^m)}.
\]

We will prove the following two statements.

\begin{proposition} \label{prp:existence}
There exists $\phi_\infty \in W_*^{1, \infty}(\R^n; \R^m)$ such that $E_\infty(\phi_\infty) = e_\infty$.
\end{proposition}

\begin{theorem} \label{thm:current}
There exists $T \in \ncrt_{m \times n}(\R^n) \setminus \{0\}$ with the following properties.
\begin{enumerate}
\item $\supp \partial T \subseteq A$, and $\partial T_i = 0$ for
$i = m_0 + 1, \dotsc, m$.
\item If $S \in \ncrt_{m \times n}(\R^n)$ satisfies $\partial S = \partial T$,
then $\M_F(T) \le \M_F(S)$.
\item Let $\phi \in W_*^{1, \infty}(\R^n; \R^m)$ be a minimiser of
$E_\infty$ in $W_*^{1, \infty}(\R^n; \R^m)$ and let
$e_\infty = E_\infty(\phi)$. Then $\partial T(\phi) = 2 e_\infty \M_F(T)$ and
\[
\lim_{\delta \searrow 0} \int_{\R^n} \Biggl|\rho_\delta * D\phi - e_\infty \frac{\grad F^*(x, \vec{T})}{\sqrt{F^*(x, \vec{T})}}\Biggr|^2 \, d\|T\| = 0.
\]
\end{enumerate}
\end{theorem}

Note that the last statement gives a lot of information about the behaviour
of $\phi$ on $\supp T$. Indeed, we interpret it as a generalised version of
the equation
\[
D\phi = e_\infty \frac{\grad F^*(x, \vec{T})}{\sqrt{F^*(x, \vec{T})}},
\]
but since $\supp T$ will be a Lebesgue null set in general, such an equation
does not make sense pointwise.

The existence of a minimiser of $E_\infty$ can be proved with the direct method,
although, as one needs to work with weak* convergence in an $L^\infty$-space, some
of the details are not so obvious. We will use a different method for the proof of
Proposition \ref{prp:existence}, because we will obtain the minimiser as a side product
of the arguments for the proof of Theorem \ref{thm:current}.

\subsection{Properties of $F$ and $F^*$}

We first derive some properties of the function $F$ and its Legendre transform
$F^*$ that follow from the above assumptions, in particular the strict
uniform convexity \eqref{eq:strict-convexity}.

The homogeneity of $F$ implies that
\begin{equation} \label{eq:homogeneity}
\grad F(x, M) : M = 2F(x, M).
\end{equation}
Hence \eqref{eq:uniformly-bounded} gives rise to the inequality
\[
F(x, M) \le C|M|^2.
\]
According to Taylor's theorem, for any $x \in \R^n$ and any
$M, N \in \R^{m \times n}$, there exists $\theta \in (0, 1)$ such that
\begin{multline*}
F(x, N) = F(x, M) + \grad F(x, M) : (N - M) \\
+ \frac{1}{2} D_M^2 F(x, \theta M + (1 - \theta) N) (N - M, N - M).
\end{multline*}
Therefore,
\begin{equation} \label{eq:convexity1}
c|N - M|^2 \le F(x, N) - F(x, M) - \grad F(x, M) : (N - M).
\end{equation}
Using \eqref{eq:homogeneity} again, we can write the above inequality in the form
\begin{equation} \label{eq:convexity2}
c|N - M|^2 \le F(x, N) + F(x, M) - \grad F(x, M) : N.
\end{equation}
Inserting $N = 0$, we also see that
\[
F(x, M) \ge c|M|^2.
\]

The Legendre transform of $F$ is automatically homogeneous of degree $2$
and strictly convex again.
It is a well-known property of the Legendre transform that $\grad F^*(x, \blank)$
is the inverse of $\grad F(x, \blank)$ regarded as a map
$\R^{m \times n} \to \R^{m \times n}$. That is, if $N = \grad F(x, M)$,
then $M = \grad F^*(x, N)$, and vice versa. We then also find that
\[
2F^*(x, N) = \grad F^*(x, N) : N = M : \grad F(x, M) = 2F(x, M).
\]
That is,
\[
F(x, M) = F^*(x, \grad F(x, M)) \quad \text{and} \quad F^*(x, N) = F(x, \grad F^*(x, N)).
\]

In addition to the standard definition of the Legendre transform, we have
the following characterisation.

\begin{lemma} \label{lem:Legendre-transform}
For any $x \in \R^n$ and $N \in \R^{m \times n}$,
\[
F^*(x, N) = \frac{1}{4} \sup\set{(M : N)^2}{M \in \R^{m \times n} \text{ with } F(x, M) \le 1}.
\]
\end{lemma}

\begin{proof}
Fix $x \in \R^n$. For every $M \in \R^{m \times n} \setminus \{0\}$ there
exists $t > 0$ such that $F(x, tM) = 1$. Hence
\[
\begin{split}
F^*(x, N) & = \sup_{F(x, M) = 1} \, \sup_{t \in \R} \ (tM : N - F(x, tM)) \\
& = \sup_{F(x, M) = 1} \, \sup_{t \in \R} \ (tM : N - t^2).
\end{split}
\]
The function $t \mapsto tM : N - t^2$ attains its maximum at
$t = \frac{1}{2} M : N$. Therefore,
\[
F^*(x, N) = \sup_{F(x, M) = 1} \frac{(M : N)^2}{4}.
\]
Clearly this supremum is identical with the one in the lemma.
\end{proof}

As a consequence, we have an alternative representation of the $F$-mass $M_F$.

\begin{proposition} \label{prp:F-mass}
Let $T$ be an $\R^m$-valued $1$-current in $\R^n$ with finite $F$-mass. Then
\[
\M_F(T) = \int_{\R^n} \sqrt{F^*(x, \vec{T})} \, d\|T\|.
\]
\end{proposition}

\begin{proof}
If $\zeta \in C_0^\infty(\R^n; \R^{m \times n})$ with
$\|F(x, \zeta)\|_{C^0(\R^n)} \le 1$, then
\[
T(\zeta) = \int_{\R^n} \zeta : \vec{T} \, d\|T\| \le 2\int_{\R^n} \sqrt{F^*(x, \vec{T})} \, d\|T\|
\]
by Lemma \ref{lem:Legendre-transform}. It follows that
\[
\M_F(T) \le \int_{\R^n} \sqrt{F^*(x, \vec{T})} \, d\|T\|.
\]

To prove the reverse inequality, we first note that there exists 
a sequence of uniformly bounded functions
$\vec{T}_k \in C_0^\infty(\R^n; \R^{m \times n})$ such that
$\vec{T}_k \to \vec{T}$ almost everywhere with respect to $\|T\|$.
Define
\[
\zeta_k(x) = \frac{k\grad F^*(x, \vec{T}_k(x))}{\sqrt{k^2 F^*(x, \vec{T}_k(x)) + 1}}, \quad k \in \N.
\]
Then
\[
\vec{T}_k(x) : \zeta_k(x) = \frac{2k F^*(x, \vec{T}_k(x))}{\sqrt{k^2 F^*(x, \vec{T}_k(x)) + 1}}
\]
for every $x \in \R^n$. Hence
\[
\vec{T}_k(x) : \zeta_k(x) \to 2\sqrt{F^*(x, \vec{T}(x))}
\]
almost everywhere. We also compute
\[
F(x, \zeta_k(x)) = \frac{k^2 F(x, \grad F^*(x, \vec{T}_k(x)))}{k^2 F^*(x, \vec{T}_k(x)) + 1} = \frac{k^2 F^*(x, \vec{T}_k(x))}{k^2 F^*(x, \vec{T}_k(x)) + 1} \le 1.
\]
Therefore,
\[
\int_{\R^n} \sqrt{F^*(x, \vec{T})} \, d\|T\| = \frac{1}{2} \lim_{k \to \infty}\int_{\R^2} \vec{T}_k : \zeta_k \, d\|T\| \le \M_F(T).
\]
This completes the proof.
\end{proof}

\begin{lemma} \label{lem:F-composed-with-convolution}
Let $K \subseteq \R^n$ be compact. Let $Q \in L^\infty(\R^n; \R^{m \times n})$, and let $U \subseteq \R^n$ be an open set with $K \subseteq U$.
Then
\[
\limsup_{\delta \searrow 0} \sup_{x \in K} F(x, \rho_\delta * Q(x)) \le \esssup_{x \in U} F(x, Q(x)).
\]
\end{lemma}

\begin{proof}
Since $F$ is continuous, it is uniformly continuous on the compact set
\[
L = \set{(x, M) \in \R^n \times \R^{m \times n}}{\dist(x, K) \le 1 \text{ and } |M| \le \|Q\|_{L^\infty(\R^n)}}.
\]
Let $\epsilon > 0$, and fix $\delta_0 \in (0, 1]$ such that
$|F(x, M) - F(y, M)| \le \epsilon$ for all $(x, M) \in L$ and $(y, M) \in L$ with $|x - y| \le \delta_0$.
Let $x \in K$. By the convexity of $F$ and Jensen's inequality, we can estimate
\[
\begin{split}
F(x, \rho_\delta * Q) & \le \int_{\R^n} \rho_\delta(x - y) F(x, Q(y)) \, dy \\
& \le \int_{\R^n} \rho_\delta(x - y) F(y, Q(y)) \, dy + \epsilon \\
& \le \esssup_{y \in U} F(y, Q(y)) + \epsilon
\end{split}
\]
when $\delta < \delta_0$ and $\delta < \dist(K, \R^n \setminus U)$. The claim follows.
\end{proof}

\subsection{Approximation in $L^p$}

In order to take advantage of the usual tools from the calculus of variations,
we construct a minimiser of $E_\infty$ as the limit of solutions to more conventional
problems. To this end, we replace the $L^\infty$-norm by $L^p$-norms.

Let $V \colon \R^n \to (0, \infty)$ be a smooth function such that
\[
\int_{\R^n} V(x) \, dx = 1.
\]
For $n < p < \infty$, define the functionals
\[
E_p(\phi) = \left(\int_{\R^n} \bigl(F(x, D\phi)\bigr)^{p/2} V(x) \, dx\right)^{1/p}
\]
on the Sobolev space $W_\loc^{1, p}(\R^n; \R^m)$. Let $W_*^{1, p}(\R^n; \R^m)$ denote the space of all
$\phi \in W_\loc^{1, p}(\R^n; \R^m)$ such that $\phi_i = \phi_i^0$ on $A$
for $i = 1, \dotsc, m_0$. We consider the problem of minimising
$E_p$ in $W_*^{1, p}(\R^n; \R^m)$.

Since we now have a strictly convex functional given in terms of an integral,
we can use standard methods from the calculus of variations to make a few
statements immediately: there exists a unique minimiser
$\phi_p \in W_*^{1, p}(\R^n; \R^m)$, which satisfies the Euler-Lagrange
equation
\[
\div\left(V(x) \bigl(F(x, D\phi_p)\bigr)^{p/2 - 1} \grad F(x, D\phi_p)\right) = 0
\]
weakly in $\Omega$.
(Here the divergence is applied row-wise, so that we actually have a system
of $m$ equations.) We use the notation $\grad_i F$ for the $i$-th row of $\grad F$.
Then we can write
\begin{equation} \label{eq:Euler-Lagrange0}
\div\left(V(x) \bigl(F(x, D\phi_p)\bigr)^{p/2 - 1} \grad_i F(x, D\phi_p)\right) = 0, \quad i = 1, \dotsc, m.
\end{equation}
We note that this is satisfied weakly in $\Omega$ for $i = 1, \dotsc m_0$,
and even weakly in $\R^n$ for $i = m_0 + 1, \dotsc m$.

We obtain another necessary condition when we study inner variations of
$\phi_p$ of the form $\phi_p^t(x) = \phi_p(x + t\chi(x))$ for a vector field
$\chi \in C_0^\infty(\Omega; \R^2)$. A standard computation then gives
\begin{equation} \label{eq:inner-variations0}
\begin{split}
0 & = \left.\frac{d}{dt}\right| \bigl(E_p(\phi_p^t)\bigr)^p \\
& = \frac{p}{2} \int_{\R^n} V(x) \bigl(F(x, D\phi_p)\bigr)^{\frac{p - 2}{2}} \left(\grad F(x, D\phi_p) : (D\phi_p D\chi) - \dd{F}{x}(x, D\phi_p) \chi\right) \, dx \\
& \quad - \int_{\R^n} \bigl(F(x, D\phi_p)\bigr)^{p/2} \left(DV \chi + V \div \chi\right) \, dx,
\end{split}
\end{equation}
where $\dd{F}{x}$ denotes the derivative of $F$ with respect to the first
argument.

\begin{proof}[Proof of Proposition \ref{prp:existence}]
Let $\psi \in W_*^{1, \infty}(\R^n; \R^m)$. By H\"older's inequality,
for $n < q < p$, we have the inequalities
\begin{equation} \label{eq:monotone}
E_q(\phi_p) \le E_p(\phi_p) \le E_p(\psi) \le E_\infty(\psi).
\end{equation}
Thus, the family $(\phi_p)_{q < p < \infty}$ is bounded in
$W^{1, q}(B_R(0); \R^m)$ for any $q < \infty$ and any $R > 0$. We may
therefore choose a sequence $p_k \to \infty$ such that we have the
weak convergence $\phi_{p_k} \rightharpoonup \phi_\infty$ as $k \to \infty$
simultaneously in all of these spaces for some $\phi_\infty \in \bigcap_{q < \infty} W_*^{1, q}(\R^n; \R^m)$. We further see that
\begin{equation} \label{eq:e_p->e_infty}
E_\infty(\phi_\infty) = \lim_{q \to \infty} E_q(\phi_\infty) \le \liminf_{q \to \infty} \liminf_{k \to \infty} E_q(\phi_{p_k}) \le \liminf_{k \to \infty} E_{p_k}(\phi_{p_k}) \le E_\infty(\psi)
\end{equation}
by the lower semicontinuity of $E_q$ with respect to weak convergence,
H\"older's inequality again, and \eqref{eq:monotone}. In particular, we
see that $\phi_\infty$ belongs to $W_*^{1, \infty}(\R^n; \R^m)$
and is a minimiser of $E_\infty$ in this space.
\end{proof}

We continue to use the functions $\phi_p$ and $\phi_\infty$ for the rest of this section. Let
\[
e_p = E_p(\phi_p), \quad n < p \le \infty.
\]
Because we have assumed that $\phi^0$ is not constant, it is clear that
$e_p \neq 0$. The inequalities in \eqref{eq:monotone} and \eqref{eq:e_p->e_infty} imply that
$e_\infty = \lim_{p \to \infty} e_p$. It is now convenient to introduce the measures
\[
\mu_p = e_p^{2 - p} V(x) \bigl(F(x, D\phi_p)\bigr)^{p/2 - 1} \Le^n,
\]
where $\Le^n$ denotes the Lebesgue measure in $\R^n$. This is normalised so that
\[
\mu_p(\R^n) = e_p^{2 - p} \int_{\R^n} V(x) \bigl(F(x, D\phi_p)\bigr)^{p/2 - 1} \, dx \le e_p^{2 - p} \bigl(E_p(\phi_p)\bigr)^{p - 2} = 1
\]
by H\"older's inequality. The Euler-Lagrange equation \eqref{eq:Euler-Lagrange0} implies that
\begin{equation} \label{eq:Euler-Lagrange}
\int_{\R^n} \grad F(x, D\phi_p) : D\psi \, d\mu_p = 0
\end{equation}
for all $\psi \in C_0^\infty(\Omega)$. Note, however, that the identity also holds true
if we merely know that $\psi \in W_\loc^{1, p}(\R^n; \R^m)$ and $E_p(\psi) < \infty$ and
$\psi_1 = \dotsb = \psi_{m_0} = 0$ on $A$, because these conditions imply that the derivative
$\frac{d}{dt}|_{t = 0} E_p(\phi + t\psi)$ is given by the usual expression.

Equation \eqref{eq:inner-variations0} has the following representation
in terms of $\mu_p$:
\begin{equation} \label{eq:inner-variation}
\begin{split}
0 & = \frac{p}{2} \int_{\R^n} \left(\grad F(x, D\phi_p) : (D\phi_p D\chi) - \dd{F}{x}(x, D\phi_p) \chi\right) \, d\mu_p \\
& \quad - \int_{\R^n} F(x, \phi_p) \left(D(\log V) \chi + \div \chi\right) \, d\mu_p.
\end{split}
\end{equation}

The measure $\mu_p$ should be considered together with the function
$D\phi_p$. These two objects form a measure-function pair $(\mu_p, D\phi_p)$
with values in $\R^{m \times n}$ in the sense of Hutchinson \cite{Hutchinson:86}.
Since
\[
\int_{\R^n} |D\phi_p|^2 \, d\mu_p \le \frac{1}{c e_p^{p - 2}} \int_{\R^n} \bigl(F(x, D\phi_p)\bigr)^{p/2} V(x) \, dx \le \frac{e_p^2}{c}
\]
for every $p \in (n, \infty)$, we may assume that for the above sequence
$p_k \to \infty$, we simultaneously have the weak convergence of
$(\mu_{p_k}, D\phi_{p_k})$ to a measure-function pair $(\mu_\infty, Z_\infty)$
in the sense of Hutchinson. (This follows from \cite[Theorem 4.4.2]{Hutchinson:86}, but we may have to pass to a suitable subsequence.)
Here $\mu_\infty$ is a Radon measure on $\R^n$ and
$Z_\infty \in L^2(\mu_\infty; \R^{m \times n})$. Of course, we have a
uniform bound for
\[
\int_{\R^n} |\grad F(x, D\phi_p)|^2 \, d\mu_p
\]
as well, and we may therefore assume at the same time that
$(\mu_{p_k}, \grad F(x, D\phi_{p_k}))$ converges weakly to $(\mu_\infty, Y_\infty)$
for some function $Y_\infty \in L^2(\mu_\infty; \R^{m \times n})$.
Because of \eqref{eq:Euler-Lagrange}, we have the identity
\begin{equation} \label{eq:Euler-Lagrange-limit}
\int_{\R^n} Y_\infty : D\psi \, d\mu_\infty =0
\end{equation}
for all $\psi \in C_0^1(\R^n; \R^m)$ with $\psi_1 = \dotsb = \psi_{m_0} = 0$ on $A$.

We now want to prove the following additional properties.

\begin{proposition} \label{prp:measure-function}
\begin{enumerate}
\item \label{itm:strong-convergence}
Let $K \subseteq \Omega$ be a compact set. Then the convergence $(\mu_{p_k}, D\phi_{p_k}) \to (\mu_\infty, Z_\infty)$ is
strong in $K$ the sense of Hutchinson \cite{Hutchinson:86}. Equivalently,
\[
\int_K |Z_\infty|^2 \, d\mu_\infty = \lim_{k \to \infty} \int_K |D\phi_{p_k}|^2 \, d\mu_{p_k}.
\]
\item \label{itm:pointwise}
The identity $F(x, Z_\infty) = e_\infty^2$ holds at $\mu_\infty$-almost
every point $x \in \Omega$.
\item \label{itm:gradient-of-minimiser}
For any minimiser $\psi \in W_*^{1, \infty}(\R^n; \R^m)$ of $E_\infty$,
\[
\lim_{\delta \searrow 0} \int_\Omega |\rho_\delta * D\psi - Z_\infty|^2 \, d\mu_\infty = 0.
\]
\end{enumerate}
\end{proposition}

The proof follows the strategy of the aforementioned paper \cite{Katzourakis-Moser:25}, which in turn
makes use of some ideas of Evans and Yu \cite{Evans-Yu:05} at this point. First, we require the following lemma.

\begin{lemma} \label{lem:cut-off}
Let $\xi \in C^0(\R^n)$ be a bounded function with $\xi \ge 0$.
Then for any $p \in (n, \infty)$ and any $\beta \in (0, 1)$,
\[
\int_{\R^n} \xi F(x, D\phi_p) \, d\mu_p \ge \beta^2 e_p^2 \int_{\R^n} \xi \, d\mu_p - \beta^p e_p^2 \sup_{\R^n} \xi.
\]
\end{lemma}

\begin{proof}
Consider the sets
\[
S_p = \set{x \in \R^n}{F(x, D\phi_p(x)) \le \beta^2 e_p^2}.
\]
Then
\[
\mu_p(S_p) = e_p^{2 - p} \int_{S_p} V(x) \bigl(F(x, D\phi_p)\bigr)^{p/2 - 1} \, dx \le \beta^{p - 2}.
\]
Therefore,
\[
\begin{split}
\int_{\R^n} \xi F(x, D\phi_p) \, d\mu_p & \ge \int_{\R^n \setminus S_p} \xi F(x, D\phi_p) \, d\mu_p \\
& \ge \beta^2 e_p^2 \int_{\R^n \setminus S_p} \xi \, d\mu_p \\
& = \beta^2 e_p^2 \left(\int_{\R^n} \xi \, d\mu_p - \int_{S_p} \xi \, d\mu_p\right) \\
& \ge \beta^2 e_p^2 \int_{\R^n} \xi \, d\mu_p - \beta^p e_p^2 \sup_{R^n} \xi,
\end{split}
\]
as claimed.
\end{proof}

\begin{proof}[Proof of Proposition \ref{prp:measure-function}]
Fix an arbitrary minimiser $\psi \in W_*^{1, \infty}(\R^n; \R^m)$
of $E_\infty$. Define $\psi_\delta = \rho_\delta * \psi$.
Let $\xi \in C_0^\infty(\Omega)$ with $0 \le \xi \le 1$, and let
$\beta \in (0, 1)$. Then by \eqref{eq:convexity1},
\eqref{eq:Euler-Lagrange}, and Lemma \ref{lem:cut-off}, we find that
\[
\begin{split}
\lefteqn{c\int_\Omega \xi |D\psi_\delta - D\phi_p|^2 \, d\mu_p} \quad \\
& \le \int_\Omega \xi \bigl(F(x, D\psi_\delta) - F(x, D\phi_p) - \grad F(x, D\phi_p) : (D\psi_\delta - D\phi_p)\bigr) \, d\mu_p \\
& = \int_\Omega \xi \bigl(F(x, D\psi_\delta) - F(x, D\phi_p)\bigr) \, d\mu_p + \int_\Omega \grad F(x, D\phi_p) : (\psi_\delta - \phi_p) \otimes D\xi \, d\mu_p \\
& \le \int_\Omega \xi F(x, D\psi_\delta) \, d\mu_p - \beta^2 e_p^2 \int_\Omega \xi \, d\mu_p + \beta^p e_p^2 \\
& \quad + \int_\Omega \grad F(x, D\phi_p) : (\psi_\delta - \phi_p) \otimes D\xi \, d\mu_p.
\end{split}
\]
Applying this inequality to $p_k$ and letting $k \to 0$, we obtain
\begin{multline*}
c \limsup_{k \to \infty} \int_\Omega \xi |D\psi_\delta - D\phi_{p_k}|^2 \, d\mu_{p_k} \\
\le \int_\Omega \xi F(x, D\psi_\delta) \, d\mu_\infty - \beta^2 e_\infty^2 \int_\Omega \xi \, d\mu_\infty + \int_\Omega Y_\infty : (\psi_\delta - \phi_\infty) \otimes D\xi \, d\mu_\infty.
\end{multline*}

For $r > 0$, define $B_r(A) = \bigcup_{x \in A} B_r(x)$. Suppose that $r$ is so small that
$B_r(a) \cap B_r(b) = \emptyset$ when $a, b \in A$ with $a \neq b$. Choose $\chi \in C_0^\infty(B_r(0))$
with $0 \le \chi \le 1$ and such that $\chi \equiv 1$ in $B_{r/2}(0)$ and $|D\chi| \le 4/r$
in $B_r(0)$. Furthermore, for $R \ge 1$ such that $B_r(A) \subseteq B_R(0)$, choose $\eta \in C_0^\infty(B_{2R}(0))$
with $0 \le \eta \le 1$ such that $\eta \equiv 1$ in $B_R(0)$ and $|D\eta| \le 2/R$. Set
\[
\xi(x) = \eta(x) - \sum_{a \in A} \chi(x - a).
\]

Note that $\psi_i(a) = \phi_{\infty i}(a)$ for $i = 1, \dotsc, m_0$ and $a \in A$. Hence
\[
\begin{split}
\lefteqn{\int_\Omega Y_\infty : (\psi_\delta - \phi_\infty) \otimes D\xi \, d\mu_\infty} \\
& = \int_{B_{2R}(0) \setminus B_R(0)} Y_\infty : (\psi_\delta - \phi_\infty) \otimes D\eta \, d\mu_\infty \\
& \quad - \sum_{a \in A} \int_{B_r(a)} Y_\infty(x) : (\psi_\delta(x) - \phi_\infty(x)) \otimes D\chi(x - a) \, d\mu_\infty(x) \\
& = \int_{B_{2R}(0) \setminus B_R(0)} Y_\infty : (\psi_\delta - \phi_\infty) \otimes D\eta \, d\mu_\infty \\
& \quad - \sum_{a \in A} \int_{B_r(a)} Y_\infty(x) : \bigl(\psi_\delta(x) - \psi(a) - \phi_\infty(x) + \phi_\infty(a)\bigr) \otimes D\chi(x - a) \, d\mu_\infty(x)
\end{split}
\]
by \eqref{eq:Euler-Lagrange-limit}. Therefore,
\begin{multline*}
c \limsup_{k \to \infty} \int_\Omega \xi |D\psi_\delta - D\phi_{p_k}|^2 \, d\mu_{p_k} \\
\begin{aligned}
& \le \int_\Omega \xi F(x, D\psi_\delta) \, d\mu_\infty - \beta^2 e_\infty^2 \int_\Omega \xi \, d\mu_\infty \\
& \quad + \int_{B_{2R}(0) \setminus B_R(0)} Y_\infty : (\psi_\delta - \phi_\infty) \otimes D\xi \, d\mu_\infty \\
& \quad + \sum_{a \in A} \int_{B_r(a)} Y_\infty : \bigl(\psi_\delta - \psi(a) - \phi_\infty + \phi_\infty(a)\bigr) \otimes D\xi \, d\mu_\infty.
\end{aligned}
\end{multline*}

Since $\phi_\infty$ is a minimiser of $E_\infty$, it satisfies
$F(x, D\phi_\infty) \le e_\infty^2$ almost everywhere. Hence
\[
|D\phi_\infty| \le \left(\frac{1}{c} F(x, D\phi_\infty)\right)^{1/2} \le \frac{e_\infty}{\sqrt{c}}.
\]
Similarly, we find that $|D\psi| \le e_\infty/\sqrt{c}$, as $\psi$ is also
a minimiser of $E_\infty$. Hence
\[
|\psi - \psi(a) - \phi_\infty + \phi_\infty(a)| \le \frac{2r e_\infty}{\sqrt{c}}
\]
in $B_r(a)$. Similarly, if $R \ge |\psi(0) - \phi_\infty(0)|$, then
\[
|\psi - \phi_\infty| \le \left(\frac{4e_\infty}{\sqrt{c}} + 1\right)R
\]
in $B_{2R}(0)$.
If $\delta \le r$, the we also have the inequality
\[
|\psi - \psi_\delta| \le \frac{re_\infty}{\sqrt{c}}.
\]
Therefore,
\begin{multline*}
\int_{B_{2R}(0) \setminus B_R(0)} Y_\infty : (\psi_\delta - \phi_\infty) \otimes D\xi \, d\mu_\infty \\
+ \sum_{a \in A} \int_{B_r(a)} Y_\infty : \bigl(\psi_\delta - \psi(a) - \phi_\infty + \phi_\infty(a)\bigr) \otimes D\xi \, d\mu_\infty \\
\begin{aligned}
& \le \left(\frac{10e_\infty}{\sqrt{c}} + 2\right) \int_{B_{2R}(0) \setminus B_R(0)} |Y_\infty| \, d\mu_\infty + \frac{12 e_\infty}{\sqrt{c}} \int_{B_r(A) \setminus B_{r/2}(A)} |Y_\infty| \, d\mu_\infty \\
& \le \left(\frac{12 e_\infty}{\sqrt{c}} + 2\right) \left(\mu_\infty(G_{R, r}) \int_{\R^n} |Y_\infty|^2 \, d\mu_\infty\right)^{1/2},
\end{aligned}
\end{multline*}
where $G_{R, r} = (B_{2R}(0) \setminus B_R(0)) \cup (B_r(A) \setminus B_{r/2}(A))$.
Given $\epsilon > 0$, we can choose $r$ so small and $R$ so large that
\[
\left(\frac{12 e_\infty}{\sqrt{c}} + 2\right) \left(\mu_\infty(G_{R, r}) \int_{\R^n} |Y_\infty|^2 \, d\mu_\infty\right)^{1/2} \le \epsilon.
\]

Let $K \subseteq \Omega$ be a compact set. Then we can further assume that $K \subseteq B_R(0)$ and 
$K \cap B_r(A) = \emptyset$. With the help of Lemma \ref{lem:F-composed-with-convolution}, we conclude that
\begin{equation} \label{eq:key-estimate}
\begin{split}
c \limsup_{k \to \infty} \int_K |D\psi_\delta - D\phi_{p_k}|^2 \, d\mu_{p_k} & \le \int_\Omega \xi \bigl(F(x, D\psi_\delta) - \beta^2 e_\infty^2\bigr) \, d\mu_\infty + \epsilon \\
& \le (1 - \beta^2) e_\infty^2 + 2\epsilon
\end{split}
\end{equation}
whenever $\delta$ is sufficiently small.
By \cite[Theorem 4.4.2]{Hutchinson:86}, it follows that
\[
c \int_K |D\psi_\delta - Z_\infty|^2 \, d\mu_\infty \le (1 - \beta^2) e_\infty^2 + 2\epsilon.
\]
This holds true for any $\beta \in (0, 1)$ and any $\epsilon > 0$, provided
that $\delta$ is small enough (depending on $\epsilon$). Hence
\begin{equation} \label{eq:local-L2-convergence}
\lim_{\delta \searrow 0} \int_K |D\psi_\delta - Z_\infty|^2 \, d\mu_\infty = 0.
\end{equation}

This $L^2$-convergence implies that there exists a sequence
$\delta_\ell \searrow 0$ such that $D\psi_{\delta_\ell} \to Z_\infty$
almost everywhere in $K$ (with respect to the measure $\mu_\infty$)
as $\ell \to \infty$. Hence
$F(x, D\psi_{\delta_\ell}) \to F(x, Z_\infty)$ almost everywhere in $K$.
Recalling that $F(x, D\psi) \le e_\infty^2$ almost everywhere (with respect
to the Lebesgue measure) and using Lemma \ref{lem:F-composed-with-convolution} again, we conclude that
$F(x, Z_\infty) \le e_\infty^2$ almost everywhere in $K$ (with respect to $\mu_\infty$).

Choosing another compact set $K' \subseteq \Omega$ such that
$\supp \xi \subseteq K'$, we similarly obtain the convergence
$D\psi_\delta \to Z_\infty$ in $L^2(\mu_\infty \restr K'; \R^{m \times n})$.
Hence
\[
\int_\Omega \xi F(x, Z_\infty) \, d\mu_\infty = \lim_{\delta \searrow 0} \int_\Omega \xi F(x, D\psi_\delta) \, d\mu_\infty.
\]
Inequality \eqref{eq:key-estimate}, on the other hand, implies that
\[
\lim_{\delta \searrow 0} \int_\Omega \xi F(x, D\psi_\delta) \, d\mu_\infty \ge \beta^2 e_\infty^2 \int_\Omega \xi \, d\mu_\infty - \epsilon.
\]
Again this holds true for any $\beta \in (0, 1)$ (and provided that $r$ is sufficiently small and $R$ is sufficiently large,
depending on $\epsilon$ but not on $\beta$). Hence
\[
\int_\Omega \xi F(x, Z_\infty) \, d\mu_\infty \ge e_\infty^2 \int_\Omega \xi \, d\mu_\infty - \epsilon.
\]
When we let $r \searrow 0$ and $R \to \infty$ again, then the integrand on the left-hand side converges to $F(x, Z_\infty)$
pointwise in $\Omega$. Since we know that it is bounded by $e_\infty^2$, we can apply Lebesgue's dominated
convergence theorem and obtain
\[
\int_\Omega F(x, Z_\infty) \, d\mu_\infty \ge e_\infty^2 \mu_\infty(\Omega).
\]
As we already know that $F(x, Z_\infty) \le e_\infty^2$ almost everywhere,
this implies statement \ref{itm:pointwise}.

Furthermore, since the functions $D\psi_\delta$ are uniformly bounded
and we now know that $Z_\infty \in L^\infty(\mu_\infty; \R^{m \times n})$,
the local strong convergence \eqref{eq:local-L2-convergence} implies
statement \ref{itm:gradient-of-minimiser}.

For the proof of statement \ref{itm:strong-convergence}, we go back to
\eqref{eq:key-estimate} once more. For a given number $\gamma > 0$, we see
that
\[
\limsup_{k \to \infty} \int_K |D\psi_\delta - D\phi_{p_k}|^2 \, d\mu_{p_k} \le \gamma
\]
for any sufficiently small $\delta > 0$. But for a fixed $\delta$, we also
compute
\[
\begin{split}
\lefteqn{\limsup_{k \to \infty} \int_K |D\psi_\delta - D\phi_{p_k}|^2 \, d\mu_{p_k}} \quad \\
& = \limsup_{k \to \infty} \int_K \bigl(|D\psi_\delta|^2 - 2D\psi_\delta : D\phi_{p_k}  + |D\phi_{p_k}|^2\bigr) \, d\mu_{p_k} \\
& = \int_K \bigl(|D\psi_\delta|^2 - 2D\psi_\delta : Z_\infty\bigr) \, d\mu_\infty + \limsup_{k \to \infty} \int_K |D\phi_{p_k}|^2\, d\mu_{p_k} \\
& = \int_K |D\psi_\delta - Z_\infty|^2 \, d\mu_\infty - \int_K |Z_\infty|^2 \, d\mu_\infty + \limsup_{k \to \infty} \int_K |D\phi_{p_k}|^2\, d\mu_{p_k}.
\end{split}
\]
It follows that
\[
\limsup_{k \to \infty} \int_K |D\phi_{p_k}|^2\, d\mu_{p_k} \le \int_K |Z_\infty|^2 \, d\mu_\infty,
\]
and by \cite[Theorem 4.4.2]{Hutchinson:86}, we have strong $L^2$-convergence
in $K$.
\end{proof}

We can improve the first statement in Proposition \ref{prp:measure-function} if we test
with functions that vanish on $A$.

\begin{corollary} \label{cor:strong-convergence}
Let $G \colon \R^n \times \R^{m \times n} \to \R$ be a continuous function, and suppose that there exists $h \in C_0^0(\R^n; [0, \infty))$ such that
$h(a) = 0$ for all $a \in A$ and $|G(x, M)| \le h|M|^2$ for all $x \in \R^n$ and all $M \in \R^{m \times n}$.
Then
\[
\int_{\R^n} G(x, Z_\infty) \, d\mu_\infty = \lim_{k \to \infty} \int_{\R^n} G(x, D\phi_{p_k}) \, d\mu_{p_k}.
\]
\end{corollary}

\begin{proof}
Let $\epsilon > 0$. There exists $r > 0$ such that $|h| \le c\epsilon$ in $B_r(A)$, which implies that
\[
|G(x, M)| \le \epsilon F(x, M)
\]
for every $x \in B_r(A)$ and $M \in \R^{m \times n}$. Therefore,
\[
\int_{B_r(A)} |G(x, D\phi_{p_k})| \, d\mu_{p_k} \le \epsilon \int_{\R^n} F(x, D\phi_{p_k}) \, d\mu_{p_k} \le e_{p_k}^2 \epsilon
\]
and
\[
\int_{B_r(A)} |G(x, Z_\infty)| \, d\mu_\infty \le \epsilon \int_{\R^n} F(x, Z_\infty) \, d\mu_\infty \le e_\infty^2 \epsilon.
\]
Choose $\xi \in C_0^0(\R^n \setminus A)$ with $\xi \equiv 1$ in $(\supp h) \setminus B_r(A)$.
By Proposition \ref{prp:measure-function}.(i), we have the convergence
\[
\int_{\R^n} \xi G(x, Z_\infty) \, d\mu_\infty = \lim_{k \to \infty} \int_{\R^n} \xi G(x, D\phi_{p_k}) \, d\mu_{p_k}.
\]
It now suffices to combine these facts.
\end{proof}

The following is also important, because it rules out that Proposition \ref{prp:measure-function} is vacuous.

\begin{proposition} \label{prp:non-zero}
The measure $\mu_\infty$ does not vanish.
\end{proposition}

\begin{proof}
Let $r > 0$ such that $B_r(a_1) \cap B_r(a_2) = \emptyset$ for $a_1, a_2 \in A$
with $a_1 \neq a_2$. Choose $\psi \in C_0^1(\R^n; \R^m)$ such that $\psi_i = \phi_i^0$
on $A$ for $i = 1, \dotsc, m_0$ and $D\psi = 0$ in $B_r(a)$ for every $a \in A$. Note that
\[
\begin{split}
e_p^2 & = \int_{\R^n} F(x, D\phi_p) \, d\mu_p \\
& = \frac{1}{2} \int_{\R^n} \grad F(x, D\phi_p) : D\phi_p \, d\mu_p \\
& = \frac{1}{2} \int_{\R^n} \grad F(x, D\phi_p) : D\psi \, d\mu_p
\end{split}
\] 
for every $p \in (n, \infty)$ by the definition of $\mu_p$ and equations
\eqref{eq:homogeneity} and \eqref{eq:Euler-Lagrange}. It then follows from
the local strong convergence in Proposition \ref{prp:measure-function} that
\[
e_\infty^2 = \lim_{k \to \infty} e_{p_k}^2 = \frac{1}{2} \int_{\R^n} \grad F(x, Z_\infty) : D\psi \, d\mu_\infty.
\]
But as the boundary data do not admit a constant function, we have
$e_\infty \neq 0$. Hence $\mu_\infty$ cannot vanish.
\end{proof}

We can finally improve the convergence from Proposition \ref{prp:measure-function} even more.

\begin{proposition} \label{prp:no-point-masses}
If $a \in A$, then $Z_\infty(a) = 0$ or $\mu_\infty(\{a\}) = 0$. Furthermore, for any compact set $K \subseteq \R^n$,
the convergence $(\mu_{p_k}, D\phi_{p_k}) \to (\mu_\infty, Z_\infty)$ is strong in $K$.
\end{proposition}

\begin{proof}
We use \eqref{eq:inner-variation} with $\chi(x) = \eta(|x - a|) (x - a)$
for a function $\eta \in C_0^\infty((0, R))$, where $R > 0$ is so small
that $\chi$ will vanish in a neighbourhood of $A$. We obtain
\begin{equation} \label{eq:inner-variations2}
\begin{split}
0 & = \int_{\R^n} \eta(|x - a|) \left(\grad F(x, D\phi_p) : D\phi_p - \dd{F}{x}(x, D\phi_p) (x - a)\right) \, d\mu_p \\
& \quad + \int_{\R^n} \frac{\eta'(|x - a|)}{|x - a|} \grad F(x, D\phi_p) : \bigl(D\phi_p ((x - a) \otimes (x - a))\bigr) \, d\mu_p \\
& \quad - \frac{2}{p} \int_{\R^n} \eta(|x - a|) F(x, D\phi_p) \left(D(\log V) (x - a) + n\right) \, d\mu_p \\
& \quad - \frac{2}{p} \int_{\R^n} |x - a| \eta'(|x - a|) F(x, D\phi_p) \, d\mu_p.
\end{split}
\end{equation}

Now consider $\eta \in C_0^\infty((-R, R))$ such that $\eta'$ vanishes in
a neighbourhood of $0$, and define
$\chi(x) = \eta(|x - a|)(x - a)$ again. We can find a sequence of functions
$\eta_\ell \in C_0^\infty((0, R))$ such that $\eta_\ell(t) = \eta(t)$ for
$t \ge 1/\ell$ and $|\eta_\ell'| \le C_1\ell$ for some constant $C_1$ independent of $\ell$. Then the functions
$\chi_\ell(x) = \eta_\ell(|x - a|) (x - a)$ converge uniformly to $\chi$,
and their derivatives $D\chi_\ell$ are uniformly bounded and converge to
$D\chi$ at every point in $\R^n \setminus \{a\}$. Using Lebesgue's
dominated convergence theorem, we therefore conclude that
\eqref{eq:inner-variations2} holds in this situation as well.

Therefore,
\begin{multline*}
\int_{\R^n} \eta(|x - a|) F(x, D\phi_p) \, d\mu_p \\
\begin{aligned}
& = \frac{1}{2} \int_{\R^n} \eta(|x - a|) \grad F(x, D\phi_p) : D\phi_p \, d\mu_p \\
& = \frac{1}{2} \int_{\R^n} \eta(|x - a|) \dd{F}{x}(x, D\phi_p) (x - a) \, d\mu_p \\
& \quad - \frac{1}{2} \int_{\R^n} \frac{\eta'(|x - a|)}{|x - a|} \grad F(x, D\phi_p) : \bigl(D\phi_p ((x - a) \otimes (x - a))\bigr) \, d\mu_p \\
& \quad + \frac{1}{p} \int_{\R^n} \eta(|x - a|) F(x, D\phi_p) \left(D(\log V) (x - a) + n\right) \, d\mu_p \\
& \quad + \frac{1}{p} \int_{\R^n} |x - a| \eta'(|x - a|) F(x, D\phi_p) \, d\mu_p.
\end{aligned}
\end{multline*}

We now restrict the identity to $p_k$ and let $k \to \infty$.
Clearly, we have  a constant $C_2$ such that
\[
\frac{1}{p} \int_{\R^n} |\eta(|x - a|)| F(x, D\phi_p) \left|D(\log V) (x - a) + n\right| \, d\mu_p \le C_2 \frac{e_p^2}{p}
\]
and
\[
\frac{1}{p} \int_{\R^n} |x - a| |\eta'(|x - a|)| F(x, D\phi_p) \, d\mu_p \le C_2 \frac{e_p^2}{p},
\]
and the right-hand sides converge to $0$ as $p \to \infty$.
For the remaining terms, we can use Corollary \ref{cor:strong-convergence}. We finally find
the identity
\begin{multline*}
\lim_{k \to \infty} \int_{\R^n} \eta(|x - a|) F(x, D\phi_{p_k}) \, d\mu_{p_k} \\
\begin{aligned}
& = \frac{1}{2} \int_{\R^n} \eta(|x - a|) \dd{F}{x}(x, Z_\infty) (x - a) \, d\mu_\infty \\
& \quad - \frac{1}{2} \int_{\R^n} \frac{\eta'(|x - a|)}{|x - a|} \grad F(x, Z_\infty) : \bigl(Z_\infty ((x - a) \otimes (x - a))\bigr) \, d\mu_\infty.
\end{aligned}
\end{multline*}

Let $r \in (0, R/2]$. If we choose $\eta$ such that $\eta \equiv 1$ in
$[0, r]$ and $\eta \equiv 0$ in $[2r, \infty)$, and such that it satisfies
$|\eta'| \le 2/r$ everywhere, then this inequality gives rise to a constant
$C_3$, independent of $r$, such that
\[
\limsup_{k \to \infty} \int_{B_r(a)} F(x, D\phi_{p_k}) \, d\mu_{p_k} \le C_3r + C_3\mu_\infty(B_{2r}(a) \setminus B_r(0)).
\]
Because
\[
\sum_{\ell = 1}^\infty \mu_\infty(B_{2^{1 - \ell}}(a) \setminus B_{2^{-\ell}}(a)) < \infty,
\]
it is clear that
\[
\liminf_{r \searrow 0} \mu_\infty(B_{2r}(a) \setminus B_r(0)) = 0.
\]
It therefore follows that
\begin{equation} \label{eq:no-concentration}
\lim_{r \searrow 0} \limsup_{k \to \infty} \int_{B_r(a)} F(x, D\phi_{p_k}) \, d\mu_{p_k} = 0.
\end{equation}
By \cite[Theorem 4.4.2]{Hutchinson:86} and Proposition \ref{prp:measure-function},
\[
\begin{split}
\limsup_{k \to \infty} \int_{B_r(a)} F(x, D\phi_{p_k}) \, d\mu_{p_k} & \ge \int_{B_r(a)} F(x, Z_\infty) \, d\mu_\infty \\
& = e_\infty^2 \mu_\infty(B_r(a) \setminus \{a\}) + F(a, Z_\infty(a)) \mu_\infty(\{a\}).
\end{split}
\]
Thus \eqref{eq:no-concentration} implies that $Z_\infty(a) = 0$ or $\mu_\infty(\{a\}) = 0$.

Moreover, combining this information with \eqref{eq:no-concentration} and the statement of Proposition \ref{prp:measure-function}.(i) in a way similar to the proof
of Corollary \ref{cor:strong-convergence}, we obtain
\[
\int_K |Z_\infty|^2 \, d\mu_\infty = \lim_{k \to \infty} \int_K |D\phi_{p_k}|^2 \, d\mu_{p_k},
\]
which is equivalent to strong convergence in $K$ by the results of Hutchinson \cite{Hutchinson:86}.
\end{proof}

\subsection{Currents} \label{sct:currents}

The measure-function pair $(\mu_\infty, Z_\infty)$ constructed in the preceding
subsection gives rise to the $1$-current from Theorem \ref{thm:current}.
Indeed, we define the $\R^m$-valued $1$-current $T_\infty$ such that
\[
T_\infty(\zeta) = \int_{\R^n} \grad F(x, Z_\infty) : \zeta \, d\mu_\infty
\]
for $\zeta \in C_0^\infty(\R^n; \R^{m \times n})$. It then follows from
\eqref{eq:Euler-Lagrange} and Proposition \ref{prp:no-point-masses} that
\[
\partial T_\infty(\xi) = 0
\]
for all $\xi \in C_0^\infty(\R^n; \R^m)$ such that
$\xi_1 = \dotsb = \xi_{m_0} = 0$ on $A$. That is, we know that
$\supp \partial T_\infty \subseteq A$ and $\partial T_{\infty i} = 0$ for
$i = m_0 + 1, \dotsc, m$.

To prove the remaining statements of Theorem \ref{thm:current}, we require
another proposition. This result also reveals a deeper connection
between $\R^m$-valued $1$-currents and the above variational problem.

\begin{proposition} \label{prp:calibration-current}
Suppose that $T \in \ncrt_{m \times n}(\R^n)$ satisfies $\supp \partial T \subseteq A$ and
$\partial T_{m_0 + 1} = \dotsb = \partial T_m = 0$.
Then for any $\phi \in W_*^{1, \infty}(\R^n; \R^m)$ with $E_\infty(\phi) < \infty$, the inequality
\[
\partial T(\phi) \le 2E_\infty(\phi) \M_F(T)
\]
is satisfied. Equality holds if, and only if,
\begin{equation} \label{eq:gradient-along-current}
\lim_{\delta \searrow 0} \int_{\R^n} \biggl|\rho_\delta * D\phi - E_\infty(\phi) \frac{\grad F^*(x, \vec{T})}{\sqrt{F^*(x, \vec{T})}}\biggr|^2 \, d\|T\| = 0.
\end{equation}
\end{proposition}

\begin{proof}
We write $e_0 = E_\infty(\phi)$.
Define $\psi_\delta = \rho_\delta *\phi$. Using \eqref{eq:convexity2}, we
estimate
\[
\begin{split}
\lefteqn{c\int_{\R^n} \biggl|D\psi_\delta - e_0 \frac{\grad F^*(x, \vec{T})}{\sqrt{F^*(x, \vec{T})}}\biggr|^2 \sqrt{F^*(x, \vec{T})} \, d\|T\|} \quad \\
& \le \int_{\R^n} \bigl(F(x, D\psi_\delta) + e_0^2\bigr) \sqrt{F^*(x, \vec{T})} \, d\|T\| - e_0 \int_{\R^n} D\psi_\delta : \vec{T} \, d\|T\| \\
& = \int_{\R^n} \bigl(F(x, D\psi_\delta) + e_0^2\bigr) \sqrt{F^*(x, \vec{T})} \, d\|T\| - e_0 \partial T(\psi_\delta).
\end{split}
\]

Because $\phi \in W^{1, \infty}(\R^n; \R^m)$, there exists a constant $C_1 > 0$ such that
$F(x, D\psi_\delta) \le C_1$ for all $x \in \R^n$ and all $\delta > 0$.
Lemma \ref{lem:F-composed-with-convolution} implies that
\[
\limsup_{\delta \searrow 0} F(x, D\psi_\delta(x)) \le e_0^2
\]
for any $x \in \R^n$.
Applying Fatou's lemma to $C_1 - F(x, D\psi_\delta)$, we find that
\[
\limsup_{\delta \searrow 0} \int_{\R^n} F(x, D\psi_\delta) \sqrt{F^*(x, \vec{T})} \, d\|T\| \le e_0^2 \int_{\R^n} \sqrt{F^*(x, \vec{T})} \, d\|T\| = e_0^2 \M_F(T).
\]

Since $\phi$ is continuous, it is also clear that $\psi_\delta \to \phi$
locally uniformly in $\R^n$. As $\partial T$ is represented by a measure,
this implies that $\partial T(\psi_\delta) \to \partial T(\phi)$. Hence
\[
c \limsup_{\delta \searrow 0} \int_{\R^n} \biggl|D\psi_\delta - e_0 \frac{\grad F^*(x, \vec{T})}{\sqrt{F^*(x, \vec{T})}}\biggr|^2 \sqrt{F^*(x, \vec{T})} \, d\|T\| \le 2 e_0^2 \M_F(T) - e_0 \partial T(\phi).
\]
It follows that $\partial T(\phi) \le 2e_0 \M_F(T)$, and if we have
equality, then \eqref{eq:gradient-along-current} follows as well.

Now suppose that \eqref{eq:gradient-along-current} holds true.
Then
\[
\begin{split}
\partial T(\phi) & = \lim_{\delta \searrow 0} \partial T(\psi_\delta) \\
& = \lim_{\delta \searrow 0} \int_{\R^n} \vec{T} : D\psi_\delta \, d\|T\| \\
& = e_0 \int_{\R^n} \frac{\vec{T} : \grad F^*(x, \vec{T})}{\sqrt{F^*(x, \vec{T})}} \, d\|T\| \\
& = 2e_0 \int_{\R^n} \sqrt{F^*(x, \vec{T})} \, d\|T\| \\
& = 2e_0 \M_F(T).
\end{split}
\]
This completes the proof.
\end{proof}

\begin{proof}[Proof of Theorem \ref{thm:current}]
As mentioned earlier, we consider the current $T_\infty$ defined by the
condition that
\[
T_\infty(\zeta) = \int_{\R^n} \grad F(x, Z_\infty) : \zeta \, d\mu_\infty
\]
for $\zeta \in C_0^\infty(\R^n; \R^{m \times n})$. We will show that $T_\infty$
has the properties stated in Theorem \ref{thm:current}.

It is clear that $\M_F(T_\infty) < \infty$.
We have already seen at the beginning of this subsection that
$\supp \partial T_\infty \subseteq A$ and $\partial T_{\infty i} = 0$ for
$i = m_0 + 1, \dotsc, m$. This makes $\partial T_\infty$ a distribution
supported on a finite set, which means that it is a finite linear combination
of Dirac masses on $A$ and their derivatives \cite[Theorem 1.5.3]{Hoermander:64}.
But because $\|T\|(A) = 0$ by Proposition \ref{prp:no-point-masses}, it is easy to see that we have in fact
just a sum of Dirac masses. It follows that $T_\infty \in \ncrt_{m \times n}(\R^n)$.

By the definition of $T_\infty$, we have
\[
\vec{T}_\infty = \frac{\grad F(x, Z_\infty)}{|\grad F(x, Z_\infty)|}
\]
at $\mu_\infty$-almost every point. It follows that
\[
F^*(x, \vec{T}_\infty) = \frac{F^*(x, \grad F(x, Z_\infty))}{|\grad F(x, Z_\infty)|^2} = \frac{F(x, Z_\infty)}{|\grad F(x, Z_\infty)|^2} = \frac{e_\infty^2}{|\grad F(x, Z_\infty)|^2},
\]
and thus
\[
|\grad F(x, Z_\infty)| = \frac{e_\infty}{\sqrt{F^*(x, \vec{T}_\infty)}}
\]
almost everywhere. Hence
\[
Z_\infty = e_\infty \frac{\grad F^*(x, \vec{T})}{\sqrt{F^*(x, \vec{T})}}
\]
almost everywhere. In view of Proposition \ref{prp:no-point-masses}, the measure
$\|T\|$ is absolutely continuous with respect to $\mu_\infty \restr \Omega$. Moreover,
if $\phi \in W_*^{1, \infty}(\R^n; \R^m)$ is a minimiser of $E_\infty$, then
Proposition \ref{prp:measure-function} gives the convergence $\rho_\delta * D\phi \to Z_\infty$ in
$L^2(\|T\|; \R^{m \times n})$, and that implies
\eqref{eq:gradient-along-current} for $T_\infty$. Proposition \ref{prp:calibration-current}
implies that $\partial T_\infty(\phi) = 2 e_\infty \M_F(T_\infty)$, and thus we have
proved the last statement of Theorem \ref{thm:current}.

To prove the second statement, consider another $\R^m$-valued $1$-current $S \in \ncrt_{m \times n}(\R^n)$
with $\partial S = \partial T_\infty$. Then Proposition \ref{prp:calibration-current} gives
\[
2 e_\infty \M_F(T_\infty) = \partial T_\infty(\phi) = \partial S(\phi) \le 2 e_\infty \M_F(S).
\]
Since our assumptions on $\phi^0$ imply that $e_\infty > 0$, it follows that
$\M_F(T_\infty) \le \M_F(S)$. All the statements of the theorem are now verified.
\end{proof}

\section{Regularisation} \label{sct:regularisation}

We return to the problem of constructing calibrations as in
Section \ref{sct:differential-inequalities}. Therefore, we consider the
domain $\R^2$ again and we study functions $\Phi \colon \R^2 \to \R^2$.
Recall the function
\[
f(M) = \frac{1}{2} \left(|M|^2 - \frac{1}{2} (\tr M)^2 + |m_{12} - m_{21}| \sqrt{|M|^2 - 2\det M}\right)
\]
from Section \ref{sct:differential-inequalities}.
We would like to apply the results from Section \ref{sct:L-infinity} to
\[
F(x, M) = \frac{f(M)}{W(x)}.
\]
Unfortunately, this function does not have the required properties: it is convex,
but not strictly convex in $M$ and is Lipschitz regular at most. Unless
$W$ is bounded, uniformly positive, and of class $C^1$, it fails to
satisfy other assumptions, too. But the potentials $W$ we are most interested
in, will have zeroes, certainly at $a^\pm$ and possibly elsewhere.

For this reason, we need to replace the above function $F$ by regularised
approximations, and we need to show that the relevant properties persist
in the limit. We do this in two steps: first, we focus on the regularisation
of $f$. We can improve some of the properties of $W$ at the same time, but
we still assume that it is uniformly positive. In the second step, we
deal with the zeroes of $W$. Before we embark on this journey, however, we
extend the definition of the $F$-mass from Definition \ref{def:current} as follows.
Suppose that $F \colon \R^2 \times \R^{2 \times 2} \to [0, \infty]$ is a given function that
is convex in the second argument. Assuming that the Legendre transform $F^*$ with respect to the
second argument is Borel measurable on $\R^2 \times \R^{2 \times 2}$, we define
\[
\M_F(T) = \int_{\R^2} \sqrt{F^*(x, \vec{T})} \, d\|T\|
\]
for any $\R^2$-valued $1$-current $T$ on $\R^2$ with locally finite mass.
This is consistent with the previous definition by Proposition \ref{prp:F-mass}.

\subsection{A Korn type inequality}

Our theory will naturally give rise to inequalities such as
$f(D\Phi) \le W$ in $\R^2$ for certain functions
$\Phi \in W_\loc^{1, p}(\R^2; \R^2)$ and for certain exponents
$p \in (1, \infty)$. But because the function $f$ controls only the
trace free part of $D\Phi$, not the full Jacobian matrix, we need to
have a closer look if we want to derive estimates in $W_\loc^{1, p}(\R^2; \R^2)$.
Such estimates are available, and follow in fact quite easily from a
local version of Korn's inequality (as stated and proved, e.g., by
Kondrat\cprime ev and Ole\u\i nik \cite[\S 2, Theorem 8]{Kondratev-Oleinik:88}).
In this subsection, we
formulate the appropriate inequality in the balls $B_R(0)$ and study
how the corresponding constant depends on $R$.

First, however, we write down an observation that explains why Korn's
inequality is useful here. Given $\Phi \in W_\loc^{1, p}(\R^2; \R^2)$,
consider $\Phi^\perp = (\begin{smallmatrix} -\Phi_2 \\ \Phi_1 \end{smallmatrix})$
and its symmetrised derivative
\[
(D\Phi^\perp)_\sym = \begin{pmatrix}
-\dd{\Phi_2}{y_1} & \frac{1}{2} \left(\dd{\Phi_1}{y_1} - \dd{\Phi_2}{y_2}\right) \\ \frac{1}{2} \left(\dd{\Phi_1}{y_1} - \dd{\Phi_2}{y_2}\right) & \dd{\Phi_1}{y_2}
\end{pmatrix}.
\]
Then we note that
\[
|(D\Phi^\perp)_\sym|^2 = |D\Phi|^2 - \frac{1}{2} (\div \Phi)^2 \le 2f(D\Phi).
\]

We can now prove the following lemma.

\begin{lemma} \label{lem:Korn}
For every $p \in (n, \infty)$ there exists a constant $C \ge 0$ such
that the following holds true.
Let $h \colon [1, \infty) \to [0, \infty)$ be a non-decreasing function,
and suppose that $\Phi \in W_\loc^{1, p}(\R^2; \R^2)$ satisfies
\[
\left(\fint_{B_R(0)} \bigl(f(D\Phi)\bigr)^{p/2} \, dy\right)^{1/p} \le h(R)
\]
for every $R \ge 1$. Then there exists $b \in \R$ such that
\[
\left(\fint_{B_R(0)} \left|D\Phi(y) - bI\right|^p \, dy\right)^{1/p} \le C R^{2/p} h(R)
\]
and
\[
\sup_{y \in B_R(0)} |\Phi(y) - \Phi(0) - by| \le CR^{1 + 2/p} h(R)
\]
for every $R \ge 1$.
\end{lemma}

\begin{proof}
Given $R > 0$, we define $\Psi_R(x) = R^{-1} \Phi(Rx)$. Then
\[
\left(\fint_{B_1(0)} \bigl(f(D\Psi_R)\bigr)^{p/2} \, dy\right)^{1/p} \le h(R).
\]
It follows immediately from the local version of Korn's inequality
\cite[\S 2, Theorem 8]{Kondratev-Oleinik:88} that there exists $b_R \in \R$
such that
\[
\left(\fint_{B_1(0)} \left|D\Psi_R(y) - b_R I\right|^p \, dy\right)^{1/p} \le C_1 h(R)
\]
for some constant $C_1$ that depends only on $p$. Morrey's inequality then gives a constant
$C_2 = C_2(p)$ such that
\[
\sup_{y \in B_1(0)} |\Psi_R(y) - \Psi_R(0) - b_Ry| \le C_2 h(R).
\]
In terms of $\Phi$, this means that
\[
\left(\fint_{B_R(0)} \left|D\Phi(y) - b_R I\right|^p \, dy\right)^{1/p} \le C_1 h(R)
\]
and
\[
\sup_{y \in B_R(0)} |\Phi(y) - \Phi(0) - b_Ry| \le C_2 Rh(R).
\]

The first inequality implies in particular that
\[
\left(\fint_{B_1(0)} \left|D\Phi(y) - b_R I\right|^p \, dy\right)^{1/p} \le C_1 R^{2/p} h(R).
\]
But at the same time, we have the inequality
\[
\left(\fint_{B_1(0)} \left|D\Phi(y) - b_1 I\right|^p \, dy\right)^{1/p} \le C_1 h(1).
\]
Hence there exists a constant $C_3$ such that $|b_R - b_1| \le C_3 R^{2/p} h(R)$. Choosing $b = b_1$, we therefore obtain the desired inequalities.
\end{proof}

\subsection{Relaxing the strict convexity}

Suppose now that $W \colon \R^2 \to (0, \infty)$ is a continuous function such that
$W(y) \to \infty$ as $|y| \to \infty$. Then we can clearly
find a sequence of functions $W_k \in C^\infty(\R^2)$, for $k \in \N$, such that
\begin{itemize}
\item
every $W_k$ is bounded,
\item
there exists $c > 0$ such that $W_k \ge c$ in $\R^2$ for every $k \in \N$,
\item
$W_k \le W_\ell$ when $k \le \ell$, and
\item
$W(y) = \lim_{k \to \infty} W_k(y)$ for every $y \in \R^2$.
\end{itemize}

Recall that
\[
f(M) = \frac{1}{2} \left(|M|^2 - \frac{1}{2} (\tr M)^2 + |m_{12} - m_{21}| \sqrt{|M|^2 - 2\det M}\right)^2.
\]
In the coordinates
\begin{align*}
q_1 & = \frac{1}{\sqrt{2}} (m_{11} + m_{22}), & q_2 & = \frac{1}{\sqrt{2}} (m_{11} - m_{22}), \\
q_3 & = \frac{1}{\sqrt{2}} (m_{12} + m_{21}), & q_4 & = \frac{1}{\sqrt{2}} (m_{12} - m_{21}),
\end{align*}
we can write
\[
f(q) = \frac{1}{2}\left(|q_4| + \sqrt{q_2^2 + q_3^2}\right)^2.
\]
We now consider the regularisation
\[
f_k(q) = \frac{1}{2}\left(\sqrt{q_4^2 + \frac{|q|^2}{2k}} + \sqrt{q_2^2 + q_3^2 + \frac{|q|^2}{2k}}\right)^2.
\]
In the original coordinates, this is
\begin{multline*}
f_k(M) = \frac{1}{2} \Biggl(\left(1 + \frac{1}{k}\right) |M|^2 - \frac{1}{2} (\tr M)^2 \\
+ \left((m_{12} - m_{21})^2 + \frac{1}{k} |M|^2\right)^{1/2} \left(\left(1 + \frac{1}{k}\right) |M|^2 - 2 \det M\right)^{1/2}\Biggr).
\end{multline*}
This function is now strictly convex and smooth in
$\R^{2 \times 2} \setminus \{0\}$. Of course, it is still homogeneous of
degree $2$. We further note that $f_k \ge f$ and $f(M) = \lim_{k \to \infty} f_k(M)$
for every $M \in \R^{2 \times 2}$, and this convergence is monotone.
Define
\[
F_k(y, M) = \frac{f_k(y)}{W_k(y)}
\]
These functions then satisfy the assumptions from Section \ref{sct:L-infinity}.

We also need to consider the Legendre transforms.

\begin{lemma}
The Legendre transform of $F$ with respect to the second argument is
\[
F^*(y, N) = \begin{cases}
\frac{1}{4} W(y) \max\{|N|^2 - 2 \det N, (n_{12} - n_{21})^2\} & \text{if $\tr N = 0$}, \\
\infty & \text{else}.
\end{cases}
\]
\end{lemma}

\begin{proof}
We can work in the coordinates $q$ given above, as the transformation amounts
to an isometry between $\R^{2 \times 2}$ and $\R^4$. Moreover, it suffices
to consider
\[
f(q) = \frac{1}{2}\left(|q_4| + \sqrt{q_2^2 + q_3^2}\right)^2
\]
and its Legendre transform
\[
f^*(p) = \sup_{q \in \R^4} \bigl(p \cdot q - f(q)\bigr).
\]
It is clear that $f^*(p) = \infty$ if $p_1 \neq 0$, as $f$ does not depend
on $q_1$. Now assume that $p_1 = 0$. Then the supremum is attained at a point
$q = (0, q_2, q_3, q_4) \in \R^4$ such that either
\begin{itemize}
\item $q_4 = 0$, or
\item $q_2 = q_3 = 0$, or
\item $f$ is differentiable at $q$ and
\begin{align*}
p_2 & = \dd{f}{q_2}(q) = \left(|q_4| + \sqrt{q_2^2 + q_3^2}\right) \frac{q_2}{\sqrt{q_2^2 + q_3^2}}, \\
p_3 & = \dd{f}{q_3}(q) = \left(|q_4| + \sqrt{q_2^2 + q_3^2}\right) \frac{q_3}{\sqrt{q_2^2 + q_3^2}}, \\
p_4 & = \dd{f}{q_4}(q) = \left(|q_4| + \sqrt{q_2^2 + q_3^2}\right) \frac{q_4}{|q_4|}.
\end{align*}
\end{itemize}

In the first case, we find that
\[
f^*(p) = \sup_{q_2, q_3 \in \R} \left(p_2 q_2 + p_3 q_3 - \frac{q_2^2 + q_3^2}{2}\right) = \frac{p_2^2 + p_3^2}{2}.
\]
Similarly, in the second case,
\[
f^*(p) = \sup_{q_4 \in \R} \left(p_4 q_4 - \frac{q_4^2}{2}\right) = \frac{p_4^2}{2}.
\]
In any case, $f^*(p)$ will be at least as large as either of these expressions,
so
\begin{equation} \label{eq:f^*}
f^*(p) \ge \frac{1}{2} \max\{p_2^2 + p_3^2, p_4^2\}
\end{equation}
for every $p \in \R^4$. Finally, in the third of the above cases, we conclude
that $p_2^2 + p_3^2 = p_4^2$. Hence this case occurs only for points on this
double cone.

To summarise, $f^*$ is a convex function that satisfies \eqref{eq:f^*} and
such that $f^*(p) = (p_2^2 + p_3^2)/2$ or $f^*(p) = p_4^2/2$
whenever $p_2^2 + p_3^2 \neq p_4^2$. There exists only one function with
these properties, which is
\[
f^*(p) = \frac{1}{2} \max\{p_2^2 + p_3^2, p_4^2\}.
\]
In terms of the original coordinates, we then have the expression
\[
f^*(N) = \frac{1}{4} \max\{|N|^2 - 2 \det N, (n_{12} - n_{21})^2\},
\]
and the claim follows.
\end{proof}

We do not need to compute the Legendre transforms of $F_k$ explicitly,
but we note that they are convex and homogeneous of degree $2$ in the second argument.
Furthermore,
\[
\begin{split}
F^*(y, N) & = \sup_{M \in \R^{2 \times 2}} \left(M : N - \inf_{k \in \N} F_k(y, M)\right) \\
& = \sup_{k \in \N} \sup_{M \in \R^{2 \times 2}} \bigl(M : N - F_k(y, N)\bigr) \\
& = \sup_{k \in \N} F_k^*(y, N)
\end{split}
\]
for any $y \in \R^2$ and $N \in \R^{2 \times 2}$.

We now want to prove the following.

\begin{proposition} \label{prp:existence-of-calibration-and-current}
Suppose that $W \in C^0(\R^2; (0, \infty))$ satisfies $\lim_{|y| \to \infty} W(y) = \infty$.
Then there exist $\Phi \in \bigcap_{p < \infty} W_\loc^{1, p}(\R^2; \R^2)$ and $T \in \ncrt_{2 \times 2}^0$
such that
\begin{enumerate}
\item $f(D\Phi) \le W$ almost everywhere,
\item $\M_F(T) \le \M_F(S)$ for any $S \in \ncrt_{2 \times 2}^0$, and
\item $\partial T(\Phi) \ge 2\M_F(T)$.
\end{enumerate}
\end{proposition}

\begin{proof}
We define the functions $F_k$ as explained above. For any fixed $k \in \N$, we consider the functional
\[
E_\infty^k(\Phi) = \esssup_{y \in \R^2} \sqrt{F_k(y, D\Phi(y))}.
\]
By Proposition \ref{prp:existence}, there exists a minimiser
$\tilde{\Phi}_k \in W_\loc^{1, \infty}(\R^2; \R^2)$ of $E_\infty^k$
subject to the conditions $\Phi_1(a^-) = 0$ and $\Phi_1(a^+) = 1$. Set
\[
\Phi_k = \frac{\tilde{\Phi}_k}{E_\infty^k(\tilde{\Phi}_k)}.
\]
Then $F_k(y, D\Phi_k) \le 1$, i.e., $f_k(D\Phi_k) \le W_k$, almost everywhere by construction.

Theorem \ref{thm:current} gives rise to a non-trivial current $T_k \in \ncrt_{2 \times 2}(\R^2)$
for every $k \in \N$ with $\supp \partial T_{k 1} \subseteq \{a^\pm\}$ and $\partial T_{k 2} = 0$,
which minimises $\M_{F_k}$ for its boundary data and satisfies
$\partial T_k(\Phi_k) = 2 \M_{F_k}(T_k)$. Because all of these
properties are invariant under multiplication with a positive constant, we can renormalise this
current such that $T_k \in \ncrt_{2 \times 2}^0$.

Next we study the limit as $k \to \infty$.
For any $R > 0$, the functions $W_k \le W$ are uniformly bounded in
$B_R(0)$ by the continuity of $W$. Hence
\[
\sup_{k \in \N} \sup_{y \in B_R(0)} f(D\Phi_k) < \infty.
\]
Lemma \ref{lem:Korn} implies that there
exist $b_k \in \R$ such that the functions
\[
\hat{\Phi}_k(y) = \Phi_k(y) - b_k y
\]
are uniformly bounded in $W^{1, p}(B_R(0); \R^2)$ for all $p < \infty$ and all
$R > 0$. Therefore, we may assume that we have weak convergence of $\hat{\Phi}_k$ in
$W_\loc^{1, p}(\R^2; \R^2)$ for any $p < \infty$ to some limit
$\Phi \colon \R^2 \to \R^2$. Since the set
\[
\set{\Psi \in W^{1, p}(B_R(0); \R^2)}{F(x, D\Psi) \le 1 \text{ almost everywhere}}
\]
is convex and closed in $W^{1, p}(B_R(0); \R^2)$, and every $\hat{\Phi}_k$
belongs to this set, it follows that $F(y, D\Phi) \le 1$ almost everywhere.

Because $\partial T_{k 1}$ is supported on $\{a^\pm\}$, the functions
$\hat{\Phi}_k$ still satisfy the condition
$\partial T_k(\hat{\Phi}_k) = 2 \M_{F_k}(T_k)$.

We can estimate $f_k(M) \le 4|M|^2$ for all $M \in \R^{2 \times 2}$ and all $k \in \N$.
Hence $F_k^*(y, N) \ge \frac{1}{16} W(y) |N|^2$ for all $y \in \R^2$
and $N \in \R^{2 \times 2}$. Since $W$ is bounded below under the above assumptions,
it follows that
\[
\|T_k\|(\R^2) \le C\int_{\R^2} \sqrt{F_k^*(y, \vec{T}_k)} \, d\|T_k\| = C\M_{F_k}(T_k) = \frac{C}{2} \partial T_k(\hat{\Phi}_k)
\]
for some constant $C > 0$ that is independent of $k$, and the right-hand side is obviously
bounded. Hence we may assume that $T_k$ converges
weakly* in the dual space of $C_0^0(\R^2; \R^{2 \times 2})$ to some
limit $T$, which will automatically belong to $\ncrt_{2 \times 2}^0$.
From the definition of the $F$-mass in Definition \ref{def:current}, it follows easily that $\M_{F_\ell}$
is lower semicontinuous with respect to such convergence for any
fixed $\ell \in \N$. Thus
\[
\M_{F_\ell}(T) \le \liminf_{k \to \infty} \M_{F_\ell}(T_k) \le \liminf_{k \to \infty} \M_{F_k}(T_k).
\]
Moreover, Beppo Levi's monotone convergence theorem gives
\[
\M_F(T) = \int_{\R^2} \sqrt{F^*(y, \vec{T})} \, d\|T\| = \lim_{\ell \to \infty} \int_{\R^2} \sqrt{F_\ell^*(y, \vec{T})} \, d\|T\| = \lim_{\ell \to \infty} \M_{F_\ell}(T).
\]
Therefore,
\[
\M_F(T) \le \liminf_{k \to \infty} \M_{F_k}(T_k).
\]

Recall that the currents $T_k$ all have the same boundary by construction. Since $\hat{\Phi}_k \to \Phi$ locally uniformly, it follows that
\[
\partial T(\Phi) = \lim_{k \to \infty} \partial T_k(\hat{\Phi}_k) = 2\lim_{k \to \infty} \M_{F_k}(T_k) \ge 2\M_F(T).
\]

It remains to prove that $T$ minimises the $F$-mass in $\ncrt_{2 \times 2}^0$.
Let $S \in \ncrt_{2 \times 2}^0$. Then $\partial S = \partial T_k$ for every $k \in \N$, and we
know that $\M_{F_k}(T_k) \le \M_{F_k}(S)$. As above, we see that
\[
\M_F(T) \le \liminf_{k \to \infty} \M_{F_k}(T_k) \le \liminf_{k \to \infty} \M_{F_k}(S) = \M_F(S).
\]
This finally concludes the proof.
\end{proof}

\subsection{Potentials with zeroes}

We now want to remove the assumption that $W$ is positive. While we do not obtain
a specific current with the properties of Proposition \ref{prp:existence-of-calibration-and-current}
in this case, we can still prove the following.

\begin{theorem} \label{thm:existence-of-calibration}
Let $W \colon \R^2 \to [0, \infty)$ be a continuous function. Then there exists
$\Phi \in \bigcap_{p < \infty} W_\loc^{1, p}(\R^2; \R^2)$ such that $f(D\Phi) \le W$ almost everywhere
and
\[
\Phi_1(a^+) - \Phi_1(a^-) \ge 2 \inf_{T \in \ncrt_{2 \times 2}^0} M_F(T).
\]
\end{theorem}

\begin{proof}
For $k \in \N$, define $W_k(y) = W(y) + \frac{1}{k}(1 + |y|^2)$, and then let
$F_k(y, M) = f(M)/W_k(y)$. Then $\M_{F_k} \ge  \M_F$.

For each $k \in \N$, Proposition \ref{prp:existence-of-calibration-and-current}
provides a function $\Phi_k \colon \R^2 \to \R^2$ such that $f(D\Phi_k) \le W_k$ almost everywhere,
and it also provides a current $T_k \in \ncrt_{2 \times 2}^0$
that minimises $\M_{F_k}$.
Furthermore, Proposition \ref{prp:existence-of-calibration-and-current} tells us that
\[
\Phi_{k1}(a^+) - \Phi_{k1}(a^-) = \partial T_k(\Phi_k) \ge 2 \M_{F_k}(T_k) \ge 2\M_F(T_k) \ge 2 \inf_{T \in \ncrt_{2 \times 2}^0} \M_F(T).
\]

When we let $k \to \infty$, we see with the same arguments as in the proof of
Proposition \ref{prp:existence-of-calibration-and-current} that we can modify each $\Phi_k$
such that it still has the above properties, but such that we have
weak convergence of some subsequence of $(\Phi_k)_{k \in \N}$,
in the space $\bigcap_{p < \infty} W_\loc^{1, p}(\R^2; \R^2)$,
to a limit $\Phi$ that satisfies $f(D\Phi) \le W$ almost everywhere.
Since this also implies locally uniform convergence, it further follows that
\[
\Phi_1(a^+) - \Phi_1(a^-) \ge 2 \inf_{T \in \ncrt_{2 \times 2}^0} M_F(T).
\]
This concludes the proof.
\end{proof}

\subsection{How calibrations give a lower bound}

We expect that calibrations give rise to lower bounds for the energy, and this is indeed the reason why
we consider them. Formal calculations give a good idea of the underlying estimates, but in order to
obtain a rigorous proof, we need some control of the corresponding integrals when $u$ is potentially unbounded.
The purpose of this subsection is to justify the following statement, which
depends on Proposition \ref{prp:calibration}.

Once this is proved, we can proceed to prove Theorem \ref{thm:main} and Corollary \ref{cor:1D}, which we do at the end of the section.

\begin{lemma} \label{lem:calibration-implies-inequality}
Suppose that $W \colon \R^2 \to [0, \infty)$ is H\"older continuous and satisfies the growth condition \eqref{eq:growth-W}.
Suppose that $\Phi \in \bigcap_{p < \infty} W_\loc^{1, p}(\R^2; \R^2)$
satisfies $f(D\Phi) \le W$ almost everywhere. Then
\[
\mathcal{E}(a^-, a^+) \ge \Phi_1(a^-) - \Phi_1(a^+).
\]
\end{lemma}

\begin{proof}
For any two constants $b, c \in \R$, the function $\Psi(y) = \Phi(y) + by + c$ satisfies
$f(D\Psi) = f(D\Phi)$ and
\[
\Psi_1(a^-) - \Psi_1(a^+) = \Phi_1(a^-) - \Phi_1(a^+).
\]
Hence we may assume without loss of generality that $\Phi_2(a^+) = \Phi_2(a^-)$ and $\Phi(0) = 0$.
Let $\alpha = - \frac{1}{2} \div \Phi$. Then the inequality $f(D\Phi) \le W$
is equivalent to $g(D\Phi + \alpha I) \le W$ for the function $g$ from Section \ref{sct:differential-inequalities}.

We regularise the calibration $\Phi$ and the potential function $W$ the
same way as in Proposition \ref{prp:calibration}. That is, we define $\Phi_\delta = \rho_\delta * \Phi$ and
$\alpha_\delta = \rho_\delta * \alpha + \delta$, and furthermore,
\[
W_\delta = \frac{\rho_\delta * W}{1 - \delta} + \delta.
\]
According to Proposition \ref{prp:calibration}, there exist vector fields
$\omega_\delta \in C^1(\R^2; \R^2)$ such that
\begin{equation} \label{eq:divergence-inequality}
\div \Phi_\delta(u)  + \alpha_\delta(u) \div u \le \frac{\epsilon}{2} |Du|^2 + \frac{1}{2\epsilon} W_\delta(u) - \epsilon \div \bigl((Du)^\transpose \omega_\delta(u)\bigr)^\perp
\end{equation}
for all $u \in C^2(B_1(0); \R^2)$ and all $\epsilon > 0$, and such that
\[
|\omega_\delta(y)| \le C_1(1 + |y| \log |y|)
\]
for every $y \in \R^2$, where $C_1$ is a constant independent of $\delta$
or $y$.

Inequality \eqref{eq:divergence-inequality}, in its weak form, says that
\begin{multline} \label{eq:divergence-inequality-weak}
\int_{B_1(0)} \bigl(\eta \alpha_\delta(u) \div u - \nabla \eta \cdot \Phi_\delta(u)\bigr) \, dx \le \int_{B_1(0)} \eta \left(\frac{\epsilon}{2} |Du|^2 + \frac{1}{2\epsilon} W_\delta(u)\right) \, dx \\
- \epsilon \int_{B_1(0)} \omega_\delta(u) \cdot (Du \nabla^\perp \eta) \, dx
\end{multline}
for all $u \in C^2(B_1(0); \R^2)$ and all $\eta \in C_0^\infty(B_1(0))$
with $\eta \ge 0$.

Fix $q > 2$. Using Lemma \ref{lem:Korn}, we find a constant $C_2$ (depending on $q$) such that
\begin{equation} \label{eq:L^q-estimate-Phi}
\left(\fint_{B_R(0)} |D\Phi|^q \, dy\right)^{1/q} \le C_2 R^{\bar{p} + 2/q}
\end{equation}
for every $R \ge 1$ and
\begin{equation} \label{eq:pointwise-estimate-Phi}
|\Phi(y)| \le C_2\bigl(|y|^{\bar{p} + 1+ 2/q} + 1\bigr)
\end{equation}
for every $y \in \R^2$. With the help of H\"older's inequality, we then also estimate
\begin{equation} \label{eq:estimate-alpha_delta}
\begin{split}
|\alpha_\delta(y)| & = \left|\delta + \int_{B_\delta(y)} \rho_\delta(y - z) \alpha(z) \, dz\right| \\
& \le \delta + \frac{1}{2} \|\rho_\delta\|_{L^{q/(q - 1)}(\R^2)} \left(\int_{B_{|y| + 1}(0)} |\div \Phi|^q \, dz\right)^{1/q} \\
& \le C_3 \delta^{-2/q} \bigl(|y|^{\bar{p} + 4/q} + 1\bigr),
\end{split}
\end{equation}
for some constant $C_3$, whenever $\delta \in (0, 1]$.

Combining \eqref{eq:L^q-estimate-Phi} with Morrey's inequality, we find a constant $C_4$ such that
\[
|\Phi(y) - \Phi(z)| \le C_4 R^{\bar{p} + 4/q} |y - z|^{1 - 2/q}
\]
for all $y, z \in B_R(0)$. Hence there exists a constant $C_5$ such that
\begin{equation} \label{eq:convergence-rate-Phi_delta}
|\Phi_\delta(y) - \Phi(y)| = \left|\int_{B_\delta(y)} \rho_\delta(y - z) (\Phi(z) - \Phi(y)) \, dz\right| \le C_5 R^{\bar{p} + 4/q} \delta^{1 - 2/q}
\end{equation}
when $y \in B_R(0)$ with $R \ge 1$ and $\delta \le 1$.

Since the functions $\Phi_\delta$ and $\alpha_\delta$ have at most polynomial growth by these
estimates, and we know that the same applies to $W_\delta$ and $\omega_\delta$, a standard
approximation argument now shows that \eqref{eq:divergence-inequality-weak}
holds for all $u \in W^{1, 2}(B_1(0); \R^2)$.

Recall that in the definition of $\mathcal{E}(a^-, a+)$, we consider
$u_0 \colon \R^2 \to \R^2$, defined by
\[
u_0(x) = \begin{cases}
a^+ & \text{if $x_1 > 0$}, \\
a^- & \text{if $x_1 < 0$}.
\end{cases}
\]
The set $\mathcal{U}(a^-, a^+)$ then comprises all families
$(u_\epsilon)_{\epsilon > 0}$ in $W^{1, 2}(B_1(0); \R^2)$ such that
$u_\epsilon \to u_0$ in $L^1(B_1(0); R^2)$ and such that there exist $\tau > 0$ and
$s > 2$ with $\epsilon^{-\tau} \div u_\epsilon \to 0$ in
$L^s(B_1(0))$ as $\epsilon \searrow 0$. Then
\[
\mathcal{E}(a^-, a^+) = \frac{1}{2} \inf\set{\liminf_{\epsilon \searrow 0} E_\epsilon(u_\epsilon; B_1(0))}{(u_\epsilon)_{\epsilon > 0} \in \mathcal{U}(a^-, a^+)}.
\]
We now fix $(u_\epsilon)_{\epsilon > 0}$ from $\mathcal{U}(a^-, a^+)$.
Choose a sequence $\epsilon_k \searrow 0$ such that
\[
\lim_{k \to \infty} E_{\epsilon_k}(u_{\epsilon_k}; B_1(0)) = \liminf_{\epsilon \searrow 0} E_\epsilon(u_\epsilon; B_1(0)).
\]
We may assume that this limit is finite.

By the growth condition \eqref{eq:growth-W}, there exist constants $c, \theta > 0$ such that $W(y) \ge c|y|^{2\bar{p}}$
when $|y| \ge \theta$. For $u \in W^{1, 2}(B_1(0); \R^2)$, we consider $v = \max\{|u|^{\bar{p} + 1}, \theta\}$. Then we
note that
\[
\begin{split}
\int_{B_1(0)} |Dv| \, dx & \le (\bar{p} + 1) \int_{\{|u| > \theta\}} |u|^{\bar{p}} |Du| \, dx \\
& \le (\bar{p} + 1) \int_{\{|u| > \theta\}} \left(\frac{\epsilon}{2} |Du|^2 + \frac{1}{2\epsilon} |u|^{2\bar{p}}\right) \, dx
\end{split}
\]
by Young's inequality. Hence
the functions $\max\{|u_{\epsilon_k}|^{\bar{p} + 1}, \theta\}$ are uniformly bounded in
$W^{1, 1}(B_1(0))$. By the Sobolev inequality, they are also bounded in
$L^2(B_1(0))$. Thus $(u_{\epsilon_k})_{k \in \N}$ is bounded in
$L^{2\bar{p} + 2}(B_1(0); \R^2)$. Since it converges to $u_0$ in $L^1(B_1(0); \R^2)$,
we conclude that this convergence holds in $L^r(B_1(0); \R^2)$ as well for any $r < 2\bar{p} + 2$.

We now fix a number $\ell > 1$ and define $\delta_k = \epsilon_k^\ell$. Using \eqref{eq:pointwise-estimate-Phi}
and \eqref{eq:convergence-rate-Phi_delta}, we see that
\[
\int_{B_1(0)} \nabla \eta \cdot \Phi_{\delta_k}(u_{\epsilon_k}) \, dx \to \int_{B_1(0)} \nabla \eta \cdot \Phi(u_0) \, dx
\]
if $q$ is chosen sufficiently large. Recall that there exist $\tau > 0$ and
$s > 2$ with $\epsilon^{-\tau} \div u_\epsilon \to 0$ in
$L^s(B_1(0))$ as $\epsilon \searrow 0$. Choose
\[
q > \frac{4s}{\bar{p}(s - 2) + 2s - 2}.
\]
Then \eqref{eq:estimate-alpha_delta} implies that
\[
\|\alpha_{\delta_k}(u_{\epsilon_k})\|_{L^{s/(s - 1)}(B_1(0))} \le C_4 \epsilon_k^{-2\ell/q}
\]
for a constant $C_4$ that is independent of $k$. It follows that
\[
\int_{B_1(0)} \eta \alpha_{\delta_k}(u_{\epsilon_k}) \div u_{\epsilon_k} \, dx \to 0
\]
if we also choose $q > 2\ell/\tau$. It is not difficult to see that
\[
\epsilon_k \int_{B_1(0)} \omega_{\delta_k}(u_{\epsilon_k}) \cdot Du_{\epsilon_k} \nabla^\perp \eta \, dx \to 0.
\]

Because we assume that $W$ is locally H\"older continuous, and because we have the
growth condition \eqref{eq:growth-W}, we have a number $\gamma > 0$ and a constant $C_4$ such that
$W_\delta(y) \le W(y) + C_4\delta^\gamma(1 + |y|^{2\bar{p}})$ for any $y \in \R^2$. If we choose $\ell > 1/\gamma$,
then it follows that
\[
\limsup_{k \to \infty} \int_{B_1(0)} \eta \left(\frac{\epsilon_k}{2} |Du_{\epsilon_k}|^2 + \frac{1}{2\epsilon_k}W_{\delta_k}(u_{\epsilon_k})\right) \, dx \le \lim_{k \to \infty} E_{\epsilon_k}(u_{\epsilon_k}; B_1(0)).
\]
Combining all the inequalities, we find that
\[
\mathcal{E}(a^-, a^+) \ge -\frac{1}{2} \int_{B_1(0)} \nabla \eta \cdot \Phi(u_0) \, dx = -\frac{1}{2} \int_{B_1(0)} \dd{\eta}{x_1} \Phi_1(u_0) \, dx.
\]
The integral on the right-hand side is easy to calculate because of the
specific form of $u_0$: we conclude that
\[
\mathcal{E}(a^-, a^+) \ge \frac{1}{2} \bigl(\Phi_1(a^+) - \Phi_1(a^-)\bigr) \int_{-1}^1 \eta(0, x_2) \, dx_2.
\]
If we approximate the characteristic function of $B_1(0)$ with $\eta$,
we therefore obtain the desired inequality.
\end{proof}

We now have all the ingredients for the proof of our main result.

\begin{proof}[Proof of Theorem \ref{thm:main}]
The functional $\M_F$ defined in the introduction is identical to the
$F$-mass defined in Section \ref{sct:regularisation}. Under the assumptions of
Theorem \ref{thm:main}, we can use Theorem \ref{thm:existence-of-calibration} to obtain
a suitable calibration. Lemma \ref{lem:calibration-implies-inequality} then yields the desired inequality.
\end{proof}

\begin{proof}[Proof of Corollary \ref{cor:1D}]
We compute
\[
\M_F(T^0) = \frac{1}{2} \int_{[a^-, a^+]} \sqrt{W} \, d\Ha^1.
\]
If $T^0$ minimises $\M_F$ in $\ncrt_{2 \times 2}^0$, then Theorem \ref{thm:main} therefore
implies that
\[
\mathcal{E}(a^-, a^+) \ge \int_{[a^-, a^+]} \sqrt{W} \, d\Ha^1.
\]
The reverse inequality follows from a standard construction, which can be found, e.g., in
a paper by Ignat and Monteil \cite[Proposition 4.1]{Ignat-Monteil:20}.
\end{proof}

\section{The geometric problem} \label{sct:geometric}

Theorem \ref{thm:existence-of-calibration} suggests that we study the minimisers of
\[
\M_F(T) = \int_{\R^2} \sqrt{F^*(x, \vec{T})} \, d\|T\|
\]
for $T \in \ncrt_{2 \times 2}^0$. This now constitutes a geometric problem, which
is similar in spirit to the problem of finding geodesics. But it is also a
novel problem, because we have to consider \emph{vector-valued} currents,
the components of which interact in non-trivial ways with each other.
This is the problem that we analyse in this section.

First recall that $F^*(x, N) = W(x) f^*(N)$, where
\[
f^*(N) = \begin{cases}
\frac{1}{4} \max\{|N|^2 - 2 \det N, (n_{12} - n_{21})^2\} & \text{if $\tr N = 0$}, \\
\infty & \text{else}.
\end{cases}
\]
As in the introduction, we consider the current $T^0$, defined by
\[
T^0(\zeta) = \int_{[a^-, a^+]} \zeta : \begin{pmatrix} 0 & 1 \\ 0 & 0 \end{pmatrix} \, d\Ha^1
\]
for $\zeta \in C_0^\infty(\R^2; \R^{2 \times 2})$. Above all, we are interested
in conditions that guarantee that $T^0$ minimises the functional.

\subsection{An estimate for $F^*$}

We will estimate $\M_F(T)$ in terms of the first component $T_1$ of $T$.
Assuming that $\M_F(T) < \infty$, we first observe that $\vec{T}$ must be
of the form
\[
\vec{T} = \begin{pmatrix} r & s \\ t & -r \end{pmatrix}
\]
for some $r, s, t \in \R$ almost everywhere. The numbers $r$ and $s$ will
effectively be determined by $T_1$, but this does not apply to $t$.
It turns out, however, that we can estimate
$F^*(x, (\begin{smallmatrix} r & s \\ t & -r \end{smallmatrix}))$ in terms
of the following functions: for $\lambda \in [-1, 1]$, we define
\[
\Theta_\lambda(z) = \begin{cases}
(1 + \lambda) \frac{z_1^2}{|z_2|} + |z_2| & \text{if $z_2 > 0$ and $(1 + \lambda) z_1^2 < (1 - \lambda) z_2^2$}, \\
(1 - \lambda) \frac{z_1^2}{|z_2|} + |z_2| & \text{if $z_2 < 0$ and $(1 - \lambda) z_1^2 < (1 + \lambda) z_2^2$}, \\
2|z_1| \sqrt{1 - \lambda^2} + \lambda z_2 & \text{else}.
\end{cases}
\]
We note that $\Theta_\lambda$ is positive homogeneous of degree $1$ in $z$
and that $\Theta_\lambda(z) \ge |z_2|$ for all $\lambda \in [-1, 1]$
and all $z \in \R^2$. Furthermore, we have the following inequality.

\begin{lemma} \label{lem:f^*-estimate}
For any $\lambda \in [- 1, 1]$ and any $r, s, t \in R$,
\[
\sqrt{f^*\begin{pmatrix} r & s \\ t & - r \end{pmatrix}} \ge \frac{1}{2} \bigl(\Theta_\lambda(r, s) + \lambda t\bigr).
\]
\end{lemma}

\begin{proof}
We fix $r$ and $s$ and regard the left-hand side of the desired inequality
as a function of $t$. Thus we define
\[
\theta(t) = \sqrt{f^*\begin{pmatrix} r & s \\ t & - r \end{pmatrix}} = \frac{1}{2} \max\{\sqrt{4r^2 + (s + t)^2}, |s - t|\}.
\]
We also consider $\theta_1(t) = \frac{1}{2} \sqrt{4r^2 + (s + t)^2}$ and
$\theta_2(t) = \frac{1}{2} |s - t|$. These are convex functions, and hence
$\theta = \max\{\theta_1, \theta_2\}$ is convex, too. Moreover, we see that
$\theta$ is differentiable at every point with the exception of $t = -r^2/s$
(which is the unique point where $\theta_1(t) = \theta_2(t)$) if $s \neq 0$.

If $s > 0$ and $t > -r^2/s$ or if $s < 0$ and $t < -r^2/s$, then
$\theta_1(t) > \theta_2(t)$, and we compute
\[
\theta'(t) = \theta_1'(t) = \frac{s + t}{2\sqrt{4r^2 + (s + t)^2}}.
\]
If $s > 0$ and $t < -r^2/s$ or if $s < 0$ and $t > -r^2/s$, then
$\theta_1(t) < \theta_2(t)$, and
\[
\theta'(t) = \theta_2'(t) = \frac{t - s}{2|t - s|}.
\]
(If $s = 0$, then $\theta = \theta_1$.)

Now fix $\lambda \in [-1, 1]$. If $-1 < \lambda < 1$, then there exists
a unique point $t_\lambda \in \R$ such that $\lambda/2$ is a subderivative
of $\theta$ at $t_\lambda$. That point is $t_\lambda = -r^2/s$ if
\[
s > 0 \quad \text{and} \quad (1 + \lambda)r^2 \le (1 - \lambda) s^2
\]
or if
\[
s < 0 \quad \text{and} \quad (1 - \lambda) r^2 \le (1 + \lambda) s^2.
\]
Otherwise, it is the unique point where $\theta_1'(t) = \lambda/2$, namely
\[
t_\lambda = - s + \frac{2\lambda |r|}{\sqrt{1 - \lambda^2}}.
\]

We now have the inequality
\[
\theta(t) \ge \theta(t_\lambda) + \frac{\lambda}{2}(t - t_\lambda) = \theta_1(t_\lambda) + \frac{\lambda}{2}(t - t_\lambda).
\]
If we compute the right-hand side, we obtain exactly $\frac{1}{2}(\Theta_\lambda(r, s) + \lambda t)$.

For $\lambda = -1$ and $\lambda = 1$, the inequality now follows by continuity.
\end{proof}

For our subsequent estimates, it will be useful to know more about the structure of
$\Theta_\lambda$. This turns out to be a convex function; in fact, the following is true.

\begin{lemma} \label{lem:convex-envelope}
For $-1 \le \lambda \le 1$, let
\[
H_\lambda(z) = \begin{cases}
(1 + \lambda) \frac{z_1^2}{z_2} + z_2 & \text{if $z_2 > 0$}, \\
-(1 - \lambda) \frac{z_1^2}{z_2} - z_2 & \text{if $z_2 < 0$}, \\
\infty & \text{if $z_1 \neq 0$ and $z_2 = 0$}, \\
0 & \text{if $z = 0$}.
\end{cases}
\]
Then $\Theta_\lambda$ is the convex envelope of $H_\lambda$.
\end{lemma}

\begin{proof}
Suppose first that $-1 < \lambda < 1$. Consider the sets
\begin{align*}
C_+ & = \set{z \in \R^2}{z_2 > 0 \text{ and }(1 + \lambda) z_1^2 < (1 - \lambda) z_2^2}, \\
C_- & = \set{z \in \R^2}{z_2 < 0 \text{ and } (1 - \lambda) z_1^2 < (1 + \lambda) z_2^2},
\end{align*}
and $D = \R^2 \setminus (C_+ \cup C_-)$.

If $z_2 > 0$, then we observe that
\[
0 \le (1 + \lambda) z_2 \left(\frac{|z_1|}{z_2} - \sqrt{\frac{1 - \lambda}{1 + \lambda}}\right)^2 = (1 + \lambda) \frac{z_1^2}{z_2} - 2|z_1| \sqrt{1 - \lambda^2} + (1 - \lambda)z_2.
\]
From this, we conclude that
$H_\lambda(z) \ge 2|z_1|\sqrt{1 - \lambda^2} + \lambda z_2$ when $z_2 > 0$,
with equality on $\partial C_+$.
Similarly, we show that
$H_\lambda(z) \ge 2|z_1|\sqrt{1 - \lambda^2} + \lambda z_2$ when $z_2 < 0$,
with equality on $\partial C_-$.

Let
\[
L = \set{\ell \colon \R^2 \to \R}{\ell \text{ is linear with } \ell \le H_\lambda \text{ in } \R^2},
\]
and let $\check{H}_\lambda(z) = \sup_{\ell \in L} \ell(z)$
denote the convex envelope of $H_\lambda$. (Note that it suffices to consider
linear rather than affine functions, because $H_\lambda$ is positive homogeneous of degree $1$.)
Then the above observations imply that $\phi_+(z) = 2z_1 \sqrt{1 - \lambda^2} + \lambda z_2$ and $\phi_-(z) = -2z_1 \sqrt{1 - \lambda^2} + \lambda z_2$
belong to $L$. It follows that $\check{H}_\lambda(z) \ge 2|z_1|\sqrt{1 - \lambda^2} + \lambda z_2$ for all $z \in \R^2$.
Since $H_\lambda(z) = \phi_+(z)$ when $z \in \partial C_{\pm}$ and $z_1 \ge 0$,
it also follows that $\check{H}_\lambda \le \phi_+$ in $\set{z \in D}{z_1 \ge 0}$ (the convex hull of
$(\partial C_- \cup \partial C_+) \cap \{z_1 \ge 0\}$). Similarly,
$\check{H}_\lambda \le \phi_-$ in $\set{z \in D}{z_1 \le 0}$.
Combining these inequalities, we see that $\check{H}_\lambda = \Theta_\lambda$ in $D$.

The restriction of $H_\lambda$ to $C_+$ is smooth. It suffices to examine the
Hessian to see that it is also convex. Thus for any $z_0 \in C_+$, there
exists a linear function $\ell_0 \colon \R^2 \to \R$ such that
$\ell_0(z_0) = H_\lambda(z_0)$ and $\ell_0(z) \le H_\lambda(z)$ for all
$z \in C_+$. Indeed, differentiating $H_\lambda$, we see that
\[
\ell_0(z) = 2(1 + \lambda) \frac{z_{01}}{z_{02}} z_1 + \left(1 - (1 + \lambda) \frac{z_{01}^2}{z_{02}^2}\right) z_2
\]
for $z \in \R^2$, where $z_0 = (z_{01}, z_{02})$. We claim that $\ell_0 \in L$. To see why,
we note that
\[
1 - (1 + \lambda) \frac{z_{01}^2}{z_{02}^2} \ge \lambda
\]
because $z_0 \in C_+$. For $\tilde{z} \in \partial C_+$, we already know that
\[
\ell_0(\tilde{z}) \le H_\lambda(\tilde{z}) = 2|\tilde{z}_1| \sqrt{1 - \lambda^2} + \lambda \tilde{z}_2.
\]
For $z \in D \cup C_-$, choose $\tilde{z} \in \partial C_+$ with $\tilde{z}_1 = z_1$. Then
$z_2 \le \tilde{z}_2$, and therefore,
\[
\ell_0(z) \le \ell_0(\tilde{z}) + \lambda(z_2 - \tilde{z}_2)  \le 2|z_1| \sqrt{1 - \lambda^2} + \lambda z_2 \le H_\lambda(z).
\]
Hence $\ell_0 \in L$.

We conclude that $\check{H}_\lambda(z_0) = H_\lambda(z_0) = \Theta_\lambda(z_0)$.
Similar arguments apply to $C_-$ as well. Hence
$\check{H}_\lambda = \Theta_\lambda$ everywhere.

It remains to study the cases $\lambda = 1$ and $\lambda = - 1$. In both
cases, we have the identity $\Theta_\lambda(z) = |z_2|$. Furthermore, in both
cases, we compute $H_\lambda(0, z_2) = |z_2|$ and $H_\lambda(z) \ge |z_2|$
for all $z \in \R^2$. As $\liminf_{s \to \infty} H_\lambda(z_1, s) = 0$ for any
$z_1 \in \R$, it is clear that $\Theta_\lambda$ is the convex envelope.
\end{proof}

The above information allows us to prove the following.

\begin{lemma} \label{lem:subdifferential}
Suppose that $\kappa, \iota, \lambda \in [-1, 1]$ are three numbers such
that
\[
\iota^2 \le \min\{1 - \lambda^2, (1 + \kappa)(1 - \lambda), (1 - \kappa)(1 + \lambda)\}.
\]
Then
\[
\Theta_\lambda(z) \ge 2\iota z_1 + \kappa z_2
\]
for all $z \in \R^2$.
\end{lemma}

\begin{proof}
Suppose first that $-1 < \lambda < 1$. Since $\Theta_\lambda$ is positive $1$-homogeneous,
the convexity implied by Lemma \ref{lem:convex-envelope} means that for any $z_0 \in \R^2$, if $\Theta_\lambda$ is
differentiable at $z_0$, then
\[
\Theta_\lambda(z) \ge D\Theta_\lambda(z_0)z
\]
for every $z \in \R^2$. If $r, s \in \R$ such that $s> 0$ and
$(1 + \lambda) r^2 < (1 - \lambda) s^2$, then we can differentiate $\Theta_\lambda$ at $z_0 = (r, s)$. We conclude that
\[
\Theta_\lambda(z) \ge 2(1 + \lambda) \frac{r}{s} z_1 + \left(1 - (1 + \lambda)\frac{r^2}{s^2}\right)z_2.
\]
By continuity, the inequality still holds true when
$(1 + \lambda) r^2 \le (1 - \lambda) s^2$.

Given a number $\iota \in [-1, 1]$ such that $\lambda^2 + \iota^2 \le 1$,
we can set $s = 1 + \lambda$ and $r = \iota$. Then
\[
\left|\frac{r}{s}\right| = \frac{|\iota|}{1 + \lambda} \le \sqrt{\frac{1 - \lambda}{1 + \lambda}},
\]
and the inequality applies. Thus
\begin{equation} \label{eq:Theta1}
\Theta_\lambda(z) \ge 2\iota z_1 + \left(1 - \frac{\iota^2}{1 + \lambda}\right)z_2.
\end{equation}

Similarly, if $s <0$ and $(1 - \lambda) r^2 \le (1 + \lambda) s^2$, then
\[
\Theta_\lambda(z) \ge -2(1 - \lambda) \frac{r}{s} z_1 + \left((1 - \lambda)\frac{r^2}{s^2} - 1\right)z_2.
\]
If $\lambda^2 + \iota^2 \le 1$, we consider $s = -(1 - \lambda)$ and
$r = \iota$. Thus we derive the inequality
\begin{equation} \label{eq:Theta2}
\Theta_\lambda(z) \ge 2\iota z_1 + \left(\frac{\iota^2}{1 - \lambda} - 1\right)z_2.
\end{equation}

Finally, as we always have $\Theta_\lambda(z) \ge 2|z_1| \sqrt{1 - \lambda^2} + \lambda z_2$, we find in particular that
\begin{equation} \label{eq:Theta3}
\Theta_\lambda(z) \ge 2\iota z_1 + \lambda z_2.
\end{equation}

The right-hand sides of \eqref{eq:Theta1}--\eqref{eq:Theta3} therefore
represent subdifferentials of $\Theta_\lambda$ at $0$. Since the space of
subdifferentials is necessarily convex, the same applies to any convex
combination. That is, whenever $\lambda^2 + \iota^2 \le 1$ and
\[
\frac{\iota^2}{1 - \lambda} - 1 \le \kappa \le 1 - \frac{\iota^2}{1 + \lambda},
\]
then
\[
\Theta_\lambda(z) \ge 2\iota z_1 + \kappa z_2
\]
for all $z \in \R^2$. The above inequalities for $\kappa$, $\iota$, and $\lambda$ are clearly
equivalent to the inequality from the statement of the lemma.

We also note that for $\lambda = \pm 1$, the condition of the lemma requires that $\iota = 0$.
As $\Theta_\lambda(z) \ge |z_2|$ in any case, we still have the desired estimate.
\end{proof}

\subsection{Decomposition into curves}

According to a theory by Bonicatto and Gusev \cite{Bonicatto-Gusev:22},
any normal ($\R$-valued) $1$-current on $\R^2$ has a decomposition
into Lipschitz curves. We will apply this result to the first component
of an $\R^2$-valued $1$-current. To this end, we consider the space
$\Gamma$, comprising all Lipschitz functions $\gamma \colon [0, 1] \to \R^2$,
equipped with the uniform norm. Given $\gamma \in \Gamma$, let
$[\gamma]$ denote the $1$-current induced by $\gamma$ through the formula
\[
[\gamma](\zeta) = \int_0^1 \zeta(\gamma(t)) \dot{\gamma}(t) \, dt
\]
for $\zeta \in C_0^\infty(\R^2; \R^{1 \times 2})$. We also write
$\Gamma_0$ for the set of all $\gamma \in \Gamma$ with $\gamma(0) = \gamma(1)$,
and $\Gamma_1$ for the set of all $\gamma \in \Gamma$ with $\gamma(0) = a^-$
and $\gamma(1) = a^+$. Let $\gamma^0 \in \Gamma_1$ denote the curve
with $\gamma^0(t) = ta^+ + (1 - t) a^-$ for $t \in [0, 1]$ (parametrising the line
segment between $a^-$ and $a^+$).

Given a Borel measurable function $\lambda \colon \R^2 \to [-1, 1]$ and a
function $h \in C^2(\R^2)$, define the functional
\begin{multline*}
Z_{\lambda, h}(\gamma) = \int_0^1 \left(\sqrt{W(\gamma(t))} \Theta_{\lambda(\gamma(t))}(\dot{\gamma}(t)) + \frac{\partial^2 h}{\partial y_2^2}(\gamma(t)) \dot{\gamma}_2(t)\right) \, dt \\
- \dd{h}{y_2}(\gamma(1)) + \dd{h}{y_2}(\gamma(0))
\end{multline*}
for $\gamma \in \Gamma$. Recall that $C_{\bar{p}}^j(\R^2)$ denotes the space of all
$\phi \in C^2(\R^2)$ such that there exists a constant $C \ge 0$ satisfying $|D^k\phi(y)| \le C(|y|^{\bar{p} - k} + 1)$
for all $y \in \R^2$ and for $k = 0, \dots, j$.

\begin{theorem} \label{thm:decomposition}
Suppose that $W \colon \R^2 \to [0, \infty)$ is continuous and satisfies the growth
condition \eqref{eq:growth-W}. Let $\lambda \colon \R^2 \to [-1, 1]$ be a
Borel measurable function, and suppose that
$h \in C_{2 + \bar{p}}^2(\R^2)$ satisfies $\frac{\partial^2 h}{\partial y_1^2} = -\lambda \sqrt{W}$ in $\R^2$.
Then for any $T \in \ncrt_{2 \times 2}^0(\R^2)$,
\[
\M_F(T) \ge \frac{1}{2} \inf_{\gamma \in \Gamma_1} Z_{\lambda, h}(\gamma).
\]
\end{theorem}

\begin{proof}
Let
\[
m_0 = \frac{1}{2} \inf_{\gamma \in \Gamma_1} Z_{\lambda, h}(\gamma).
\]
If $m_0 < 0$, then there is nothing to prove, as $\M_F(T) \ge 0$ for all $T \in \ncrt_{2 \times 2}^0$.
We therefore assume that $m_0 \ge 0$.

Let $T \in \ncrt_{2 \times 2}^0(\R^2)$ with $\M_F(T) < \infty$.
We write
\[
\vec{T} = \begin{pmatrix} T_{11} & T_{12} \\ T_{21} & T_{22} \end{pmatrix},
\]
and we write $T_1, T_2$ for the components of $T$, i.e., for the
$\R$-valued $1$-currents such that
\[
T_i(\zeta) = \int_{\R^2} \zeta \begin{pmatrix} T_{i1} \\ T_{i2} \end{pmatrix} \, d\|T\|
\]
for $\zeta \in C_0^\infty(\R^2; \R^{1 \times 2})$. We consider the
Radon-Nikodym decomposition of the measure $\|T\|$ with respect to
$\|T_1\|$. Thus we obtain two measures $\nu_1$ and $\nu_2$ with
$\|T\| = \nu_1 + \nu_2$, such that $\nu_1 \ll \|T_1\|$ and $\nu_2 \perp \|T_1\|$.
Since $F^*(y, \vec{T}(y)) < \infty$ at $\|T\|$-almost every $y \in \R^2$,
it follows that
$\vec{T} = \pm (\begin{smallmatrix} 0 & 0 \\ 1 & 0 \end{smallmatrix})$
at $\nu_2$-almost every point. At $\nu_1$-almost every point, on the
other hand, we conclude that
$\vec{T} = (\begin{smallmatrix} T_{11} & T_{12} \\ T_{21} & -T_{11} \end{smallmatrix})$.
According to Lemma \ref{lem:f^*-estimate}, we now have the inequality
\begin{equation} \label{eq:M_F-first-estimate}
\begin{split}
\M_F(T) & = \int_{\R^2} \sqrt{W(y) f^*(\vec{T})} \, d\|T\| \\
& \ge \frac{1}{2} \int_{\R^2} \sqrt{W(y)} \bigl(\Theta_\lambda(T_{11}, T_{12}) + \lambda T_{21} \bigr) \, d\nu_1 + \frac{1}{2} \int_{\R^2} \sqrt{W(y)} \, d\nu_2 \\
& \ge \frac{1}{2} \int_{\R^2} \sqrt{W(y)} \Theta_\lambda(T_{11}, T_{12}) \, d\nu_1 + \frac{1}{2} \int_{\R^2} \sqrt{W(y)} \lambda T_{21} \, d\|T\| \\
& = \frac{1}{2} \int_{\R^2} \sqrt{W(y)} \Theta_\lambda(T_{11}, T_{12}) \, d\nu_1 - \frac{1}{2} \int_{\R^2}  \frac{\partial^2 h}{\partial y_1^2} T_{21} \, d\|T\|.
\end{split}
\end{equation}

We now want to test the condition $\partial T_2 = 0$ with the function
$\dd{h}{y_1}$, but since it does not have compact support in general, we
require an approximation procedure here. For $R > 0$, let
$\chi_R \in C_0^\infty(B_{3R})$ with $\chi_R \equiv 1$ in $B_R(0)$
and $0 \le \chi_R \le 1$ everywhere, and such that $|D\chi_R| \le 1/R$.
Then
\[
\begin{split}
0 & = \partial T_2\left(\chi_R \dd{h}{y_1}\right) \\
& = \int_{\R^2} \chi_R \left(\frac{\partial^2 h}{\partial y_1^2} T_{21} + \frac{\partial^2 h}{\partial y_1 \partial y_2} T_{22}\right) \, d\|T\| \\
& \quad + \int_{\R^2} \dd{h}{y_1} \left(\dd{\chi_R}{y_1} T_{21} + \dd{\chi_R}{y_2} T_{22}\right) \, d\|T\|.
\end{split}
\]
Because of the assumptions on $h$, there is a constant $C_1 > 0$ such that
$|D^2 h| \le C_1 (\sqrt{W} + 1)$ and $|Dh| |D\chi_R| \le C_1 (\sqrt{W} + 1)$ in $\R^2$ for all $R> 0$.
Since $\|T\|$ is a Radon measure and  $\M_F(T) < \infty$, we know that $\sqrt{W} + 1$ is integrable with
respect to $\|T\|$. Hence we can use Lebesgue's dominated convergence
theorem when we take the limit $R \to \infty$. It follows that
\[
0 = \int_{\R^2} \left(\frac{\partial^2 h}{\partial y_1^2} T_{21} + \frac{\partial^2 h}{\partial y_1 \partial y_2} T_{22}\right) \, d\|T\|.
\]
Similarly, since $\partial T_1 = \partial T_1^0$, we can show that
\[
\int_{\R^2} \left(\frac{\partial^2 h}{\partial y_1 \partial y_2} T_{11} + \frac{\partial^2 h}{\partial y_2^2} T_{12}\right) \, d\|T\| = \dd{h}{y_2}(a^+) - \dd{h}{y_2}(a^-).
\]
Using also the fact that $T_{11} + T_{22} = 0$ almost everywhere, and
combining these formulas with \eqref{eq:M_F-first-estimate}, we obtain
\begin{equation} \label{eq:M_F-second-estimate}
\M_F(T) \ge \frac{1}{2} \int_{\R^2} \left(\sqrt{W(y)} \Theta_\lambda(T_{11}, T_{12}) + \frac{\partial^2 h}{\partial y_2^2} T_{12}\right) \, d\|T\| - \dd{h}{y_2}(a^+) + \dd{h}{y_2}(a^-).
\end{equation}

The results of Bonicatto and Gusev \cite{Bonicatto-Gusev:22} give rise to a Borel measure $\mu$ on $\Gamma$
such that
\begin{equation} \label{eq:decomposition}
T_1 = \int_\Gamma [\gamma] \, d\mu(\gamma),
\end{equation}
in the sense that
\[
T_1(\zeta) = \int_\Gamma [\gamma](\zeta) \, d\mu(\gamma)
\]
for any $\zeta \in C_0^\infty(\R^2; \R^{1 \times 2})$. Moreover,
this measure also satisfies
\begin{equation} \label{eq:decomposition-total-variation}
\|T_1\| = \int_\Gamma \|[\gamma]\| \, d\mu(\gamma)
\end{equation}
and
\begin{equation} \label{eq:decomposition-boundary}
\|\partial T_1\| = \int_\Gamma \|\partial [\gamma]\| \, d\mu(\gamma)
\end{equation}
(which is to be interpreted similarly).

For any $\gamma \in \Gamma$, we clearly have $\partial [\gamma] = 0$
if $\gamma \in \Gamma_0$ and $\|\partial [\gamma]\|(\R^2) = 2$
otherwise. Thus \eqref{eq:decomposition-boundary} implies that
$\mu(\Gamma_1) = 1$ and $\mu(\Gamma \setminus (\Gamma_0 \cup \Gamma_1)) = 0$,
while \eqref{eq:M_F-second-estimate}--\eqref{eq:decomposition-total-variation} imply that
\[
\M_F(T) \ge \frac{1}{2} \int_\Gamma Z_{\lambda, h}(\gamma) \, d\mu(\gamma).
\]

Since $Z_{\lambda, h}(\gamma) \ge 2m_0$ for all $\gamma \in \Gamma_1$, we automatically have the inequality $Z_{\lambda, h}(\gamma) \ge 0$
for all $\gamma \in \Gamma_0$. (If we had $\hat{\gamma} \in \Gamma_0$ with $Z_{\lambda, h}(\hat{\gamma}) < 0$,
then we could construct a sequence of curves $\gamma_k \in \Gamma_1$ with
$\lim_{k \to \infty} Z_{\lambda, h}(\gamma_k) = - \infty$ as follows: choose
$\tilde{\gamma}_1, \tilde{\gamma}_2 \in \Gamma$ with $\tilde{\gamma}_1(0) = a^-$,
$\tilde{\gamma}_2(1) = a^+$, and
$\tilde{\gamma}_1(1) = \tilde{\gamma}_2(0) = \hat{\gamma}(0)$.
Concatenate $\tilde{\gamma}_1$ with $k$ copies of $\hat{\gamma}$ and then
$\tilde{\gamma}_2$, and reparametrise appropriately. Note that
$Z_{\lambda, h}$ is invariant under reparametrisation.)

It therefore follows that
\[
\M_F(T) \ge \frac{1}{2} \int_{\Gamma_1} Z_{\lambda, h}(\gamma) \, d\mu(\gamma) \ge m_0.
\]
This concludes the proof.
\end{proof}

\subsection{Proof of Corollary \ref{cor:PDE}}

Given suitable functions $\lambda$ and $h$, the minimisers of the functional $Z_{\lambda, h}$
can in principle be determined with the conventional tools from the calculus of
variations. There are some difficulties coming from the fact that $\Theta_\lambda$ has linear
growth, but nevertheless, an analysis of certain ordinary differential equations can then potentially reveal some
information about the central question of this paper. In practice, however, it is difficult to
make any specific statements this way. The proof of Corollary \ref{cor:PDE}, however, also relies on Theorem \ref{thm:decomposition}.

\begin{proof}[Proof of Corollary \ref{cor:PDE}]
We first observe that we may assume without loss of generality that $w \ge 0$ in $\R^2$. If this
condition does not hold true, then we can replace $w$ by $|w|$ and replace $\iota$, $\kappa$, and $\lambda$
by $w\iota/|w|$, $w\kappa/|w|$, and $w\lambda/|w|$, respectively.
This will change neither the inequalities in the statement nor equation \eqref{eq:PDE-criterion}.

Let $b \in \R$. Define
\begin{multline*}
h(y) = - \int_{a_1^-}^{y_1} (y_1 - s) (\lambda w)(s, y_2) \, ds \\
+ (y_1 - a_1^-) \int_b^{y_2} \left(2(\iota w)(a_1^-, t) - (y_2 - t) \dd{}{y_1}(\kappa w)(a_1^-, t)\right) \, dt.
\end{multline*}
Then
\[
\frac{\partial^2 h}{\partial y_1^2} = - \lambda w
\]
and
\begin{multline*}
\frac{\partial^2 h}{\partial y_1 \partial y_2}(y) = - \int_{a_1^-}^{y_1} \dd{}{y_2}(\lambda w)(s, y_2) \, ds + 2(\iota w)(a_1^-, y_2) \\
- \int_b^{y_2} \dd{}{y_1}(\kappa w)(a_1^-, t) \, dt.
\end{multline*}
Moreover,
\[
\begin{split}
\frac{\partial^2 h}{\partial y_2^2}(y) & = - \int_{a_1^-}^{y_1} (y_1 - s) \frac{\partial^2}{\partial y_2^2}(\lambda w)(s, y_2) \, ds \\
& \quad + (y_1 - a_1^-) \left(2\dd{}{y_2}(\iota w)(a_1^-, y_2) - \dd{}{y_1}(\kappa w)(a_1^-, y_2)\right) \\
& = \int_{a_1^-}^{y_1} (y_1 - s) \left(2\frac{\partial^2}{\partial y_1 \partial y_2}(\iota w)(s, y_2) - \frac{\partial^2}{\partial y_1^2}(\kappa w)(s, y_2)\right) \, ds \\
& \quad + (y_1 - a_1^-) \left(2\dd{}{y_2}(\iota w)(a_1^-, y_2) - \dd{}{y_1}(\kappa w)(a_1^-, y_2)\right) \\
& = \int_{a_1^-}^{y_1} \left(2\dd{}{y_2}(\iota w)(s, y_2) - \dd{}{y_1}(\kappa w)(s, y_2)\right) \, ds \\
& = 2\int_{a_1^-}^{y_1} \dd{}{y_2}(\iota w)(s, y_2) \, ds - (\kappa w)(y) + (\kappa w)(a_1^-, y_2)
\end{split}
\]
by \eqref{eq:PDE-criterion} and an integration by parts.
From this, we see that $h \in C_{2 + \bar{p}}^2(\R^2)$.

Also consider the function
\begin{multline*}
\phi(y) = \dd{h}{y_2}(y) + \int_b^{y_2} (\kappa w)(y_1, t) \, dt - 2y_1 (\iota w)(a_1^-, b)\\
+ \int_{a_1^-}^{y_1} \left((y_1 - s) \dd{}{y_2} (\lambda w)(s, b) + 2 (\iota w)(s, b)\right) \, ds.
\end{multline*}
Then
\begin{equation} \label{eq:derivative-of-phi}
\begin{split}
\dd{\phi}{y_1}(y) & = \frac{\partial^2 h}{\partial y_1 \partial y_2}(y) + \int_b^{y_2} \dd{}{y_1}(\kappa w)(y_1, t) \, dt - 2(\iota w)(a_1^-, b) \\
& \quad + \int_{a_1^-}^{y_1} \dd{}{y_2}(\lambda w)(s, b) \, ds + 2(\iota w)(y_1, b) \\
& = - \int_{a_1^-}^{y_1} \dd{}{y_2}(\lambda w)(s, y_2) \, ds - \int_b^{y_2} \dd{}{y_1}(\kappa w)(a_1^-, t) \, dt \\
& \quad + \int_b^{y_2} \dd{}{y_1}(\kappa w)(y_1, t) \, dt + \int_{a_1^-}^{y_1} \dd{}{y_2}(\lambda w)(s, b) \, ds \\
& \quad + 2(\iota w)(a_1^-, y_2) + 2(\iota w)(y_1, b) - 2(\iota w)(a_1^-, b).
\end{split}
\end{equation}
Because of equation \eqref{eq:PDE-criterion}, we compute
\[
\begin{split}
0 & = \int_{a_1^-}^{y_1} \int_b^{y_2} \left(\frac{\partial^2}{\partial y_1^2}(\kappa w)(s, t) - \frac{\partial^2}{\partial y_2^2}(\lambda w)(s, t) - 2\frac{\partial^2}{\partial y_1 \partial y_2}(\iota w)(s, t)\right) \, dt \, ds \\
& = \int_b^{y_2} \left(\dd{}{y_1}(\kappa w)(y_1, t) - \dd{}{y_1}(\kappa w)(a_1^-, t)\right) \, dt \\
& \quad - \int_{a_1^-}^{y_1} \left(\dd{}{y_2}(\lambda w)(s, y_2) - \dd{}{y_2}(\lambda w)(s, b)\right) \, ds \\
& \quad - 2\int_{a_1^-}^{y_1} \left(\dd{}{y_1}(\iota w)(s, y_2) - \dd{}{y_1}(\iota w)(s, b)\right) \, ds \\
& = \int_b^{y_2} \dd{}{y_1}(\kappa w)(y_1, t) \, dt - \int_b^{y_2} \dd{}{y_1}(\kappa w)(a_1^-, t) \, dt \\
& \quad - \int_{a_1^-}^{y_1} \dd{}{y_2}(\lambda w)(s, y_2) \, ds + \int_{a_1^-}^{y_1} \dd{}{y_2}(\lambda w)(s, b) \, ds \\
& \quad - 2(\iota w)(y) + 2(\iota w)(a_1^-, y_2) + 2(\iota w)(y_1, b) - 2(\iota w)(a_1^-, b).
\end{split}
\]
Comparing with \eqref{eq:derivative-of-phi}, we conclude that
\[
\dd{\phi}{y_1} = 2\iota w.
\]
Furthermore, we compute
\[
\dd{\phi}{y_2} = \kappa w + \frac{\partial^2 h}{\partial y_2^2}.
\]

Let $\gamma \in \Gamma_1$. By Lemma \ref{lem:subdifferential}, we can now estimate
\[
\begin{split}
Z_{\lambda, h}(\gamma) & = \int_0^1 \left(w(\gamma) \Theta_{\lambda(\gamma)}(\dot{\gamma}) + \frac{\partial^2 h}{\partial y_2^2}(\gamma) \dot{\gamma}_2\right) \, dt - \dd{h}{y_2}(a^+) + \dd{h}{y_2}(a^-) \\
& \ge \int_0^1 D\phi(\gamma) \dot{\gamma} \, dt - \dd{h}{y_2}(a^+) + \dd{h}{y_2}(a^-) \\
& = \phi(a^+) - \phi(a^-) - \dd{h}{y_2}(a^+) + \dd{h}{y_2}(a^-) \\
& = \int_{a_2^-}^{a_2^+} (\kappa w)(a_1^-, t) \, dt.
\end{split}
\]
The claim then follows from Theorem \ref{thm:decomposition}.
\end{proof}

\section{Examples} \label{sct:examples}

\subsection{Variants of the Aviles-Giga functional}

A singular perturbation problem involving the quantity
\[
\frac{1}{2} \int_\Omega \left(\epsilon |D^2 \phi|^2 + \frac{1}{\epsilon} (1 - |D\phi|^2)^2\right) \, dx,
\]
was studied by Aviles and Giga \cite{Aviles-Giga:87}, and subsequently by
many other authors, including, e.g., Ambrosio, De Lellis, and Mantegazza \cite{Ambrosio-DeLellis-Mantegazza:99}
and DeSimone, Kohn, M\"uller, and Otto \cite{DeSimone-Kohn-Mueller-Otto:01}. A key contribution by Jin and Kohn \cite{Jin-Kohn:00}
determined the energy required for a jump of $D\phi$, with tools similar
to what we use in this paper.

If we define $u = \nabla^\perp \phi$, then we have the functional $E_\epsilon(u; \Omega)$ from
the introduction with the constraint $\div u = 0$. Our theory therefore applies in principle
(but will of course not give anything new, as the problem is well understood, at least in relation
to the question that we study here).

Indeed, the function $w(y) = 1 - |y|^2$ (corresponding to $W(y) = (w(y))^2 = (1 - |y|^2)^2$) is
a solution of the wave equation
\[
\frac{\partial^2 w}{\partial y_1^2} - \frac{\partial^2 w}{\partial y_2^2} = 0
\]
(a fact that was also observed by Ignat and Monteil \cite{Ignat-Monteil:20}),
thus it satisfies the hypothesis of Corollary \ref{cor:PDE} with $\kappa = \lambda = 1$ and $\iota = 0$.
For $a^-, a^+ \in S^1 = \set{y \in \R^2}{|y| = 1}$, we therefore obtain
\[
\mathcal{E}(a^-, a^+) = \int_{[a^-, a^+]} w \, d\Ha^1 = \frac{1}{6} |a^+ - a^-|^3.
\]

We now consider potentials that are different, but similar in structure, including
\[
w(y) = |y|^{2n}(1 - |y|^2) \quad \text{and} \quad w(y) = 1 - |y|^{2n}
\]
for some $n \in \N$, and
\begin{equation} \label{eq:AG-beta}
w(y) = (1 - |y|^2)^\beta
\end{equation}
for a number $\beta \in (0, 1)$. We first note that for the last of these, when $\beta > 1$, the optimal transitions between
two points $a^-, a^+ \in S^1$ are \emph{not} expected to be one-dimensional by the results of
Ambrosio, De Lellis, and Mantegazza \cite{Ambrosio-DeLellis-Mantegazza:99} (see also the
discussion by Ignat and Merlet \cite{Ignat-Merlet:12}).

We restrict our attention
to transitions between the points $a^- = (0, -1)$ and $a^+ = (0, 1)$ here, because we make use
of the resulting symmetry. It is an open problem whether the corresponding statements hold
true in general, but for \eqref{eq:AG-beta}, the work of Ignat and Merlet \cite{Ignat-Merlet:12} at least
gives some results supporting the conjecture that the optimal transition profile will be one-dimensional
when $\beta \in (0, 1)$ and $|a^+ - a^-|$ is small.

We wish to make use of Corollary \ref{cor:PDE}, but it suffices to consider the case $\iota = 0$.
Thus we study the question whether there exist two functions $\kappa, \lambda \colon \R^2 \to [-1, 1]$
such that
\[
\frac{\partial^2}{\partial y_1^2}(\kappa w) = \frac{\partial^2}{\partial y_2^2}(\lambda w)
\]
in $\R^2$ and $\kappa = 1$ on $[a^-, a^+]$. Note that these conditions are satisfied if there exists $\phi \in C^4(\R^2)$
such that
\begin{equation} \label{eq:second-derivatives-bounded-by-w}
\left|\frac{\partial^2 \phi}{\partial y_1^2}\right| \le w \quad \text{and} \quad \left|\frac{\partial^2 \phi}{\partial y_2^2}\right| \le w
\end{equation}
and
\begin{equation} \label{eq:equality-on-line-segment}
\frac{\partial^2 \phi}{\partial y_2^2} = w \quad \text{on $[a^-, a^+]$}.
\end{equation}
Indeed, in this case, we can set
\[
\kappa = w^{-1} \frac{\partial^2 \phi}{\partial y_2^2} \quad \text{and} \quad \lambda = w^{-1} \frac{\partial^2 \phi}{\partial y_1^2}.
\]
We therefore consider the set
\[
\mathcal{W}(a^-, a^+) = \set{w \in C^0(\R^2)}{\text{there exists } \phi \in C^4(\R^2) \text{ satisfying \eqref{eq:second-derivatives-bounded-by-w} and \eqref{eq:equality-on-line-segment}}}.
\]
It is easy to see that $\mathcal{W}(a^-, a^+)$ has the following properties.
\begin{enumerate}
\item If $w \in \mathcal{W}(a^-, a^+)$ and $t \ge 0$, them $tw \in \mathcal{W}(a^-, a^+)$.
\item If $w_1, w_2 \in \mathcal{W}(a^-, a^+)$, then $w_1 + w_2 \in \mathcal{W}(a^-, a^+)$.
\end{enumerate}
Thus $\mathcal{W}(a^-, a^+)$ is a convex cone.

\begin{proposition} \label{prp:polynomial}
For any $n \in \N_0$, there exists a polynomial $P \colon \R^2 \to \R$ such
that
\begin{equation} \label{eq:second-derivatives-polynomial}
\left|\frac{\partial^2 P}{\partial y_1^2}(y)\right| \le |y|^{2n} \bigl|1 - |y|^2\bigr| \quad \text{and} \quad \left|\frac{\partial^2 P}{\partial y_2^2}(y)\right| \le |y|^{2n} \bigl|1 - |y|^2\bigr|
\end{equation}
for every $y \in \R^2$ and
and
\begin{equation} \label{eq:equality-polynomial}
\frac{\partial^2 P}{\partial y_2^2}(0, y_2) = y_2^{2n} (1 - y_2^2)
\end{equation}
for every $y_2 \in \R$.
\end{proposition}

\begin{proof}
For $n = 0$, a suitable polynomial is
\[
P(y) = \frac{|y|^2}{2} - \frac{|y|^4}{12} - \frac{y_1^2 y_2^2}{3}.
\]

We now assume that $n \ge 1$. We look for a polynomial such that
\[
\frac{\partial^2 P}{\partial y_2^2}(y) = (1 - |y|^2) \left(y_2^{2n} + \sum_{k = 0}^{n - 1} c_k y_1^{2n - 2k} y_2^{2k}\right)
\]
for certain coefficients $c_0, \dotsc, c_{n - 1}$. Setting $c_n = 1$, we
can write
\[
\begin{split}
\frac{\partial^2 P}{\partial y_2^2}(y) & = (1 - |y|^2) \sum_{k = 0}^n c_k y_1^{2n - 2k} y_2^{2k} \\
& = (1 - y_1^2) \sum_{k = 0}^n c_k y_1^{2n - 2k} y_2^{2k} - \sum_{k = 0}^n c_k y_1^{2n - 2k} y_2^{2k + 2}.
\end{split}
\]
A possible solution is
\[
\begin{split}
P(y) & = (1 - y_1^2) \sum_{k = 0}^n \frac{c_k}{(2k + 2)(2k + 1)} y_1^{2n - 2k} y_2^{2k + 2} \\
& \quad - \sum_{k = 0}^n \frac{c_k}{(2k + 4)(2k + 3)} y_1^{2n - 2k} y_2^{2k + 4} \\
& \quad + \frac{y_1^{2n + 2}}{(2n + 2)(2n + 1)} - \frac{y_1^{2n + 4}}{(2n + 4)(2n + 3)} \\
& = \sum_{k = 0}^{n - 1} \frac{c_k}{(2k + 2)(2k + 1)} y_1^{2n - 2k} y_2^{2k + 2} + \frac{y_1^{2n + 2} + y_2^{2n + 2}}{(2n + 2)(2n + 1)} \\
& \quad - \sum_{k = 1}^{n - 1} \frac{c_k + c_{k - 1}}{(2k + 2)(2k + 1)} y_1^{2n - 2k + 2} y_2^{2k + 2} - \frac{c_0}{2} y_1^{2n + 2} y_2^2 \\
& \quad - \frac{1 + c_{n - 1}}{(2n + 2)(2n + 1)} y_1^2 y_2^{2n + 2} - \frac{y_1^{2n + 4} + y_2^{2n + 4}}{(2n + 4)(2n + 3)}.
\end{split}
\]

We want to impose the symmetry condition $P(y_1, y_2) = P(y_2, y_1)$
(so that the above condition on $\frac{\partial^2 P}{\partial y_2^2}$
automatically gives a similar condition for
$\frac{\partial^2 P}{\partial y_1^2}$). This requires that
\begin{equation} \label{eq:coefficients1}
\frac{c_0}{2} = \frac{1 + c_{n - 1}}{(2n + 2)(2n + 1)}
\end{equation}
and
\begin{equation} \label{eq:coefficients2}
\frac{c_k}{(2k + 2)(2k + 1)} = \frac{c_{n - k - 1}}{(2n - 2k)(2n - 2k - 1)}, \quad k = 0, \dotsc, n - 1,
\end{equation}
and
\begin{equation} \label{eq:coefficients3}
\frac{c_k + c_{k - 1}}{(2k + 2)(2k + 1)} = \frac{c_{n - k} + c_{n - k - 1}}{(2n - 2k + 2)(2n - 2k + 1)}, \quad k = 1, \dotsc, n - 1.
\end{equation}
(It may appear at first that there are too many equations for the $n$ variables
$c_0, \dotsc, c_{n - 1}$, but there is some repetition here.
Once the redundant equations are discarded, it is easy to see that there
is a unique solution.)

The combination of \eqref{eq:coefficients2} and \eqref{eq:coefficients3} gives
\[
\begin{split}
c_k + c_{k - 1} & = \frac{(2k + 2)(2 k + 1)}{(2n - 2k + 2)(2n - 2k + 1)}(c_{n - k} + c_{n - k - 1}) \\
& = \frac{(2n - 2k)(2n - 2k - 1)}{(2n - 2k + 2)(2n - 2k + 1)} c_k + \frac{(2k + 2)(2 k + 1)}{2k(2k - 1)} c_{k - 1}
\end{split}
\]
for $k = 1, \dotsc, n - 1$. Thus
\[
\left(1 - \frac{(2n - 2k)(2n - 2k - 1)}{(2n - 2k + 2)(2n - 2k + 1)}\right) c_k = \left(\frac{(2k + 2)(2 k + 1)}{2k(2k - 1)} - 1\right) c_{k - 1}.
\]
We compute
\[
1 - \frac{(2n - 2k)(2n - 2k - 1)}{(2n - 2k + 2)(2n - 2k + 1)} = \frac{8(n - k) + 2}{(2n - 2k + 2)(2n - 2k + 1)}
\]
and
\[
\frac{(2k + 2)(2 k + 1)}{2k(2k - 1)} - 1 = \frac{8k + 2}{2k(2k - 1)}.
\]
Therefore, we obtain the equation
\[
\frac{c_k}{c_{k - 1}} = \frac{(2n - 2k + 2)(2n - 2k + 1)(4k + 1)}{2k(2k - 1)(4(n - k) + 1)}
\]
for $k = 1, \dotsc, n - 1$.

Define
\[
b_k = \frac{c_k}{(2k + 1)\binom{n}{k}}, \quad k = 0, \dotsc, n - 1.
\]
Then
\[
\begin{split}
\frac{b_k}{b_{k - 1}} & = \frac{k(2k - 1)}{(n - k + 1)(2k + 1)} \frac{c_k}{c_{k - 1}} \\
& = \frac{(2n - 2k + 1)(4k + 1)}{(2k + 1)(4(n - k) + 1)} \\
& = \frac{8nk - 8k^2 + 2n + 2k + 1}{8nk - 8k^2 + 4n - 2k + 1}.
\end{split}
\]
We note that $b_k/b_{k - 1} \ge 1$ if, and only if, $k \ge n/2$; and
$b_k/b_{k - 1} \le 1$ if, and only if, $k \le n/2$.
Hence $b_k$, as a function of $k$, is first decreasing, may possibly be constant
for one step, and is then increasing. (If $n = 1$, then this is a vacuous statement, as
we have only $b_0$ in this case.)

From \eqref{eq:coefficients1} and \eqref{eq:coefficients2} for $k = 0$, we also
conclude that
\[
\frac{1 + c_{n - 1}}{(2n + 2)(2n + 1)} = \frac{c_{n - 1}}{2n(2n - 1)}.
\]
Solving this equation, we obtain
\[
c_{n - 1} = \frac{2n^2 - n}{4n + 1}.
\]
It then also follows from \eqref{eq:coefficients2} that
\[
c_0 = \frac{c_{n - 1}}{2n^2 - n} = \frac{1}{4n + 1}.
\]
This means that
\[
b_0 = c_0 = \frac{1}{4n + 1} \quad \text{and} \quad b_{n - 1} = \frac{c_{n - 1}}{2n^2 - n} = \frac{1}{4n + 1}.
\]
We conclude that $b_k \le \frac{1}{4n + 1}$ for all $k = 0, \dotsc, n - 1$,
i.e.,
\[
c_k \le \frac{2k + 1}{4n + 1} \binom{n}{k}.
\]
Because $c_k/c_{k - 1} > 0$ for every $k = 1, \dotsc, n - 1$, it is clear that
$c_k > 0$ for every $k = 1, \dotsc, n$. Therefore,
\[
0 \le \sum_{k = 0}^n c_k y_1^{2n - 2k} y_2^{2k} \le \sum_{k = 0}^n \binom{n}{k} y_1^{2n - 2k} y_2^{2k} = |y|^{2n}.
\]
The inequalities in \eqref{eq:second-derivatives-polynomial} follow.
We also have identity \eqref{eq:equality-polynomial} by construction.
\end{proof}

Recall that we consider the points $a^- = (0, -1)$ and $a^+ = (0, 1)$ here.
It follows from Proposition \ref{prp:polynomial} that $|y|^{2n}\bigl|1 - |y|^2\bigr| \in \mathcal{W}(a^-, a^+)$
for every $n \in \N$. Since
\[
|y|^{2n} \bigl|1 - |y|^{2m}\bigr| = (|y|^{2n} + \dotsb + |y|^{2n + 2m - 2}) \bigl|1 - |y|^2\bigr|,
\]
these potentials belong to $\mathcal{W}(a^-, a^+)$, too. Now for
$\beta \in (0, 1)$, we consider
\[
w(y) = \bigl|1 - |y|^2\bigr|^\beta.
\]
We define the function
\[
\psi(t) = (1 - t)^{\beta - 1}, \quad -1 < t < 1.
\]
We compute
\[
\psi^{(n)}(t) = (1 - \beta) \dotsb (n - \beta) (1 - t)^{\beta - n - 1}.
\]
The function is analytic in $(-1, 1)$ and we have the Taylor expansion
\[
\psi(t) = \sum_{n = 0}^\infty a_n t^n,
\]
where $a_n = \frac{\psi^{(n)}(0)}{n!} > 0$ for every $n \in \N_0$. We
therefore have the formula
\[
\bigl|1 - |y|^2\bigr|^\beta = (1 - |y|^2) \psi(|y|^2) = (1 - |y|^2) \sum_{n = 0}^\infty a_n |y|^{2n}
\]
in $B_1(0)$.

It is not clear if the space $\mathcal{W}(a^-, a^+)$ is closed in a suitable
topology, but examining the coefficients of the polynomials from
Proposition \ref{prp:polynomial}, it is not difficult to see that $w$
satisfies a condition like \eqref{eq:second-derivatives-bounded-by-w} and
\eqref{eq:equality-on-line-segment} in the unit ball. But it is still not clear
how to extend this observation to $\R^2$. Therefore, rather than using
the series, we use an approximation by
\[
w_n(y) = \bigl|1 - |y|^2\bigr| \sum_{k = 0}^n a_k |y|^{2k}.
\]
The limit, as $n \to \infty$, is
\[
w_\infty(y) = \begin{cases}
(1 - |y|^2)^\beta & \text{if $|y| \le 1$}, \\
\infty & \text{if $|y| > 1$}.
\end{cases}
\]
This function does of course not fit into the above theory, as it is not continuous.
Nevertheless, we can prove the following.

\begin{corollary} \label{cor:examples}
Let $a^- = (0, -1)$ and $a^+ = (0, 1)$.
\begin{enumerate}
\item If $W(y) = |y|^{4n}(1 - |y|^{2m})^2$ for some $n, m \in \N$, then
\[
\mathcal{E}(a^-, a^+) = \frac{4m}{(2n + 1)(2n + 2m + 1)}.
\]
\item Suppose that $W(y) = (1 - |y|^2)^{2\beta}$ for some $\beta \in (0, 1)$. If $(u_\epsilon)_{\epsilon > 0}$
is a family of vector fields from $\mathcal{U}(a^-, a^+)$ such that $|u_\epsilon| \le 1$ for every $\epsilon > 0$, then
\[
\liminf_{\epsilon \searrow 0} E_\epsilon(u_\epsilon; B_1(0)) \ge 4\int_0^1 (1 - t^2)^\beta \, dt.
\]
\end{enumerate}
\end{corollary}

In the second statement, the value of the lower bound is identical to
\[
2\int_{[a^-, a^+]} \sqrt{W} \, d\Ha^1.
\]
(We have the constant $2$ because when integrating over
$B_1(0)$, we effectively account for a transition interface of length $2$.)
Thus in both statements, we conclude
that one-dimensional transition profiles are energetically optimal
under the respective assumptions.

\begin{proof}
The first statement follows immediately from Proposition \ref{prp:polynomial} and Corollary \ref{cor:PDE}
by the above observations.

To prove the second statement, we consider the potentials
\[
w_n(y) = \bigl|1 - |y|^2\bigr| \sum_{k = 0}^n a_k |y|^{2k},
\]
as explained above. We know that $w_n \in \mathcal{W}(a^-, a^+)$. Hence by Corollary \ref{cor:PDE},
\[
\frac{1}{4} \liminf_{\epsilon \searrow 0} \int_{B_1(0)} \left(\epsilon |Du_\epsilon|^2 + \frac{1}{\epsilon} (w_n(u_\epsilon))^2\right) \, dx \ge \int_{[a^-, a^+]} w_n \, d\Ha^1.
\]
Since $|u_\epsilon| \le 1$ for every $\epsilon > 0$, it follows that
\[
\frac{1}{4} \liminf_{\epsilon \searrow 0} \int_{B_1(0)} \left(\epsilon |Du_\epsilon|^2 + \frac{1}{\epsilon} \bigl(1 - |u_\epsilon|^2\bigr)^{2\beta}\right) \, dx \ge \int_{[a^-, a^+]} w_n \, d\Ha^1
\]
as well. Letting $n \to \infty$, we conclude that
\[
\frac{1}{4} \liminf_{\epsilon \searrow 0} \int_{B_1(0)} \left(\epsilon |Du_\epsilon|^2 + \frac{1}{\epsilon} \bigl(1 - |u_\epsilon|^2\bigr)^{2\beta}\right) \, dx \ge \int_{[a^-, a^+]} w \, d\Ha^1.
\]
This is the inequality from the statement in a different form.
\end{proof}

\subsection{Other candidates for minimisers}

We cannot expect that $T^0$ will always be a minimiser of $\M_F$. In this section,
we therefore have a look at some other elements of $\ncrt_{2 \times 2}^0$ that
may be minimisers for certain potential functions $W$.
We do not have any specific results here, but we do have some examples indicating
that there may be a deeper relationship between the elements of $\ncrt_{2 \times 2}^0$ and possible
transition profiles.

In Section \ref{sct:geometric}, we decomposed the first component of an $\R^2$-valued
current into curves from $a^-$ to $a^+$ (and possibly some closed curves). Conversely, given such a curve, we may wish
to consider the corresponding $\R$-valued $1$-current and complement it with a second
component to obtain an element of $T_{2 \times 2}^0$. Since we require that
$\M_F(T) < \infty$, however, we will need to make sure that $\tr \vec{T} = 0$
away from $W^{-1}(\{0\})$. This condition, on the other hand, gives rise to significant restrictions
on what is possible. For most curves from $a^-$ to $a^+$, there is no second component
with the required properties. But we can instead consider a \emph{pair} of curves that are
symmetric with respect to reflection on $[a^-, a^+]$.

We still assume that $a = (0, -1)$ and $a^+ = (0, 1)$.
Suppose now that $\gamma \colon [0, 1] \to \R^2$ is Lipschitz continuous with $\gamma(0) = a^-$ and
$\gamma(1) = a^+$, such that $\dot{\gamma}_2(t) \neq 0$ at almost every $t \in [0, 1]$ and such that
the function $\psi = \dot{\gamma}_1/\dot{\gamma}_2$ is of bounded variation in $[0, 1]$.
Its derivative, denoted by $\dot{\psi}$, is therefore given by a measure on $[0, 1]$.
We write $|\dot{\psi}|$ for its total variation measure.

Define $T \in \ncrt_{2 \times 2}$ by
\[
\begin{split}
T(\zeta) & = \frac{1}{2} \int_0^1 \frac{1}{\dot{\gamma}_2(t)} \begin{pmatrix} \dot{\gamma}_1(t) \dot{\gamma}_2(t) & (\dot{\gamma}_2(t))^2 \\ - (\dot{\gamma}_1(t))^2 & - \dot{\gamma}_1(t) \dot{\gamma}_2(t) \end{pmatrix} : \zeta(\gamma(t)) \, dt \\
& \quad + \frac{1}{2} \int_0^1 \frac{1}{\dot{\gamma}_2(t)} \begin{pmatrix} -\dot{\gamma}_1(t) \dot{\gamma}_2(t) & (\dot{\gamma}_2(t))^2 \\ -(\dot{\gamma}_1(t))^2 & \dot{\gamma}_1(t) \dot{\gamma}_2(t) \end{pmatrix} : \zeta(-\gamma_1(t), \gamma_2(t)) \, dt \\
& \quad - \frac{1}{2} \int_0^1 \int_{- \gamma_1(t)}^{\gamma_1(t)} \zeta_{21}(s, \gamma_2(t)) \, ds \, d\dot{\psi}(t).
\end{split}
\]
We note that $T_1 = \frac{1}{2}([\gamma] + [\gamma^\dagger])$, where $\gamma^\dagger(t) = (-\gamma_1(t), \gamma_2(t))$.
For any $\xi \in C_0^\infty(\R^2; \R^2)$, we compute
\[
\begin{split}
T(D\xi) & = \frac{1}{2} \int_0^1 \bigl(D\xi_1(\gamma(t)) \dot{\gamma}(t) - \psi(t) D\xi_2(\gamma(t)) \dot{\gamma}(t)\bigr) \, dt \\
& \quad + \frac{1}{2} \int_0^1 \bigl(D\xi_1(\gamma(t)) \dot{\gamma}^\dagger(t) + \psi(t) D\xi_2(\gamma(t)) \dot{\gamma}^\dagger(t)\bigr) \, dt \\
& \quad - \frac{1}{2} \int_0^1 \int_{-\gamma_1(t)}^{\gamma_1(t)} \dd{\xi_2}{y_1}(s, \gamma_2(t)) \, ds \, d\dot{\psi}(t) \\
& = \xi_1(a^+) - \xi_1(a^-) - \frac{1}{2} \int_0^1 \psi(t) \frac{d}{dt}\bigl(\xi_2(\gamma(t)) - \xi_2(\gamma^\dagger(t))\bigr) \, dt \\
& \quad - \frac{1}{2} \int_0^1 \bigl(\xi_2(\gamma(t)) - \xi_2(\gamma^\dagger(t))\bigr) \, d\dot{\psi}(t) \\
& = \xi_1(a^+) - \xi_1(a^-).
\end{split}
\]
Hence $T \in \ncrt_{2 \times 2}^0$. We further compute
\[
\begin{split}
\M_F(T) & = \frac{1}{4} \int_0^1 \left(\sqrt{W(\gamma(t))} + \sqrt{W(\gamma^\dagger(t))}\right) \frac{|\dot{\gamma}(t)|^2}{|\dot{\gamma}_2(t)|} \, dt \\
& \quad + \frac{1}{4} \int_0^1 \int_{-\gamma_1(t)}^{\gamma_1(t)} \sqrt{W(s, \gamma_2(t))} \, ds \, d|\dot{\psi}|(t).
\end{split}
\]

We now look at two specific examples of this type.

\begin{example} \label{exm:JK}
Let $b_1 > 0$, and consider the points $b^+ = (b_1, 0)$ and $b^- = (-b_1, 0)$. Suppose that
\[
\gamma(t) = \begin{cases}
(2tb_1, 2t - 1) & \text{if $0 \le t \le \frac{1}{2}$}, \\
((2 - 2t)b_1, 2t - 1) & \text{if $\frac{1}{2} < t \le 1$}
\end{cases}
\]
(a curve consisting of a line segment from $a^-$ to $b^+$ and a line segment from $b^+$ to $a^+$).
Then
\[
\psi(t) = \begin{cases}
b_1 & \text{if $0 \le t \le \frac{1}{2}$}, \\
-b_1 & \text{if $\frac{1}{2} < t \le 1$},
\end{cases}
\]
and we compute
\[
\M_F(T) = \frac{1}{4} \sqrt{b_1^2 + 1} \int_\lozenge \sqrt{W} \, d\Ha^1 + \frac{b_1}{2} \int_{[b^-, b^+]} \sqrt{W} \, d\Ha^1,
\]
where we use the abbreviation $\lozenge = [a^-, b^+] \cup [b^+, a^+] \cup [a^-, b^-] \cup [b^-, a^+]$.

Compare this with the construction by Jin and Kohn \cite[Section 4]{Jin-Kohn:00} of a
non-one-dimensional transition profile between $a^-$ and $a^+$. This is a two-scale
construction, where the coarser scale is given by
\[
\tilde{u}_0(x) = \begin{cases} (0, -1) & \text{if $x_1 \le - b_1 |x_2|$}, \\
(0, 1) & \text{if $x_1 \ge b_1 |x_2|$}, \\
(-b_1, 0) & \text{if $|x_1| < b_1 x_2$}, \\
(b_1, 0) & \text{if $|x_1| < -b_1 x_2$},
\end{cases}
\]
for $-1 < x_2 \le 1$. This is extended periodically in $x_2$, with period $2$, to the whole of $\R^2$.
Thus $\tilde{u}_0$ is piecewise constant, with a jump set as illustrated in Figure \ref{fig:JK}.
\begin{figure}
\begin{center}
\includegraphics[height=6cm]{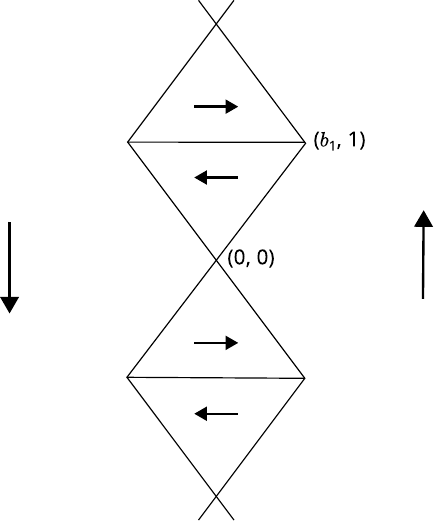}
\end{center}
\caption{A construction of Jin and Kohn \cite{Jin-Kohn:00}}
\label{fig:JK}
\end{figure}

Next, we construct $\tilde{u}_\epsilon \colon \R^2 \to \R^2$ as follows: we replace the jumps in
$\tilde{u}_0$ by the standard one-dimensional transitions with a width of
order $\epsilon$
(as explained, e.g., in a paper by Ignat and Monteil \cite[Proposition 4.1]{Ignat-Monteil:20}).
This requires some smoothing near the corners, which can be done with an insignificant gain of energy
and a small change of the divergence. (The details are tedious and
are omitted here.) We then compute
\[
\lim_{\epsilon \searrow 0} E_\epsilon(\tilde{u}_\epsilon; \R \times [s - 1, s + 1]) = \sqrt{b_1^2 + 1} \int_\lozenge \sqrt{W} \, d\Ha^1 + 2b_1 \int_{[b^-, b^+]} \sqrt{W} \, d\Ha^1
\]
for any $s \in \R \setminus \Z$.

But we want convergence to $u_0$, not $\tilde{u}_0$, as $\epsilon \searrow 0$.
Therefore, we now rescale by a parameter $\eta_\epsilon > 0$, which converges to $0$ when $\epsilon$
does, but at a slower rate; i.e., we assume that $\eta_\epsilon \to 0$ and $\epsilon/\eta_\epsilon \to 0$.
We set $u_\epsilon(x) = \tilde{u}_{\epsilon/\eta_\epsilon}(x/\eta_\epsilon)$. Then
\[
\lim_{\epsilon \searrow 0} E_\epsilon(u_\epsilon; B_1(0)) = \sqrt{b_1^2 + 1} \int_\lozenge \sqrt{W} \, d\Ha^1 + 2b_1 \int_{[b^-, b^+]} \sqrt{W} \, d\Ha^1.
\]
Therefore,
\[
\mathcal{E}(a^-, a^+) \le \frac{1}{2} \sqrt{b_1^2 + 1} \int_\lozenge \sqrt{W} \, d\Ha^1 + b_1 \int_{[b^-, b^+]} \sqrt{W} \, d\Ha^1.
\]
If the current $T$ happens to minimise $\M_T$ in $\ncrt_{2 \times 2}^0$, then we have equality
by Theorem \ref{thm:main}. Therefore, in this case, the above construction gives the optimal
energy asymptotically.
\end{example}

\begin{example} \label{exm:cross-tie}
Now suppose that $W(y) = 0$ for all $y \in S^1$. Choose $b_1, b_2 \in [0, 1]$ with $b_1^2 + b_2^2 = 1$, and define
$b^{(1)} = (b_1, -b_2)$, $b^{(2)} = (b_1, b_2)$, $b^{(3)} = (-b_1, b_2)$,
and $b^{(4)} = (-b_1, -b_2)$. Let $\theta = \arcsin b_1$ and consider the curve
\[
\gamma(t) = \begin{cases}
(\sin(\pi t), -\cos(\pi t)) & \text{if $0 \le t \le \frac{\theta}{\pi}$ or $1 - \frac{\theta}{\pi} \le t \le 1$}, \\
(b_1, \frac{(2t - 1)\pi}{\pi - 2\theta}b_2) & \text{if $\frac{\theta}{\pi} < t < 1 - \frac{\theta}{\pi}$}
\end{cases}
\]
(consisting of a circular arc from $a^-$ to $b^{(1)}$, a line segment from $b^{(1)}$ to $b^{(2)}$,
and another circular arc from $b^{(2)}$ to $a^+$). Then
\[
\psi(t) = \begin{cases}
\cot(\pi t) & \text{if $0 \le t \le \frac{\theta}{\pi}$ or $1 - \frac{\theta}{\pi} \le t \le 1$}, \\
0 & \text{if $\frac{\theta}{\pi} < t < 1 - \frac{\theta}{\pi}$}.
\end{cases}
\]

This function is not of bounded variation (not even bounded), and so, strictly speaking, the above
calculations do not apply. We ignore this problem for the sake of simplicity.
We will obtain an $\R^2$-valued current on $\R^2 \setminus \{a^-, a^+\}$ instead of $\R^2$, which
can, however, be approximated with elements of $\ncrt_{2 \times 2}^0$.

Define $V = [b^{(1)}, b^{(2)}] \cup [b^{(4)}, b^{(3)}]$ and $H = [b^{(4)}, b^{(1)}] \cup [b^{(3)}, b^{(2)}]$,
and also define $D = B_1(0) \setminus (\R \times [-b_2, b_2])$. Then we compute
\[
\begin{split}
\M_F(T) & = \frac{1}{4} \int_V \sqrt{W} \, d\Ha^1 + \frac{b_2}{4b_1} \int_H \sqrt{W} \, d\Ha^1 \\
& \quad + \frac{\pi}{4} \int_{[0, \theta] \cup [1 - \theta, 1]} \frac{1}{\sin^2(\pi t)} \int_{-\sin(\pi t)}^{\sin(\pi t)} \sqrt{W(s, -\cos(\pi t))} \, ds \, dt \\
& = \frac{1}{4} \int_V \sqrt{W} \, d\Ha^1 + \frac{b_2}{4b_1} \int_H \sqrt{W} \, d\Ha^1 + \frac{1}{4} \int_D \frac{\sqrt{W(y)}}{(1 - y_2^2)^{3/2}} \, dy. \\
\end{split}
\]
(This may be infinite, unless we impose additional conditions on $W$ at $a^\pm$.)

Compare this with the following transition profile, which is called a cross-tie wall
in the theory of micromagnetics \cite{DeSimone-Kohn-Mueller-Otto:03, Alouges-Riviere-Serfaty:02}. This is a two-scale construction again, and the coarser
scale is given by
\[
\tilde{u}_0(x) = \begin{cases}
b^{(1)} & \text{if $x_1 < 0$, $x_2 < 0$, and $-b_1 x_1 + b_2 x_2 < 0$}, \\
b^{(2)} & \text{if $x_1 > 0$, $x_2 < 0$, and $b_1 x_1 + b_2 x_2 < 0$}, \\
b^{(3)} & \text{if $x_1 > 0$, $x_2 > 0$, and $-b_1 x_1 + b_2 x_2 > 0$}, \\
b^{(4)} & \text{if $x_1 < 0$, $x_2 > 0$, and $b_1 x_1 + b_2 x_2 > 0$}, \\
\frac{x^\perp}{|x|} & \text{else},
\end{cases}
\]
when $-b_1 < x_2 \le b_1$. This is extended periodically in $x_2$, with period $2b_1$, so that $\tilde{u}_0$ is defined on all of $\R^2$.
The result is illustrated in Figure 2.
\begin{figure}
\begin{center}
\includegraphics[height=6cm]{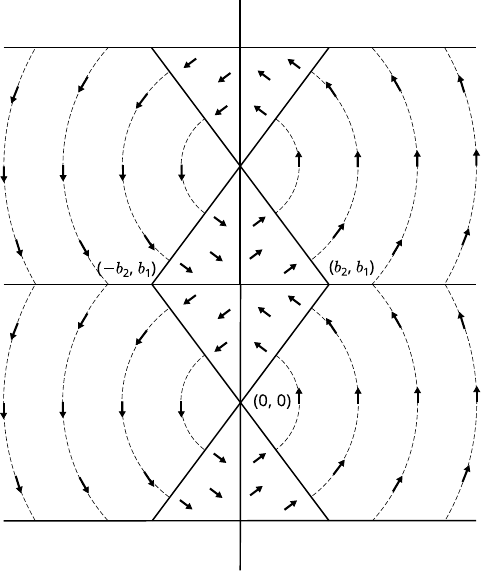}
\end{center}
\caption{The cross-tie wall profile}
\label{fig:cross-tir}
\end{figure}

The vector field $\tilde{u}_0$ has discontinuities along the lines $\{0\} \times \R$ and $\R \times \{(1 + 2k)b_1\}$ for every $k \in \Z$.
On the finer scale, these have to be replaced by smooth transitions again.
For $\epsilon > 0$, we choose the standard one-dimensional transitions, with a width of order $\epsilon$, along the line segments $\{0\} \times [-b_1, 0]$
(where we have a transition between $b^{(1)}$ and $b^{(2)}$) and $\{0\} \times [0, b_1]$ (with a transition between $b^{(3)}$ and $b^{(4)}$), as well as
$[-b_2, 0] \times \{b_1\}$ (with a transition between $b^{(4)}$ and $b^{(1)}$)
and $[0, b_2] \times \{b_1\}$ (with a transition between $b^{(3)}$ and $b^{(2)}$). This part of the construction is similar to Example \ref{exm:JK}.

Along $(-\infty, -b_2] \times \{b_1\}$ and $[b_2, \infty) \times \{b_1\}$, we still use a similar construction,
but since the jump in $\tilde{u}_0$ depends on the position,
the result will not truly be one-dimensional here. At the point $(x_1, b_1)$ with $|x_1| \ge b_2$, we have a jump between the points
\[
\frac{(-b_1, x_1)}{\sqrt{b_1^2 + x_1^2}} \quad \text{and} \quad \frac{(b_1, x_1)}{\sqrt{b_1^2 + x_1^2}}.
\]
We replace this with a smooth transition, again with width of order $\epsilon$, along the horizontal line segment between these two points in $\R^2$.
Once more this requires some smoothing at the corners, and in the end everything is extended periodically to $\R^2$.

We will have some divergence at the corners, and also near $(-\infty, -b_2] \times \{b_1\}$ and $[b_2, \infty) \times \{b_1\}$, with this construction.
We can, however, achieve that a condition similar to \eqref{eq:small-divergence} holds in any compact subset of $\R^2$. Once more, the details are omitted.

Finally, we rescale at a rate $\eta_\epsilon > 0$ as $\epsilon \searrow 0$,
where $\eta_\epsilon \to 0$, but sufficiently slowly that \eqref{eq:small-divergence} remains true.

Calculating the energy of the cross-tie wall per unit wall length, we find that
\[
\mathcal{E}(a^+, a^-) \le \frac{1}{2} \int_V \sqrt{W} \, d\Ha^1 + \frac{b_2}{2b_1} \int_H \sqrt{W} \, d\Ha^1 + \frac{1}{2} \int_D \frac{\sqrt{W(y)}}{(1 - y_2^2)^{3/2}} \, dy.
\]
If suitable approximations of $T$ give rise to a minimising sequence of $\M_F$ in $\ncrt_{2 \times 2}^0$, then Theorem \ref{thm:main}
implies that we have equality, and the cross-tie wall is energetically optimal.
\end{example}

\begin{acknowledgement}
RI is partially supported by the ANR projects ANR-21-CE40-0004 and ANR-22-CE40-0006-01. RM is partially supported by the Engineering and Physical Sciences Research Council [grant number EP/X017206/1].
\end{acknowledgement}

\def\cprime{$'$}
\providecommand{\bysame}{\leavevmode\hbox to3em{\hrulefill}\thinspace}
\providecommand{\MR}{\relax\ifhmode\unskip\space\fi MR }
\providecommand{\MRhref}[2]{%
  \href{http://www.ams.org/mathscinet-getitem?mr=#1}{#2}
}
\providecommand{\href}[2]{#2}

\end{document}